\documentclass[oneside, 12pt]{amsart}
\usepackage{amscd}
\usepackage{amssymb}
\usepackage{amsfonts}
\usepackage{remark}
\usepackage{enumerate}
\usepackage[all]{xy}
\usepackage{url}

\usepackage{color}
\usepackage[normalem]{ulem}

\newtheorem{thm}{Theorem}[section]
\newtheorem{theorem}[thm]{Theorem}
\newtheorem{facts}[thm]{Facts}

\newtheorem{lemma}[thm]{Lemma}
\newtheorem{corollary}[thm]{Corollary}
\newtheorem{proposition}[thm]{Proposition}
\newremark{definition}[thm]{Definition}
\newremark{notation}[thm]{Notation}
\newremark{ques}[thm]{Question}
\newremark{remark}[thm]{Remark}
\newremark{example}[thm]{Example}
\newremark{emp}[thm]{\kern -4pt}

\numberwithin{equation}{thm}

\newcommand\Supp{\operatorname{Supp}}
\newcommand{\Proj}{\operatorname{Proj}}
\newcommand{\Spec}{\operatorname{Spec}}
\newcommand{\Pic}{\mbox{\rm Pic}\kern 1pt}
\newcommand{\Br}{\mbox{\rm Br}\kern 1pt}

\newcommand{\cO}{{\mathcal O}}

\newcommand{\F}{{\mathcal F}}

\newcommand{\G}{{\mathcal G}}
 
\newcommand{\J}{{\mathcal J}}    
\newcommand{\m}{{\mathfrak m}}

\newcommand{\p}{{\mathfrak p}}
\newcommand{\q}{{\mathfrak q}}
\newcommand{\Ass}{\operatorname{Ass}}

\newcommand{\red}{\mathrm{red}}
\newcommand{\CH}{\mathcal A}

\newcommand{\dv}{\operatorname{div}}
\newcommand{\Coker}{\operatorname{Coker}\kern 1pt}

\newcommand{\ord}{\operatorname{ord}}

\newcommand{\card}{\operatorname{Card}}
 
\newcommand{\Frac}{\operatorname{Frac}}

\newcommand{\codim}{{\operatorname{codim}}}
\newcommand{\A}{\mathbb A}

\newcommand{\pictorsion}{\mbox{pictorsion}}

\newcommand*\oline[1]{%
  \vbox{%
    \hrule height 0.5pt%                  % Line above with certain width
    \kern0.25ex%                          % Distance between line and content
    \hbox{%
      \kern-0.1em%                        % Distance between content and left side of box, negative values for lines shorter than content
      \ifmmode#1\else\ensuremath{#1}\fi%  % The content, typeset in dependence of mode
      \kern-0.1em%                        % Distance between content and left side of box, negative values for lines shorter than content
    }% end of hbox
  }% end of vbox
}

\begin{document}
\date{\today}
\thanks{D.L. was supported by NSF Grant 0902161.}

\begin{abstract}
Let $X/S $ be a quasi-projective morphism over an affine base. We develop in this article 
a technique for proving the existence of closed subschemes $H/S$ of $X/S$ with various favorable properties.
We offer several applications of this technique, including the existence  of finite quasi-sections in certain projective morphisms,
and the existence of
hypersurfaces in $X/S$ containing a given closed subscheme $C$, and intersecting properly a closed set $F$. 

Assume now that the base $S$ is the spectrum of 
a ring $R$ such that for any finite 
morphism $Z \to S$, $\Pic(Z)$ is a torsion group.
This condition is satisfied if $R $ is the ring of integers of a number field, or 
the ring of functions of a smooth affine curve over a finite field. 
We prove in this context a moving lemma pertaining to horizontal $1$-cycles on a regular scheme
$X$ quasi-projective and flat over  $S $. We also show the existence
of a finite surjective $S$-morphism to $\mathbb P_S^d$
for any scheme $X$ projective over $S$ when $X/S$ has all its fibers of a fixed dimension $d$.

\vspace*{.5cm}
%{\bf Keywords}: 
\noindent 
\begin{tiny}KEYWORDS. \end{tiny} Avoidance lemma, Bertini-type theorem, Hypersurface,   Moving lemma, Multisection,  $1$-cycle,  Pictorsion, Quasi-section, Rational equivalence, Zero locus of a section, Noether normalization.

\vspace*{.3cm}
\noindent 
\begin{tiny}MATHEMATICS SUBJECT CLASSIFICATION:   14A15, 14C25, 14D06, 14D10,  14G40.   \end{tiny}
\end{abstract}

\title
{Hypersurfaces in projective schemes and a moving lemma}
\author{Ofer Gabber}
\author{Qing Liu}
\author{Dino Lorenzini}
\address{IH\'ES, 35 route de Chartres, 
91440 Bures-sur-Yvette, France}
\email{gabber@ihes.fr}
\address{Universit\'e de Bordeaux 1, Institut de Math\'ematiques de 
Bordeaux, CNRS UMR 5251, 33405 Talence, France} 
\email{Qing.Liu@math.u-bordeaux1.fr}
\address{Department of mathematics, University of Georgia, 
Athens, GA 30602, USA} 
\email{lorenzin@uga.edu}
\maketitle

%\begin{section}{Introduction } 
Let $S=\Spec R$ be an affine 
scheme, and let $X/S$ be a quasi-projective scheme. 
The core of this article is a method, 
summarized in \ref{sum} below, 
for proving the existence of closed subschemes  of $X$ with 
various favorable properties. As the technical details  can be 
somewhat complicated, we start this introduction by discussing 
the applications of the method that the reader will find in this article.

Recall (\ref{zerolocusinvertible}) that a global section $f$ of an invertible sheaf ${\mathcal L}$ on any scheme  $X$
 defines a closed subset $H_f$ of $X$,  
 consisting of all points $x \in X$ where the stalk $f_x$ does not generate ${\mathcal L}_x$.
 Since $\cO_X f\subseteq {\mathcal L}$, 
 the ideal sheaf ${\mathcal I}:=
\cO_X f\otimes {\mathcal L}^{-1}$ endows $H_f$ with the structure of closed subscheme of $X$.
Let $X \to S$ be any morphism.
We  call the closed subscheme $H_f$ of $X$ a \emph{hypersurface} (relative to $X \to S$) when no
irreducible component of positive dimension of $X_s$ is contained in $H_f$, for all $s\in S$. 
If, moreover, the ideal sheaf ${\mathcal I}$ is invertible, 
we say that the hypersurface $H_f$ is \emph{locally principal}.
We remark that when a fiber $X_s$ contains isolated points, 
it is possible for $H_f$ (resp.\ $(H_f)_s)$ to have codimension $0$ in $X$ (resp.\ in $X_s$), instead of the expected codimension $1$.

\bigskip \noindent
{\bf A.  An Avoidance Lemma for Families.} It is classical that if $X/k$ is a quasi-projective 
scheme over a field, $C \subsetneq X$ is a closed subset of positive codimension, 
and $\xi_1,\dots, \xi_r$ are points of $X$ not contained in $C$, then there exists a 
hypersurface $H$ in $X$ such that $C \subseteq H$ and $\xi_1, \dots, \xi_r \notin H$.
Such a statement is commonly referred to as an Avoidance Lemma (see, e.g., \ref{avoid}).
Our next theorem establishes an Avoidance Lemma for Families. 
As usual,  when $X$ is noetherian, $\Ass(X)$ denotes the finite
{\em set of associated points} of $X$. 
\medskip

\noindent {\bf Theorem \ref{bertini-type-0}.} {\em
Let $S$ be an affine scheme, and let $X\to S$ be a 
quasi-projective and finitely presented morphism.  Let $\cO_X(1)$ be a very ample sheaf relative to $X \to S$. Let 
\begin{enumerate}[{\rm (i)}]
\item $C$ be a closed subscheme of $X$, 
finitely presented over $S$;
\item $F_1, \dots, F_m$ be 
subschemes\footnote{Each $F_i$ is a closed subscheme of an open subscheme of $X$ (\cite{EGA}, I.4.1.3).} of $X$ of finite presentation over $S$;
\item $A$ be a finite subset of $X$ such that $A\cap C=\emptyset$. 
\end{enumerate} 
Assume that for all $s \in S$, $C$ does not contain any
irreducible component of positive dimension of $(F_i)_s$ and of $X_s$. 
Then there exists $n_0>0$ such that for all $n\ge n_0$, there exists 
a global section $f$ of $\cO_X(n)$ such that: 
\begin{enumerate}[\rm (1)]
\item the closed subscheme $H_f$ of $X$ is a hypersurface 
that contains $C$ as a closed subscheme;
\item for all $s \in S$ and for all $i\le m$, $H_f$ does not contain 
any irreducible component of positive dimension of $(F_i)_s$; and
\item  $H_f\cap A=\emptyset$. 
\end{enumerate}
Assume in addition that $S$ is noetherian, and 
that $C\cap\Ass(X)=\emptyset$. Then there exists such a hypersurface 
$H_f$ which is locally principal. 
}

\medskip 
When $H_f$ is locally principal, 
$H_f$ is the support of an effective ample Cartier divisor on $X$. 
This divisor is `horizontal' in the sense that it does not contain in its support any irreducible component of fibers of $X \to S$ of positive dimension. In some instances, such as in \ref{bertini-cor1} and \ref{generic-S1},  we can show that $H_f$ is  a relative effective Cartier divisor, 
i.e.,  that $H_f \to S$ is flat. Corollary \ref{bertini-cor1} also includes a Bertini-type statement for $X \to S$ with Cohen-Macaulay fibers.
We use Theorem \ref{bertini-type-0} to establish in \ref{quasisections} the existence of finite quasi-sections in 
certain projective morphisms $X/S$,
as we now discuss.

\bigskip 
\noindent {\bf B. Existence of finite quasi-sections.}
 Let $X\to S$ be a surjective morphism. Following EGA \cite{EGA}, IV,
\S 14, p.\ 200, we define: 

\begin{definition} \label{def.finite-qs}  We call a closed subscheme 
$C$ of $X$ a \emph{finite quasi-section} when $C \to S$ is finite 
and surjective.  Some authors call  \emph{multisection} a finite quasi-section $C \to S$ which is also flat, with $C$ irreducible
(see e.g., \cite{HT}, p.\ {\rm 12} and {\rm 4.7}).
\end{definition}

When $S$ is integral noetherian of dimension $1$ 
and $X\to S$ is proper and surjective, 
the existence of a finite quasi-section $C$ is well-known and easy 
to establish. It suffices to take $C$ to be the Zariski closure 
of a closed point of the generic fiber of $X\to S$. When $\dim S>1$, 
the process of taking the closure of any closed point of the generic 
fiber does not always produce a closed subset {\it finite} over $S$  
(see \ref{easy}).

\medskip
\noindent {\bf Theorem \ref{quasisections}.} 
{\em Let $S$ be an affine scheme  and let $X\to S$ be a projective, finitely 
presented 
morphism. Suppose that all fibers of $X\to S$ 
are of the same dimension $d\ge 0$. Let $C$ be a finitely presented closed subscheme of $X$, with $C \to S$ finite but not necessarily surjective.
Then there exists a finite quasi-section $T$ of finite presentation which contains $C$. 
Moreover:
\begin{enumerate}[\rm (1)]
\item Assume that $S$ is noetherian. If $C$ and $X$ are both irreducible, then there exists 
such a quasi-section with $T$ irreducible.
\item   
If $X\to S$ is flat with Cohen-Macaulay fibers (e.g., if
$S$ is regular and $X$ is Cohen-Macaulay),  
then there exists 
such a quasi-section with $T\to S$ flat. 
\item  
If $X \to S$ is flat and a local complete intersection morphism\footnote{Since the morphism $X\to S$ is flat, 
it is a local complete intersection morphism if and  only if every fiber is a local complete intersection morphism (see, e.g., \cite{Liubook}, 6.3.23).},  
then there exists such a quasi-section with $T\to S$ flat and a local complete intersection morphism.
\item 
Assume that $S$ is noetherian. Suppose that $\pi:X \to S$ has fibers pure 
of the same dimension, and that $C \to S$ is unramified. 
Let $Z$ be a finite subset of $S$ (such as the set of generic points of $\pi(C)$), and suppose that there
exists  an open subset $U$ of $S$ containing $Z$
 such that $X \times_S U \to U$ is smooth. Then there exists 
such a quasi-section $T$ of $X \to S$ and an open set $V \subseteq U$ containing $Z$
such that $T \times_S V \to V$ is \'etale.
\end{enumerate}
}
\medskip

As an application of Theorem~\ref{quasisections}, we obtain 
a strengthening in the affine case of the classical splitting 
lemma for vector bundles. 

\medskip
\noindent {\bf Proposition \ref{splitting}.} 
{\em  Let $A$ be a 
commutative ring. Let $M$ be a 
projective $A$-module of finite presentation with constant rank $r> 1$.
Then there exists an $A$-algebra $B$, finite and faithfully flat over $A$, with $B$ a local complete intersection over $A$,
such that $M\otimes_A B $ 
is isomorphic to a direct sum of projective $B$-modules of rank $1$.}
\smallskip

Another application of Theorem
 \ref{quasisections}, to the problem of extending a given 
 family of stable curves ${D} \to Z$ after a finite surjective base change, is found in \ref{extension.stable.curve}.
 It is natural to wonder whether Theorems \ref{bertini-type-0} and \ref{quasisections} hold for more general bases $S$ which are not affine,
such as a noetherian base $S$ 
having an ample invertible sheaf. 
 It is also natural to wonder if the existence of finite quasi-sections in Theorem \ref{quasisections} holds for proper morphisms. 

\bigskip \noindent
{\bf C. Existence of Integral Points.}
Let $R$ be a Dedekind domain\footnote{A Dedekind domain in this article has dimension $1$, and a Dedekind scheme is the spectrum of a Dedekind domain.} and let $S= \Spec R$.
When $X\to S$ is quasi-projective, an 
integral 
finite quasi-section is also called an \emph{integral point}
in \cite{MB1}, 1.4. The existence of a finite quasi-section
in the quasi-projective case over $S=\Spec {\mathbb Z}$ when the generic fiber is geometrically irreducible
is Rumely's famous Local-Global Principle \cite{Rum}. 
This existence result was extended in \cite{MB1}, 1.6, as follows. As in \cite{MB1}, {\rm  1.5}, we make the following definition.

\begin{definition} \label{ConditionT} {
We say that a Dedekind scheme $S$ satisfies \emph{Condition} {\rm (T)} if: 
\begin{enumerate}[\rm (a)]
\item For any finite extension 
$L$ of the field of fractions $K$ of $R$, the normalization $S'$ of $S$ in $\Spec L$ has 
torsion Picard group $\Pic(S')$, and 
\item  The residue fields 
at all closed points of $S$ are  algebraic extensions of finite fields.
\end{enumerate} 
}
\end{definition}

For example, $S$ satisfies Condition 
(T) if $S$ is an affine integral smooth curve over a finite field, 
or if $S$ is the spectrum of the ring of $P$-integers 
in a number field $K$, where $P$ is a finite set of finite places of $K$. 
\medskip

Our next theorem is only a mild sharpening of the Local-Global Principle in \cite{MB1}, 1.7:
We show in \ref{RumelyLG} that the hypothesis in \cite{MB1}, 1.7, that the base scheme $S$ is excellent, can be removed.

\medskip
\noindent {\bf Theorem \ref{RumelyLG}.} 
{\em  Let $S$ be a Dedekind scheme satisfying Condition {\rm (T)}. 
Let $X \to S$ be a separated surjective morphism of finite type. Assume that $X$ is irreducible and that the generic fiber of $X \to S$
is geometrically irreducible. Then $X \to S$ has a finite quasi-section.}

\medskip
Condition (T)(a) is necessary in the  Local-Global Principle, but
it is not sufficient, as shown by an example of Raynaud over $S=\Spec \overline{{\mathbb Q}}[x]_{(x)}$ (\cite{MB0}, 3.2, and \cite{vDM}, 5.5).
The following weaker condition is needed for our next two theorems.

\begin{definition} \label{ConditionT*}
{Let $R$ be any commutative ring and let $S= \Spec R$. We say that $R$ or $S$ is \emph{pictorsion} if $\Pic(Z)$ is a torsion group
for any finite 
morphism $Z \to S$.}
\end{definition}

Any semi-local ring $R$ is $\pictorsion$.
A Dedekind domain satisfying Condition (T) is $\pictorsion$ 
(\cite{MB1}, 2.3, see also \ref{lem.torsiondegreed} (2)). 
Rings which satisfy the primitive criterion (see \ref{localglobal}) are 
$\pictorsion$ and only have infinite residue fields. 
\bigskip 

\noindent{\bf D.  A Moving Lemma.} Let $S$ be a Dedekind scheme and let $X$ be a noetherian scheme over $S$. 
An integral closed subscheme $C$ of $X$ finite and surjective over $S$
is called an \emph{irreducible horizontal $1$-cycle} on $X$. A \emph{horizontal $1$-cycle} on $X$
is an element of the free abelian group generated by the irreducible horizontal $1$-cycles.
Our next application of the method developed in this article is a Moving Lemma 
for horizontal  $1$-cycles.

\medskip
\noindent {\bf Theorem \ref{mv-1-cycle-local}.} {\em  Let $R$ be a Dedekind domain, and let $S:=\Spec R$. 
Let $X \to S$ be a flat and quasi-projective morphism, with $X$
integral. 
Let $C$ be a horizontal $1$-cycle on $X$. Let $F$ be a 
closed subset of $X$. 
Assume that for all $s \in S$, $F\cap X_s$ and $\Supp(C) \cap X_s$ have 
positive codimension\footnote{The definition of codimension in \cite{EGA}, Chap.\ 0, 14.2.1,
implies that the codimension of the empty set in $X_s$ is $+\infty$, 
which we consider to be positive.} 
in $X_s$. Assume in addition that
either
\begin{enumerate}[\rm (a)]
\item $R$ is pictorsion and the support of $C$ is contained in the regular locus of $X$, or
\item $R$ satisfies Condition \emph{(T)}. 
\end{enumerate}

Then some positive multiple $mC$ of 
$C$ is rationally equivalent
to a horizontal $1$-cycle $C'$ on $X$ whose support does not meet $F$.
Under the assumption {\rm (a)}, if furthermore $R$ is  
semi-local, then we can take $m=1$.

Moreover, if $Y \to S$ is any separated morphism of finite type and $h: X \to Y$
is any  $S$-morphism, then $h_*(mC)$ is rationally equivalent to $h_*(C')$ on $Y$.}
\medskip
 
Example \ref{ex.extrahyp} shows that the Condition (T)(a) is necessary 
for Theorem \ref{mv-1-cycle-local} to hold.
A different proof of  Theorem \ref{mv-1-cycle-local} when $S$ is semi-local, $X$ is regular, and $X\to S$ is quasi-projective, 
is given in \cite{GLL1}, 2.3, where the result is then used to prove a formula for the index of an algebraic variety
over a Henselian field (\cite{GLL1}, 8.2).

It follows from \cite{GLL1}, 6.5, that for each $s \in S$, a multiple $m_s C_s$ of the $0$-cycle $C_s$ is rationally equivalent 
on $X_s$ to a $0$-cycle whose support is disjoint from $F_s$. Theorem  6.5 in \cite{GLL1} expresses such an integer $m_s$ 
in terms of Hilbert-Samuel multiplicities. The $1$-cycle $C$ in $X$ can be thought of   as a family of $0$-cycles, and 
Theorem \ref{mv-1-cycle-local} may be considered as a Moving Lemma for $0$-cycles in families.

Even for schemes of finite type over a finite field, Theorem \ref{mv-1-cycle-local}
is not a consequence of the classical Chow's Moving Lemma. Indeed, let $X$ be a 
smooth
quasi-projective variety over 
a field $k$.
The classical Chow's Moving Lemma \cite{Rob}
immediately implies the following statement:

\begin{emp} \label{Chow} 
{\it Let $Z$ be a $1$-cycle on $X$. 
Let $F$ be a closed subset of $X$ of codimension at least $2$ in $X$. 
Then there exists a $1$-cycle $Z'$ on $X$, rationally equivalent to $Z$, 
and such that $\Supp(Z') \cap F 
= \emptyset$.
}
\end{emp}

 Consider a morphism $X \to S$ as in Theorem \ref{mv-1-cycle-local},
and assume in addition that $S$ is a smooth affine curve over a finite field $k$.
Let $F$ be a closed subset as in \ref{mv-1-cycle-local}. Such a subset may be of 
codimension $1$ in $X$. Thus, Theorem \ref{mv-1-cycle-local} is not a consequence
of Chow's Moving Lemma for $1$-cycles just recalled, since \ref{Chow} can only be applied to 
$X\to S$ when  $F$ is a closed subset of codimension at least $2$ in $X$.

\medskip 
\noindent 
{\bf E. Existence of finite morphisms to ${\mathbb P}^d_S$.} 
Let $k$ be a field. A strong form of the Normalization Theorem of E. Noether
that applies to graded rings  (see, e.g., \cite{Eis}, 13.3) implies that every projective variety $X/k$ of dimension $d$
admits a finite $k$-morphism  $X \to {\mathbb P}_k^d$. Our next theorem   
guarantees the existence of a finite $S$-morphism $X \to {\mathbb P}_S^d$ 
when $X\to S$ is projective with $R$ $\pictorsion$, and $d:= \max\{\dim X_s, s\in S\}$. 

\medskip
\noindent {\bf Theorem \ref{theorem.finiteP^n}.} {\em
 Let $R$ be a $\pictorsion$ ring, and let $S:=\Spec R$.
Let $X \to S$ be a projective morphism, 
and set    
$d:= \max\{\dim X_s, s\in S\}$. 
Then  there exists a finite $S$-morphism 
$ X \to {\mathbb P}^d_S$.
If we assume in addition that  $\dim X_s= d$  for all $s\in S$,  then there exists a finite surjective $S$-morphism $ X \to {\mathbb P}^d_S$.
}
\medskip

The  above theorem generalizes  to schemes $X/S$ of any dimension
the results of \cite{Gre2}, Theorem 2,  and  \cite{CPT}, Theorem 1.2,  which apply to morphisms
of relative dimension $1$. After this article was written, we became aware of the preprint
\cite{CMPT}, where the general case is also discussed. 
We also prove the  converse of this theorem: 

\smallskip
\noindent {\bf Proposition (see \ref{conversepictorsion}).} {\em 
Let $R$ be any commutative 
ring and let $S:=\Spec R$. 
Suppose that for any $d\ge 0$, and for any projective morphism $X
\to S$ such that $\dim X_s=d$  for all $s\in S$, 
there exists a finite surjective $S$-morphism $ X \to {\mathbb P}^d_S$.
Then $R$ is pictorsion.
}

\bigskip \noindent 
{\bf F. Method of proof.} 
Now that the main applications of our method for proving the existence of hypersurfaces $H_f$ in projective schemes $X/S$ with certain desired properties 
have been discussed, let us summarize the method. 

\begin{emp} \label{sum} 
Let $X\to S$ be a projective morphism with $S=\Spec R$ affine and 
noetherian. Let $\cO_X(1)$ be a very ample sheaf relative to $X\to S$. 
Let $C \subseteq X$ be a closed subscheme
defined by an ideal ${\mathcal I}$, and set ${\mathcal I}(n) := {\mathcal I} \otimes {\mathcal O}_X(n)$.
Our goal is to show the existence, for all $n$ large enough, 
of a global section $f $ of ${\mathcal I}(n)$ such that the 
associated subscheme $H_f$ has the desired properties. 

To do so, we fix a system of generators $f_1,\dots, f_N$ 
of $H^0(X,{\mathcal I}(n)) $, and we consider for each $s\in S$
a subset $\Sigma(s) \subset {\mathbb A}^N(k(s))$ consisting of all the vectors $(\alpha_1,\dots,\alpha_N)$ such that $\sum_i \alpha_i f_{i|{X_s}}$
does not have the desired properties. We show then that all these
subsets 
$\Sigma(s)$, $s\in S$, are the rational points of 
a single pro-constructible subset $T$ of ${\mathbb A}^N_S$ 
(which depends on $n$). To find a 
desired global section $f:=\sum_i a_i f_i$ with $a_i \in R$ which 
avoids the subset $T$ of `bad' sections, 
we show that for some $n$ large enough the pro-constructible subset $T$ satisfies the hypotheses of the  following theorem.
The section $\sigma$ whose existence follows from \ref{globalize} provides the desired vector $(a_1,\dots, a_N) \in R^N$.

\medskip
\noindent 
{\bf Theorem (see \ref{globalize}).} {\it
Let $S$ 
be a noetherian affine scheme. Let $T:=T_1\cup \ldots \cup T_m$ 
be a finite union of pro-constructible subsets of $\mathbb A^N_S$.  
Suppose that:
\begin{enumerate}[\rm (1)]
\item For all $i\le m$, $\dim T_i<N$,  and  
 $(T_i)_s$ is constructible in $\mathbb A^N_{k(s)}$ for all $s\in S$.
\item For all $s\in S$, there exists a $k(s)$-rational point in 
$\A^N_{k(s)}$ which does not belong to $T_s$.
\end{enumerate}
\smallskip
Then there exists a section $\sigma$ of $\pi: \A^N_S\to S$
such that $\sigma(S) \cap T = \emptyset$.
}

\smallskip
To explain the phrasing of (1) in the above theorem,
note that  the union $T_1\cup \ldots \cup T_m=:T$ 
is pro-constructible since each $T_i$ is. 
However, it may happen that $\dim T > \max_i(\dim T_i)$. This can be seen already on  the spectrum $T$ of a discrete valuation ring, 
which is the union of two (constructible) points, each of dimension $0$.
The proof of 
Theorem~\ref{globalize} is given in 
section \ref{Constructible subsets}. 
The construction alluded to in \ref{sum}  of the
pro-constructible subset $T$ whose rational points are in bijection with 
$\Sigma(s)$ is 
done in Proposition~\ref{constructible-conditions}.
\medskip 

We present our next theorem  as a final illustration of the strength of the method.
This theorem, stated in a slightly stronger form in section \ref{Main application},   
is the key to the proof of Theorem \ref{mv-1-cycle-local} (a), 
as it allows for a reduction
to the case of relative dimension $1$. 
Note that in this theorem, $S$ is not assumed to be $\pictorsion$.
\end{emp}

%\medskip
\noindent {\bf Theorem (see \ref{pro.reductiondimension2}).} {\em 
Let $S$ be an affine noetherian scheme
of finite dimension, and let $X\to S$ be a quasi-projective
morphism. Let $C$ be a closed irreducible subscheme of $X$, 
of codimension $d>\dim S$ in $X$.  Assume that $C\to S$ is finite and surjective, and  that
$C \to X$ is a regular immersion. Let $F$ be a closed subset
of $X$. Fix a very ample sheaf $\cO_X(1)$ relative to $X\to S$. 
Then there exists $n_0>0$ such that for all $n \geq n_0$, 
there exists a global section $f$ of $\cO_X(n)$ such that: 
\begin{enumerate}[\rm (1)]
\item $C$ is a closed subscheme of codimension $d-1$ in $H_f$, 
and $C \to H_f$ is a regular immersion; 
\item For all $s\in S$, $H_f$ does not contain any irreducible
component  of positive dimension of $F_s$. 
\end{enumerate}
}

\smallskip
The proof of 
Theorem \ref{pro.reductiondimension2} is quite subtle and 
spans sections \ref{Main application} and \ref{compute-coh}. 
In  Theorem \ref{mv-1-cycle-local} (a), we start with the hypothesis that $C$ is contained in the regular locus of $X$.
It is not possible in general to expect that a hypersurface $H_f$ containing $C$ can be chosen so that 
$C$ is again contained in the regular locus of $H_f$. 
Thus, when  no regularity conditions can be expected on the total space, we impose regularity conditions by assuming that $C$ is 
regularly immersed in $X$. 
 Great care is then needed in the proof of \ref{pro.reductiondimension2} to insure that a hypersurface $H_f$ can be found with the property 
that $C$ is  regularly immersed in $H_f$. 

Section~\ref{Main application} contains most of the proof of 
Theorem \ref{pro.reductiondimension2}. Several lemmas needed in the
proof of Theorem \ref{pro.reductiondimension2} are discussed 
separately in section~\ref{compute-coh}. 
Sections \ref{hyperf}, \ref{finite-qs}, \ref{mv-1c} and \ref{finite-pd} 
contain the proofs of the applications of our method.
\medskip 

It is our pleasure to thank Max Lieblich and Damiano Testa for the
second example in \ref{ex.avoidance}, Angelo Vistoli for helpful 
clarifications regarding \ref{extension.stable.curve} and 
\ref{smooth_cover_mg}, and Robert Varley for Example \ref{Varley}. 
We also warmly thank the referees, for a meticulous reading of the 
article, for  several corrections and strengthenings,  and many other 
useful suggestions which led to improvements in the exposition.

\tableofcontents

\begin{section}{Zero locus of sections of a quasi-coherent sheaf} 

We start this section by reviewing basic facts on constructible 
subsets, a concept introduced by Chevalley in \cite{Che}. We follow the 
exposition in \cite{EGA}. We introduce the zero-locus ${\mathcal Z}(\mathcal F,f)$
of a global section $f$ of a finitely presented $\cO_X$-module $\mathcal F$ on a scheme $X$,
and show in \ref{zero-locus-f} that this subset is locally constructible in $X$. Given a finitely presented
morphism $\pi: X \to Y$, we further define a subset $T_{\mathcal F, f, \pi}$ in $Y$, and show in \ref{constructible-univ} that
it is locally constructible in $Y$. The main result in this section  is Proposition~\ref{constructible-conditions},
which is a key ingredient in the proofs of  Theorems \ref{exist-hyp} and \ref{pro.reductiondimension2}. 
 
Let $X$ be a topological space. A subset $T$ of $X$ is 
\emph{constructible}\footnote{See \cite{EGA}, Chapter 0, 9.1.2. Beware that in the second edition \cite{EGA1}, Chapter 0, 2.3,  a {\it globally constructible} subset now refers
to what is called a {\it constructible} subset in \cite{EGA}.} if it is a finite union of subsets of the form
$U\cap (X\setminus V)$, where $U $ and $ V$ are 
open and retro-compact\footnote{A topological space $X$ is {\it quasi-compact} if every open covering of $X$ has a finite refinement.
A continuous map $f : X \to Y $ is {\it quasi-compact} if the inverse
image $f^{-1}(V )$ of every quasi-compact open $V$ of $Y$ is quasi-compact.
A subset $Z$ of $X$ is {\it retro-compact} if the inclusion map $Z \to X$ is
quasi-compact.} in $X$. A subset $T$ of $X$ is 
\emph{locally constructible} if for any point $t\in T$, there exists
an open neighborhood $V$ of $t$ in $X$ such that $T\cap V$ is
constructible in $V$ (\cite{EGA}, Chap.\ 0, 9.1.3 and 9.1.11).
When $X$ is a quasi-compact and quasi-separated scheme\footnote{$X$ is {\it quasi-separated} if and only if the intersection of any two quasi-compact open
subsets of $X$ is quasi-compact (\cite{EGA}, IV.1.2.7).} (e.g., if 
$X$ is noetherian, or affine), then any quasi-compact open subset 
is retro-compact and 
any locally constructible subset is constructible (\cite{EGA}, IV.1.8.1). 

When $T$ is a subset of a topological space $X$, we endow $T$ with the induced topology, 
and define the \emph{dimension of $T$} to be the Krull dimension 
of the topological space $T$. As usual, $\dim T<0$ if and only 
if $T = \emptyset$. 
Let $\pi: X\to S$ be a morphism and let $T\subseteq X$ be any
subset. For any $s \in S$, we will denote by $T_s$ the subset 
$\pi^{-1}(s) \cap T$. 

\begin{emp} \label{pro-constr} 
Recall that a \emph{pro-constructible} subset 
in a \emph{noetherian} scheme $X$ is a (possibly infinite) intersection of 
constructible subsets of $X$ (\cite{EGA}, IV.1.9.4). 
Clearly, the constructible subsets of $X$ are pro-constructible in $X$, 
and so are the finite subsets of $X$ (\cite{EGA}, IV.1.9.6), and  
the constructible subsets of any fiber of a morphism of schemes $X \to Y$ 
(\cite{EGA}, IV.1.9.5 (vi)). 
The complement in $X$ of a pro-constructible subset is called 
an \emph{ind-constructible} of $X$.  Equivalently, an
ind-constructible subset  of $X$ is any union of constructible subsets of $X$. 
\end{emp}

We are very much indebted to a referee for pointing out that our original hypothesis in Theorem \ref{globalize} 
that $T$ be the union of a constructible subset and finitely many closed strict 
subsets of fibers $\mathbb A^N_{k(s)}$ could be generalized to 
the hypothesis that $T$ be pro-constructible.
We also thank this referee for suggestions which greatly improved
the exposition of the proof of  our original 
Proposition \ref{constructible-conditions}. 

\begin{lemma} \label{lem.construct}
\begin{enumerate}[\rm (a)]
\item Let $X/k$ be a scheme of finite type over a field $k$, and let $T\subseteq X$ be a constructible subset,
with closure $\overline{T} $ in $X$. Then $\dim T = \dim \overline{T}$.
\item Let $X/k$ be as in (a). Let $k'/k$ be a finite extension, and denote by $T_{k'}$ the preimage of $T$ under the map 
$X \times_k k' \to X$. Then $\dim T_{k'} = \dim T$.
\item Let $Y$ be any noetherian scheme and $\pi: X \to Y$ be a morphism of finite type. Let $T \subseteq X$ be constructible.
Assume that for each $y\in Y$, $\dim T_y\leq d$.
Then $\dim T \leq \dim Y + d$. 
\item Suppose $X$ is a noetherian scheme. Let $T$ be a
pro-constructible subset of $X$. 
Then $T$ has finitely many irreducible components 
and each of them has a generic point. 
\end{enumerate}
\end{lemma}

\proof (a)-(b) Let $\Gamma$ be an irreducible component of $\overline{T}$ of
dimension $\dim \overline{T}$. As $T$ is dense in 
$\overline{T}$, $T\cap \Gamma$ is dense in $\Gamma$. As $T\cap \Gamma$
is
constructible and dense in $\Gamma$, it contains a dense open subset
$U$ of $\Gamma$. Therefore, because $\Gamma$ is integral of finite
type over $k$,  
$\dim \Gamma=\dim U \le \dim T \le \dim \overline{T}=\dim \Gamma$ 
and $\dim T=\dim \overline{T}$. This proves (a). 

We also have 
$$\dim\Gamma=\dim\Gamma_{k'}=\dim U_{k'}\le \dim T_{k'}\le \dim
\overline{T}_{k'}=\dim \overline{T}=\dim\Gamma.$$ 
This proves (b). 

(c) Let $\{\Gamma_i\}_i$ be the irreducible components of $T$. They are
closed in $T$, thus constructible in $X$. As $\dim T=\max_i \{ \dim \Gamma_i\}$ and
the fibers of $\Gamma_i\to Y$ all have dimension bounded by $d$, it is enough to
prove the statement when $T$ itself is irreducible. Replacing $X$ with the
Zariski closure of $T$ in $X$ with reduced scheme structure, we can suppose $X$ is integral
and $T$ is dense in $X$. Let $\xi$ be the generic point of $X$ and let 
$\eta=\pi(\xi)$. As $T$ is constructible and dense in $X$, it contains
a dense open subset $U$ of $X$. Then $U_\eta$ is dense in $X_\eta$. Hence
$\dim X_\eta=\dim U_\eta\le \dim T_\eta\le d$. Therefore
$$\dim T\le \dim X\le \dim \overline{\pi(X)}+d\le \dim Y + d,$$ 
where the middle inequality is given by \cite{EGA}, IV.5.6.5.

(d) The subspace $T$ of $X$ is noetherian and, hence, it has finitely many irreducible 
components (\cite{Bou}, II, \S 
4.2, Prop.\ 8 (i), and Prop.\ 10).
Let $\Gamma$ be an irreducible component of $T$. Let $\overline{\Gamma}$ be its 
closure in $X$. Since  $\overline{\Gamma}$ is also irreducible, it has a generic point $\xi\in X$. 
We claim that $\xi \in \Gamma$, so that $ \xi$ is also the generic point of $\Gamma$.
Indeed, suppose that $\Gamma$ is contained in a constructible 
$W:=\cup_{i=1}^m U_i \cap F_i$, with $U_i$ open and $F_i$ closed
in $X$, and such that $(U_i \cap F_i) \cap \Gamma \neq \emptyset$ for all $i=1,\dots, m$. Then there exists $j$ such that $\overline{\Gamma} \subset F_j$. Since $U_j$ contains an element of $\Gamma$ by hypothesis, we find that it must also contain $\xi$, 
so that $\xi \in W$. The subset  $\Gamma$  is pro-constructible in $X$ since it is closed in the pro-constructible $T$.
Hence, by definition, $\Gamma $ is the intersection of constructible subsets, which all contain $\xi$. Hence, $\xi \in \Gamma$.
\qed

\medskip 

\begin{emp} \label{zero.locus} Let $X$ be a scheme. 
Let $\F$ be a  quasi-coherent $\cO_X$-module, such as a finitely presented $\cO_X$-module 
(\cite{EGA}, Chap.\ 0, (5.2.5)). 
Fix a section $f\in H^0(X, \F)$. For $x \in X$, 
denote by $f(x) $ the canonical image of $f$ in the fiber 
$\F(x) := \F_x \otimes_{\cO_{X,x}} k(x)$. 
We say that {\it $f$ vanishes at $x$ if $f(x)=0$} (in $\F(x)$). 
Define 
$$Z(\F, f):=\{ x\in X \mid f(x)=0 \}$$ 
to be the \emph{zero-locus of $f$}. 

Let $q: X'\to X$ be any morphism of schemes. Let 
$\F':=q^*\F$, and let $f'\in H^0(X', \F')$ be the canonical image of $f$. Then
$$Z(\F',f')= q^{-1}(Z(\F,f)).$$ Indeed, 
for any   $x'\in q^{-1}(x)$, the natural morphism 
$\F(x) \to \F'(x')=\F(x)\otimes_{k(x)} k(x')$ is injective.

When $\F$ is invertible or, more generally, locally free, 
then $Z(\F,f)$ is closed in $X$. 
As our next lemma shows,  $Z(\F,f)$ in general is locally constructible. When $X$ is noetherian, this is proved for instance in \cite{Pearlstein-Schnell}, Proposition 5.3.
 We give here a different proof. 
\end{emp}

\begin{lemma}\label{zero-locus-f} Let $X$ be a scheme and 
let $\F$ be a  finitely presented $\cO_X$-module. 
Then the set $Z(\F, f)$ is locally constructible in $X$. 
\end{lemma}

\begin{proof} 
Since the statement is local on $X$, it suffices to prove the lemma
when $X=\Spec A$ is affine.
We can use the stratification $X=\cup_{1\le i\le n} X_i$ of $X$
described in \cite{EGA}, IV.8.9.5: 
each $X_i$ is a quasi-compact subscheme of $X$, and 
$\F_i:=\F\otimes_{\cO_X} \cO_{X_i}$ is flat on 
$X_i$. Let $f_i$ be the 
canonical image of $f$ in $H^0(X_i, \F_i)$. Then 
$$Z(\F, f)=\cup_{1\le i\le n} Z(\F_i, f_i).$$ 
Since $\F_i$ is finitely presented and flat, it is projective
(\cite{Laz}, 1.4) and, hence, locally free. 
So $Z(\F_i, f_i)$ is closed in $X_i$. 
\end{proof}

\begin{emp} \label{prelimSigma}
Let $\pi: X\to Y$ be a finitely presented\footnote{A morphism $\pi:X \to Y$ is {\it finitely presented} (or  {\it of finite presentation}) if it is locally of finite presentation, quasi-compact, and quasi-separated (\cite{EGA}, IV.1.6.1).} 
morphism of schemes.
Let $T$ be a locally constructible subset of $X$. Set 
$$T_\pi:=\{ y \in Y \mid T_y \mbox{ contains a generic 
 point of } 
X_y \},$$
where a \emph{generic point} of a scheme $X$ is the generic point of an irreducible component of
$X$. Such a point is called a maximal point in \cite{EGA}, just before IV.1.1.5. 
Let $\F$ be a finitely presented $\cO_X$-module and fix a global section $f \in \F(X)$. 
Set
$$T_{\F, f, \pi}:=\{ y\in Y \mid f \ \text{vanishes at a generic point of}
\ X_y\}.$$
For future use, let us note the following equivalent expression for $T_{\F, f, \pi}$.
For any $y\in Y$, let 
$$\F_y:=\F\otimes_{\cO_X} \cO_{X_y}=\F\otimes_{\cO_Y} k(y),$$ 
and let $f_y$ be the canonical image of $f$ in 
$H^0(X_y, \F_y)$.  Let $x \in X_y$. Since the canonical map $\F(x)\to
\F_y(x)$ of fibers at $x$ is an isomorphism, 
$f_y$ vanishes at $x$ if and only if $f$ vanishes at $x$.
Thus
$$T_{\F, f, \pi}=\{ y\in Y \mid f_y \ \text{vanishes at a generic point of}
\ X_y\}.$$
When the morphism $\pi$ is understood, we may denote $T_{\F, f, \pi}$ simply by $T_{\F, f}$. 
Note that when $\pi= {\rm id}: X \to X$, the set $T_{\F, f, {\rm id}}$ is equal to the zero locus
$Z(\F,f)$ of $f$ introduced in \ref{zero.locus}.
\end{emp}

\begin{proposition} \label{constructible-univ} 
Let $\pi:X\to Y$ be a finitely presented morphism of 
schemes. Let $T$ be a locally constructible subset of $X$. 
Let $\F$ be a  finitely presented 
$\cO_X$-module, and fix a section $f\in H^0(X, \F)$. Then 
the subsets $T_{\pi}$ and  $T_{\F,f,\pi}$ are both locally constructible in
  $Y$. 
\end{proposition}

\begin{proof} Let us start by showing that $T_\pi$ is locally constructible in $Y$.
By definition of locally constructible, the statement is local on $Y$, 
and it suffices to prove the statement when $Y$ is affine. Assume then from now on that
$Y$ is 
affine. Since $\pi $ is quasi-compact and quasi-separated and $Y$ is affine, $X$ is also quasi-compact and quasi-separated
(\cite{EGA}, IV.1.2.6). 
Hence, $T$ is constructible and we can write it as a 
finite union of locally closed 
subsets  $T_i:=U_i\cap (X\setminus V_i)$ with $U_i$ and $ V_i$ open and 
retro-compact. Then $T_y$ contains a generic point of $X_y$ if and only if
$(T_i)_y$ contains a generic point of $X_y$ for some $i$.
Therefore,  it suffices to 
prove the statement when $T=U\cap (X\setminus V)$ with $U $ and $ V$ open and
retro-compact. 

We therefore assume now that $T= U \cap Z$, with 
$U $ and $ X\setminus Z$ open and
retro-compact. 
Fix $y\in Y$. 
We claim that {\it $T_y$ contains a generic point of $X_y$ if and only if there 
exists $x \in T_y$ such that 
$\codim_x(Z_y, X_y)=0$}.

To justify this claim, let us recall the following.
Let $\Gamma_1, \dots, \Gamma_n$ be the irreducible components of
$Z_y$ passing through $x$ ($X_y$ is noetherian). Then 
$$\codim_x(Z_y, X_y)=\min_{1\le i\le n} \{ \codim(\Gamma_i, X_y)\}$$
(\cite{EGA}, 0.14.2.6(i)). So $\codim_x(Z_y, X_y)=0$ if and only
if $Z_y$ contains an irreducible component of $X_y$ passing
through $x$.  Now, if $T_y$ contains a generic point $\xi$ of 
$X_y$,  then $Z_y$ contains the 
irreducible component $\overline{\{ \xi\}}\ni \xi$ of $X_y$ 
and $\codim_{\xi}(Z_y, X_y)=0$. Conversely, if 
$\codim_x(Z_y, X_y)=0$ for some $x\in T_y$, then $Z_y$ contains
an irreducible component $\Gamma$ of $X_y$ passing through $x$. 
As $T_y$ is open in $Z_y$, 
$T_y\cap \Gamma$ is open in $\Gamma$ and non-empty, so 
$T_y$ contains the generic point of $\Gamma$. 

Since $Z$ is closed,  we can apply \cite{EGA}, IV.9.9.1(ii),
and find that the set 
$$X_0:= \{ x \in X \ | \ \codim_x(Z_{\pi(x)}, X_{\pi(x)})=0\}$$
is locally constructible in $X$. 
It is easy to check that 
$$T_\pi = \pi( T \cap X_0).$$ 
Since $T \cap X_0$ is locally constructible,
it follows then from Chevalley's theorem (\cite{EGA}, IV.1.8.4)
that  $T_\pi$ is locally constructible in $Y$.

Let us now show that $T_{\F, f,\pi}$ is locally constructible in $Y$.
Set $T$ to be the zero locus $Z(\F,f)$ of $f$ in $X$, which is locally constructible in $X$
by \ref{zero-locus-f}. Then  $T_{\F, f,\pi}$ is nothing but the associated 
subset $T_\pi$
which was shown to be locally constructible in $Y$ in the first part of the proposition.
\end{proof}

The formation of $T_{\F, f, \pi}$ is compatible with
base changes $Y'\to Y$, as our next lemma shows.

\begin{lemma} \label{(i)} Let $\pi:X\to Y$ be a finitely presented morphism of 
schemes. Let $q: Y' \to Y$ be any morphism of schemes. Let $X':=X\times_Y Y'$ and
$\pi':X'\to Y'$. Let $\F$ be a  finitely presented 
$\cO_X$-module, and fix a section $f\in H^0(X, \F)$.
Let $\F':=\F\otimes_{\cO_Y} \cO_{Y'}$  and let $f'$ be the image of $f$ in 
$H^0(X', \F')$. Then
$T_{\F', f',\pi'}=q^{-1}(T_{\F, f,\pi})$.
\end{lemma}

\proof For any $y'\in Y'$, we have a natural $k(y')$-isomorphism 
$X'_{y'}\to (X_y)_{k(y')}$. 
Any generic point $\xi'$ of $X'_{y'}$ maps to a generic point 
$\xi$ of $X_y$, and any generic point of $X_y$ is the image of a generic
point of $X'_{y'}$.  Moreover, $f'(\xi')$ is identified with the image of $f(\xi)$ 
under the natural injection   
$\F_y(\xi) \to \F'_{y'}(\xi')=\F_y(\xi)\otimes k(\xi')$. 
\qed

\begin{emp} \label{sig}
Let $X\to Y$ be a finitely presented morphism of 
schemes, and  let $\F$ be a finitely presented 
$\cO_X$-module. Let $N\ge 1$ and let
$f_1, \dots, f_N\in H^0(X, \F)$. For each $y\in Y$, define
$$
\Sigma(y):=\left\{ (\alpha_1, \ldots, \alpha_N)\in k(y)^N \ \Big| \ 
\sum_i \alpha_i f_{i,y} \ \text{\rm vanishes at some generic point of} 
\ X_y\right\}. 
$$
When $X_y = \emptyset$, we set $\Sigma(y):= \emptyset$. 
The subset $\Sigma(y) $ depends on the data $X\to Y$, $\F$, and $\{f_1, \dots, f_N\}$. 
\end{emp}

\begin{example} \label{specialcase}
Consider  the special case in \ref{sig} where $Y=\Spec k= \{ y\}$, with $k$ a field.
For each generic point $\xi $ of $X=X_y$, consider the $k$-linear map
$$k^N \longrightarrow \F \otimes k(\xi), \quad (\alpha_1,\dots, \alpha_N) \longmapsto \sum \alpha_i f_i(\xi).$$
The kernel $K(\xi)$ of this map
is a linear subspace of $k^N$ and, hence, can be defined by a system of homogeneous polynomials of degree $1$.
The same equations define a closed subscheme $T({\xi})$ of $\mathbb A^N_y$.
Then the set $\Sigma(y)$ is the union of the sets $K(\xi)$, where the union is taken over all the  generic points of $X$,
and $\Sigma(y)$ is the subset of
$k(y)$-rational points of the closed scheme $T:= \cup_\xi T({\xi})$ of  $\mathbb A^N_Y$. 
This latter statement is generalized  to any base $Y$ in our next proposition.
\end{example}

\begin{proposition} \label{constructible-conditions}
Let $X\to Y$, $\F$, and 
$\{f_1, \dots, f_N\} \subset H^0(X, \F)$, be as in {\rm \ref{sig}}. Then there exists a 
locally constructible subset 
$T$
of $\mathbb A^N_Y$ such
that for all $y\in Y$,
the set of $k(y)$-rational points of $\mathbb A^N_{k(y)}$ 
contained in $T_y$ is equal to  $\Sigma(y)$. Moreover:
\begin{enumerate}[\rm (a)]
\item 
The set $T$ satisfies the 
following natural compatibility with respect to base change. 
Let $Y' \to Y$ be any morphism of schemes, and denote by 
$q: {\mathbb A}^N_{Y'} \to {\mathbb A}^N_Y$ the associated morphism. Let 
$X':=X\times_Y Y'\to Y'$ and let 
$\F':=\F\otimes_{\cO_Y} \cO_{Y'}$.  Let  $f_1', \dots, f_N'$ be the images of $f_1, \dots, f_N$ in $H^0(X', \F')$. 
Then the constructible set $T'$ associated with the data $X'\to Y'$, $\F'$, and $f'_1, \dots, f'_N$, is
equal to $q^{-1}(T)$. In particular, 
for all $y'\in Y'$,   
the set of $k(y')$-rational points of $\mathbb A^N_{k(y')}$ 
contained in 
$(q^{-1}(T))_{y'}$ 
is equal to the set $\Sigma(y')$ associated with $f_1',\dots, f'_N\in H^0(X', \F')$. 
\item We have $\dim T \leq \dim Y + \sup_{y \in Y} \dim T_y$ when 
$Y$ is noetherian. In general
for each $y \in Y$, $\dim T_y$ is the maximum
 of the dimension over $k(y)$ of
 the kernels of the $k(y)$-linear maps
$$k(y)^N \longrightarrow \F \otimes k(\xi), \quad  
(\alpha_1, \dots, \alpha_N)\longmapsto \sum_i \alpha_i f_i(\xi),$$
for each generic point $\xi$ of $X_y$. 
\end{enumerate}
\end{proposition}

\begin{proof} 
Let $\pi: \mathbb A^N_X\to \mathbb A^N_Y$ be the finitely presented morphism
induced by the given morphism $X \to Y$. Let $p: \mathbb A^N_X\to X$ be the natural projection,
and consider the 
finitely 
presented sheaf $p^*\F$ on $\mathbb A^N_X$ induced by $\F$.

Write  $\mathbb A^N_{\mathbb Z}= \Spec {\mathbb Z}[u_1, \dots, u_N]$, 
and identify  $H^0(\mathbb A^N_X, p^*\F)$ with
$H^0(X, \F)\otimes_\mathbb Z \mathbb Z[u_1, \dots, u_N].$
Using this identification, let $f \in H^0(\mathbb A^N_X, p^*\F)$
denote the section corresponding to $\sum_{1\le i\le N} f_i\otimes u_i$.
Apply now Proposition~\ref{constructible-univ} to the data
$\pi: \mathbb A^N_X\to \mathbb A^N_Y$, $p^*\F$, and $f$, to obtain 
the locally constructible subset $T:=T_{p^*\F, f}$ of $\mathbb A^N_Y$.

Fix $y\in Y$, and let $z$ be a $k(y)$-rational point in $ \mathbb A^N_Y$ above $y$. 
We may write $ z=(\alpha_1, \dots, \alpha_N)\in k(y)^N$.
The fiber of $\pi$ above $z$ is 
isomorphic to $X_y$, and the section   
$ f_z\in (p^*\F)_z$ is identified with 
$\sum_i \alpha_i f_i \in H^0(X_y, \F_y)$. 
Therefore, it follows from the definitions that $z\in T_y$ if and only if $z\in \Sigma(y)$, and the
first part of the proposition is proved.  

(a) The compatibility of 
$T$ with respect to a base change $Y'\to Y$ results from Lemma \ref{(i)}. 

(b) The inequality on the dimensions follows from \ref{lem.construct} (c).
By the compatibility described in (a), we are immediately reduced to the
case $Y=\Spec k$ for a field $k$, which is discussed in \ref{specialcase}.
\end{proof}

\end{section} 

\begin{section}{Sections in an affine space avoiding pro-constructible subsets} 
\label{Constructible subsets}

\medskip 
The following theorem is an essential part of our method for 
producing interesting closed  subschemes of a scheme $X$ when $X \to S$ is projective and
$S$ is affine. 

\begin{theorem} \label{globalize} 
Let $S=\Spec R$ be a noetherian affine scheme. Let 
$T:=T_1\cup \ldots \cup T_m$ 
be a finite union of pro-constructible subsets 
of $\mathbb A^N_S$. Suppose that:
\begin{enumerate}[\rm (1)]
\item There exists an open subset $V\subseteq S$ with 
zero-dimensional complement such that 
for all $i\le m$, $\dim (T_i\cap \mathbb A^N_{V})<N$ and 
 $(T_i)_s$ is constructible in $\mathbb A^N_{k(s)}$ for all $s\in V$. 
\item For all $s\in S$, 
there exists a $k(s)$-rational point in $\A^N_{k(s)}$ which does not 
belong to $T_s$. 
\end{enumerate}
\smallskip
Then there exists a section $\sigma$ of $\pi: \A^N_S\to S$
such that $\sigma(S) \cap T = \emptyset$. 
\end{theorem}

\proof 
We proceed by induction on $N$, using  Claims (a) and (b) below.  

\medskip
\noindent {\bf Claim (a).} {\it There exists $\delta\ge 1$ such that for all $s\in V$, 
$T_s$ is contained in a hyper\-sur\-face  in $\mathbb A^N_{k(s)}$ of 
degree at most  $ \delta$}.

\proof 
It is enough to prove the claim for each $T_i$. So to lighten the notation
we set $T:=T_i$ in this proof. 
Thus, by hypothesis, $\dim T\cap \mathbb A^N_V<N$. 
We start by proving that for each $s \in V$, there exist a positive
integer $\delta_s$ and an ind-constructible subset (see
\ref{pro-constr}) $W_s$ of $V$ containing
$s$,   such that for each $s' \in W_s$, $T_{s'} $ is contained in a
hypersurface of degree $\delta_s$ in ${\mathbb A}^N_{k(s')}$. Indeed, let $s \in V$.  As 
$\dim T_s\le \dim T\cap \mathbb A^N_V<N$, and $T_s$ is constructible,  
$T_s$ is not dense in ${\mathbb A}^N_{k(s)}$ (Lemma \ref{lem.construct}(a)). 
Thus, there exists some polynomial $f_s$ of degree $\delta_s>0$ whose
zero locus contains $T_s$. Hence, for some affine open neighborhood 
$V_s$ of $s$, we can find a polynomial 
$f \in \cO_V(V_s)[t_1,\dots, t_N]$ of degree $\delta_s$, lifting 
$f_s$ and defining a closed subscheme $V(f)$ of ${\mathbb A}^N_{V_s}$. 

Let $W_1:= \pi({\mathbb A}^N_{V_s} \setminus V(f))$, which is constructible in $V$
by Chevalley's Theorem.  Let $W_2$ be the 
complement in $V$ of $ \pi(T \cap ({\mathbb A}^N_{V_s} \setminus V(f)))$, which is 
ind-constructible in $V$ since $\pi(T \cap ({\mathbb A}^N_{V_s} \setminus V(f)))$ is pro-constructible in $V$ (\cite{EGA}, IV.1.9.5 (vii)).
Hence, both $W_1$ and $W_2$ are ind-constructible and contain $s$. The intersection $W_s:= W_1 \cap W_2$
is the desired ind-constructible subset containing $s$.
Since $V$ is quasi-compact because it is noetherian, and since each $W_s$ is ind-constructible, it follows from 
\cite{EGA}, IV.1.9.9,  
that there exist finitely many points $s_1, \dots, s_n $ of $V$ such that
$V = W_{s_1} \cup \ldots \cup W_{s_n}$. We can take 
$\delta:= \max_i\{ \delta_{s_i} \}$, and Claim (a) is proved.  \qed
\medskip 

A proof of the following lemma in the affine case is given 
in \cite{Samuel}, Proposition 13. We provide here an alternate proof.

\begin{lemma} \label{finiteresidue} Let S be any scheme. Let $c \in \mathbb N$.
Then the subset $\{ s\in S \mid \card(k(s))
\le c\}$ is closed in $S$ and has dimension at most $ 0$. When $S$ is noetherian,
this subset is then finite. 
\end{lemma}

\proof It is enough to prove that when $S$ is a scheme 
over a finite prime field $\mathbb F_p$, and $q$ is a power of $p$, 
the set $\{ s\in S \mid \card(k(s))=q \}$ is closed of dimension $\le 0$.

Let ${\mathbb F}_q$ be a field with $q$ elements. Then any point 
$s\in S$ with $\card(k(s))=q$ is the image by the projection
$S_{{\mathbb F}_q} \to S$ of a rational point of $S_{{\mathbb F}_q}$. 
Therefore we can suppose that $S$ is a 
${\mathbb F}_q$-scheme and we have to show that $S({\mathbb F}_q)$ 
is closed of dimension $0$. Let $Z$ be the Zariski closure of $S({\mathbb F}_q)$ in
$S$, endowed with the reduced structure.  
Let $U$ be an affine open subset of $Z$. Let $f\in \cO_Z(U)$. 
For any $x\in U({\mathbb F}_q)$, $(f^q-f)(x)=0$ in $k(x)$, hence $x\in V(f^q-f)$.
As $U({\mathbb F}_q)$ is dense in $U$ and $U$ is reduced, we have
$f^q-f=0$. For any irreducible component $\Gamma$ of $U$, this
identity then holds on $\cO(\Gamma)$, so $\Gamma$ is just a rational point. 
Hence $U=U({\mathbb F}_q)$ and $\dim U=0$. Consequently, $Z=S({\mathbb F}_q)$
is closed and has dimension $0$. 
\qed

\medskip
The key to the proof of Theorem \ref{globalize}
is the following assertion:

\medskip
\noindent {\bf Claim (b).} {\it Suppose $N\ge 1$. 
Then there exist $t:=t_1 +a_1 \in R[t_1,\dots, t_N]$ with $a_1\in R$, 
and an open subset $U\subseteq S$ with 
zero-dimensional complement, such that  
$H:=V(t)$ is $S$-isomorphic to  $\mathbb A^{N-1}_S$ and the pro-constructible
subsets $T_1\cap H$, $\dots$, $T_m\cap H$ of $H$ 
satisfy: 
\begin{enumerate}[\rm (i)]
\item For all $i\le m$, $\dim (T_i\cap H_U)<N-1$,
and $(T_i \cap H)_s$ is constructible in $H_s$ for all $s\in U$. 
\item For all $s\in S$, 
there exists a $k(s)$-rational point in $H_s$ which does not 
belong to $T_s \cap H_s$. 
\end{enumerate}}
\smallskip 

Using Claim (b), we  conclude the proof of Theorem \ref{globalize} as follows.
First, note that when $N=0$, Condition \ref{globalize}\ (2) implies that 
$T=\emptyset$ and the theorem trivially holds true. 
When $N\ge 1$, we apply Claim (b) repeatedly to obtain a sequence of closed
sets $$\mathbb A^N_S \supset V(t_1+a_1) \supset  \dots \supset 
 V(t_1+a_1, t_2+a_2, \dots, t_N+a_N).$$
The latter set is the image of the desired section, as we saw in the 
case $N=0$.

\medskip
\noindent {\it Proof of} Claim (b):  
Let $\{\xi_1,\dots, \xi_{\rho}\}$ be the set of generic points of all the
irreducible components of the pro-constructible sets $T_i\cap \A^N_V$, $i=1,\dots, m$
(see \ref{lem.construct} (d) for the existence of generic points). Upon renumbering  these points  
if necessary, we can assume that for some $r\le \rho$,  
the image of $\xi_i$ under $\pi: \A^N_S\to S$ has finite 
residue field if and only if $i> r$. Let $\delta>0$.  
Let $Z$ be the union of $S\setminus V$ with 
$\{\pi(\xi_{r+1}),\ldots, \pi(\xi_{\rho})\}$ and with the finite subset 
of the closed points $s$ of $S$ satisfying $\card (k(s)) \le \delta$ (that this set is finite
follows from \ref{finiteresidue}). We will later set $\delta$ appropriately to be able to use Claim (a). 
For each $s \in Z$, we can use  \ref{globalize} (2) and fix a
$k(s)$-rational point $x_s \in {\mathbb A}_{k(s)}^N \setminus {T_s}$. 

We now construct a closed subset $V(t) \subset \mathbb A^N_S$ which contains $x_s$ for all $s\in Z$, and  does not contain any
$\xi_i$ with $i\le r$. 
Since every point of $Z$ is closed in $S$, the Chinese Remainder Theorem
implies that the  canonical map $R\to \prod_{s\in Z} k(s)$ 
is surjective. Let $a \in R$ be such that $a \equiv  t_1(x_s)$ in $k(s)$, for all $s\in Z$.
Replacing $t_1$ by $t_1-a$, we can assume that $t_1(x_s)=0$ for all $s\in Z$. 
Let 
$\p_{j} \subset R[t_1,\dots, t_N]$ be the prime ideal corresponding to $\xi_j$. 
Let $\m_s \subset R$ denote the maximal ideal of $R$ corresponding to $s \in Z$. 
Let $I :=\cap_{s \in Z} \m_s$, and in case  $Z=\emptyset$, we let $I:=R$. 
For $t \in R[t_1,\dots, t_N]$, let $I+t:=\{a+t \mid a\in I\}$.
{\it We claim that:} 
$$ I+t_1 \not\subseteq \cup_{1\le j\le r}\p_j.$$
Indeed, the intersection $(I+t_1)\cap\p_j$ is either empty, or 
contains $a_j + t_1$ for some $a_j\in I$. In the latter case, 
$(I+t_1)\cap\p_j=t_1+a_j+ (\p_j\cap I)$. 
If $ I+t_1 \subseteq \cup_{1\le j\le r}\p_j$, then every $t_1+a$ with $a \in I$ 
belongs to some $t_1+a_j+ (\p_j\cap I)$. Let $\q_j:=R\cap \p_j$.
It follows that 
$$I\subseteq \cup_j (a_j+\q_j) $$ 
where the union runs over a subset of $\{1, \ldots, r\}$. 
Since the domains $R/\q_j$ are all infinite when $j \leq r$, 
Lemma \ref{trans-pal-2} below  implies  
that $I$ is contained in some $\q_{j_0}$ for $1 \leq j_0\le r$. As 
$I =\cap_{s \in Z} \m_s$, we find that $\q_{j_0}= \m_s$ for some $s\in Z$. This is a contradiction,
since for $j\leq r$, $\pi(\xi_j)$ does not belong to $Z$ because
the residue field of $\pi(\xi_j)$ is infinite and 
$\pi(\xi_j) \notin S\setminus V$. This proves our claim.

Now that the claim is proved, we can choose 
$t\in (I+t_1)\setminus \cup_{1\le j\le r}\p_j$.
Clearly, the closed subset $H:=V(t) \subset \mathbb A^N_S$   does not contain any
$\xi_i$ with $i\le r$. 
Since $t$ has the form  $t=t_1+a_1$ for some $a_1 \in I = \cap_{s \in Z} \m_s$, 
we find that $V(t)$ contains $x_s$ for all $s \in Z$. 
Let $U:=S \setminus Z$. 
The complement of $U$   in $S$ is a finite set of closed points of $S$.
It is clear that $H:=V(t)$ is $S$-isomorphic to $ \mathbb A^{N-1}_S$,
and that for each $i$, the fibers of $T_i\cap H \to S$ are constructible. 

Let us now prove (i), i.e., that   $\dim(T_i\cap H_{U}) < N-1$ for all $i\le m$. 
Let $\Gamma$ be an irreducible component of some $T_i\cap \mathbb A^N_V$, 
with generic point $ \xi_j$ for some $j$. If $j>r$, then $\pi(\xi_j)\in Z$ and $\Gamma\cap
\mathbb A^N_U=\emptyset$. Suppose now that $j\le r$. Then $\Gamma\cap \mathbb A^N_U$
is non-empty and open in $\Gamma$, hence irreducible. 
By construction, $H $ does not contain $\xi_j$ since $j \leq r$.
So $\Gamma\cap H_U$ is a proper
closed subset of the irreducible space $\Gamma\cap \mathbb A^N_U$. Thus
$$\dim (\Gamma\cap H_U)< \dim (\Gamma\cap \mathbb A^N_U) \le \dim \Gamma<N.$$ 
As $T_i\cap H_U$ is the finite union of its various 
closed subsets $\Gamma\cap H_U$, 
this implies that $\dim (T_i\cap H_U)< N-1$.

Let us now prove (ii), i.e., that for all $s \in S$, $H_s$ contains a $k(s)$-rational 
point that does not belong to $T_s$.
When $s\in Z$, $H$ contains the $k(s)$-rational point $x_s$  and this point  does not belong to $T_s$.
Let now $s \notin Z$. Then $|k(s)|\geq \delta +1$ by construction. Choose 
now $\delta$ so that the conclusion of Claim (a) holds:
for all $s\in V$, 
$T_s$ is contained in a hyper\-sur\-face  in $\mathbb A^N_{k(s)}$ of   
degree at most  $ \delta$.
Then, since $t$ has degree $1$, we find that 
$H_s\cap T$ is contained in a hypersurface $V(f)$ of $H_s$ with 
$\deg(f)\leq  \delta$. 

We conclude that $H_s$ contains a $k(s)$-rational 
point that does not belong to $T_s$ using the following claim: 
{\it Assume that $k$ is either an infinite field or that $|k|=q  \ge \delta+1$.
Let $f\in k[T_1,\ldots, T_\ell]$ with $\deg(f) \le \delta$, $f\neq 0$. 
Then $V(f)(k)\varsubsetneq \mathbb A^\ell(k)$.}
When $k$ is a finite field, we  use the bound 
$|V(f)(k)| \leq \delta q^{\ell-1} + (q^{\ell-1} -1)/(q-1)<q^\ell$  found in \cite{Ser2}.
When $k$ is infinite, we can use induction on $\ell$ to prove the claim.
\qed 
\medskip  

Our next lemma follows from \cite{McAdam}, Theorem 5. 
We provide here a more direct proof using the earlier reference \cite{Neumann}. 
\begin{lemma} \label{trans-pal-2} 
Let $R$ be a commutative ring, and let $\q_1,\ldots, \q_r$
be (not necessarily distinct) prime ideals of $R$ with infinite 
quotients $R/\q_i$ for all $i=1,
\dots, r$.
Let $I$ be an ideal of $R$ and suppose that
 there exist $a_1,\ldots, a_r\in R$ such that 
$$I\subseteq \cup_{1\le i\le r} (a_i+\q_i).$$
Then $I$ is contained in the union of those 
$a_i+\q_i$ with $I\subseteq \q_i$. In particular, 
$I$ is contained in at least one $\q_i$. 
\end{lemma}

\proof We have $I=\cup_{i} ((a_i+\q_i)\cap I)$. If 
$(a_i+\q_i)\cap I\ne\emptyset$, then it is equal to 
$\alpha_i + (\q_i\cap I)$ for some $\alpha_i\in I$. Hence
$$I=\cup_i (\alpha_i + (\q_i\cap I))$$
where the union runs on part of $\{1,\ldots, r\}$. By \cite{Neumann}, 4.4,
$I$ is the union of those 
$\alpha_i+(\q_i\cap I)$ with $I/(\q_i\cap I)$ finite. 
For any such $i$, the ideal $(I+\q_i)/\q_i$ of $R/\q_i$ is finite and,
hence, equal to $(0)$ because $R/\q_i$ is an infinite domain.
\qed

\begin{remark} One can show that the conclusion of Theorem \ref{globalize} holds 
without assuming in \ref{globalize} (1)
that $T_s$ is constructible in ${\mathbb A}_{k(s)}^N$ for all $s \in V$.
Since we will not need this statement, let us only note that
{\it when $S$ has only finitely many points  with
finite residue field, then the conclusion of 
Theorem {\rm \ref{globalize}} holds if in {\rm \ref{globalize} (1)} we remove 
the hypothesis that $T_s$ is constructible for all $s \in V$.} 
Indeed, with this hypothesis, we do not need to use Claim (a). 
First shrink $V$ so that $k(s)$ is infinite for all $s\in V$,
and then proceed to construct the closed subset $H=V(t)$ discussed in Claim (b).
The use of Claim (a) in the proof of (ii) in Claim (b) can be avoided 
using our next lemma.
\end{remark}

Let $k$ be an infinite field, and let $V \subset {\mathbb A}^N_k$ be a closed subset.
The property that {\it if $\dim(V)<N$, then $V(k) \neq  {\mathbb A}^N_k(k)$},
can be generalized as follows.

\begin{lemma} \label{proconstructible.gen} Let $k$ be an infinite field.
Let $T \subset {\mathbb A}^N_k$ be a pro-constructible subset with $\dim(T) <N$. 
Then $T$ does not  contain all $k$-rational points of ${\mathbb A}^N_k $.
\end{lemma}
\begin{proof}
Assume that  $T$  contains all $k$-rational points of ${\mathbb A}^N_k $.
We claim first that $T$ is irreducible. Indeed,  if $T=(V(f)\cap T) \cup
(V(g)\cap T)$, then $V(fg)=V(f)\cup V(g)$ contains all $k$-rational points of ${\mathbb A}^N_k$.
Thus $V(fg)= {\mathbb A}^N_k$ and either $V(f) \cap T = T$ or $V(g) \cap T = T$.
Since $T$ is irreducible, it has a generic point $\xi$ (see \ref{lem.construct} (d)), and the closure $F$ of $\xi$ in ${\mathbb A}^N_k$
contains all $k$-rational points of ${\mathbb A}^N_k$. Hence, $F= {\mathbb A}^N_k$, so that $T$ then contains the generic point of ${\mathbb A}^N_k$.
Consider now an increasing sequence of closed linear subspaces $F_0 \subset F_1 \subset \dots \subset F_{N}$ contained in ${\mathbb A}^N_k$, with $F_i \cong {\mathbb A}^i_k$.
Then $T \cap F_i$ contains all $k$-rational points of $F_i$ by hypothesis, and the discussion above shows that it contains then the generic point of $F_i$. It follows that $\dim(T) =N$.
\end{proof}

\begin{remark} The hypothesis in \ref{globalize}\ (1) on the dimension of $T$ is needed.
Indeed, let $S= \Spec {\mathbb Z}$, and $N=1$. 
Consider the closed subset $V(t^3-t)$ of $\Spec {\mathbb Z}[t] = {\mathbb A}^1_{\mathbb Z}$. Let $T$ be the constructible subset of ${\mathbb A}^1_{\mathbb Z}$
obtained by removing from $V(t^3-t)$ the maximal ideals $(2,t-1)$ and $(3,t-1)$. 
Then, for all $s \in S$, the fiber $T_s$ is distinct from ${\mathbb A}^1_{k(s)}(k(s))$, and $\dim T_s =0$. 
However, $\dim T=1$, and we note now that there exists no section of ${\mathbb A}^1_{\mathbb Z}$
disjoint from $T$. Indeed, let $V(t-a)$ be a section. If it is disjoint from $T$, 
then $a \neq 0,1,-1$, and $6 \mid a-1$. So there exists a prime $p>3$ with $p \mid a$,
and $V(t-a)$ meets $T$ at the point $(p, t)$. 

For a more geometric example, let $k$ be any infinite field. 
Let $S=\Spec k[u]$ and $\A^1_S= \Spec k[u,t]$. When $T:=V(t^2-u) \subset \A^1_S $, 
then $\A^1_S\setminus T$ does not contain any section $V(t-g(u))$ of $\A^1_S$. Indeed,
otherwise $(t^2-u, t-g(u))=(1)$, and  $g(u)^2-u$ would be an element of $k^*$. 
\end{remark}

\end{section}

\begin{section}{Existence of hypersurfaces} 
\label{Main application}

Let us start by introducing the terminology needed to state the main results of this section.

\begin{emp} \label{zerolocusinvertible} Let $X$ be any scheme. 
A global section $f$ of an invertible sheaf ${\mathcal L}$ on  $X$
 defines a closed subset $H_f$ of $X$,  
 consisting of all points $x \in X$ where the stalk $f_x$ does not generate ${\mathcal L}_x$.
 Since $\cO_X f\subseteq {\mathcal L}$, 
 the ideal sheaf ${\mathcal I}:=
(\cO_X f)\otimes {\mathcal L}^{-1}$ endows $H_f$ with the structure of
a closed subscheme of $X$.
When $X$ is noetherian and $H_f \neq \emptyset$, it follows from Krull's Principal Ideal Theorem that 
any irreducible component $\Gamma$ of $H_f$ has codimension at most $1$ in $X$.

Assume now that $X \to \Spec R$ is a projective morphism, and write  $X=\Proj A$,
where $A$ is the quotient of a polynomial ring $R[T_0,\dots, T_N]$ by a homogeneous ideal $I$.
Let $\cO_X(1)$ denote the  very ample sheaf arising from this
presentation of $X$.  
Let $f\in A$ 
be a homogeneous element of degree $n$. Then $f$ can be identified with a global section $f\in H^0(X, \cO_X(n))$, and $H_f$ is
the closed subscheme $V_+(f)$ of $X$ defined by the homogeneous ideal $fA$. 
When $X \to S$ is a quasi-projective morphism
and $f$ is a global section of a very ample invertible sheaf ${\mathcal L}$ relative to $X \to S$, we may also sometimes denote the 
closed subset $H_f$ of $X$ by $V_+(f)$. 

Let now $S$ be any 
affine scheme and $X \to S$ any morphism.
We  call the closed subscheme $H_f$ of $X$ a \emph{hypersurface} (relative to $X \to S$) when no
irreducible component of \emph{positive dimension} of $X_s$ is contained in $H_f$, for all $s\in S$. 
If, moreover, the ideal sheaf ${\mathcal I}$ is invertible, 
we say that the hypersurface $H_f$ is \emph{locally principal}.
Note that in this case, $H_f$ is the support 
of an effective Cartier divisor on $ X$.
Hypersurfaces satisfy the following elementary properties.
\end{emp}

\begin{lemma} \label{hypersurfaces-properties} 
Let $S$ be affine. 
Let $X \to S$ be a finitely presented morphism. 
Let ${\mathcal L}$ be an invertible sheaf on $X$ and let $f \in H^0(X,
{\mathcal L})$
be such that  $H:= H_f$ 
is a hypersurface on $X$ relative to $X \to S$.  
\begin{enumerate}[\rm (1)]
\item If $\dim X_s\ge 1$, then 
$\dim H_s\leq \dim X_s -1$. If, moreover, 
$X \to S$ is projective, ${\mathcal L}$ is ample, 
and $H \neq \emptyset$, then $H_s$ meets every irreducible component of positive dimension of $X_s$, and in particular $\dim H_s=\dim X_s -1$.
\item  The morphism $H\to S$ is finitely presented.  
\item Assume that $X\to S$ is flat of finite presentation. 
Then $H$ is locally principal and flat over $S$ if 
and only if for all $s\in S$, $H$ does not contain 
any associated point of $X_s$. 
\item Assume $S$ noetherian. If $H$ does not contain any associated point of $X$,
then $H$ is locally principal.
\end{enumerate}
\end{lemma}

\proof (1) Recall that by convention, if $H_s$ is empty, then $\dim H_s <0$, and the inequality is satisfied.
Assume now that $H_s$ is not empty. By hypothesis, $H_s$ does not contain any irreducible component 
of $X_s$ of positive dimension. Since $H_s$ is locally defined by one equation, we obtain that  $\dim H_s\leq \dim X_s-1$.
The  strict inequality  
may occur for instance in case $\dim X_s \geq 2$, and $H_s$ does not meet any  component of $X_s$ of maximal dimension. 

Consider now the open set $X_f:= X \setminus H$. Under our additional
hypotheses, for any $s\in S$, 
$X_f \cap X_s=(X_s)_{f_s}$ is affine and, thus, can only contain
irreducible components of dimension $0$ of the projective scheme
$X_s$. 

(2) Results from the fact that 
$H$ is locally defined by a single equation in $X$.

(3) See \cite{EGA}, IV.11.3.8, c) $\Leftrightarrow$  a). Each fiber $X_s$ is noetherian. Use the fact that in a noetherian ring,
an element is regular if and only if it is not contained in any associated prime.

(4) The property is local on $X$, so we can suppose $X=\Spec A$ is affine and
$\mathcal L=\cO_X\cdot e$ is free. So $f=he$ for some $h\in A$. 
The hypothesis $H\cap \Ass(X)=\emptyset$ implies that $h$ is a 
regular element of $A$.
So the ideal sheaf $\mathcal I=(\cO_X f)\otimes 
\mathcal L^{-1}$ is invertible. 
\qed

\medskip 

We can now state the  main results of this section. 

\begin{theorem} \label{exist-hyp} 
Let $S$ be an affine noetherian scheme of finite dimension, and let 
$X\to S$ be a quasi-projective morphism with a given very 
ample invertible sheaf $\cO_X(1)$. 
\begin{enumerate}[{\rm (i)}] 
\item Let $C$ be a closed subscheme of $X$;
\item Let $F_1, \dots, F_m$ be locally closed subsets\footnote{
Recall that a {\it locally closed} subset $F$ of a topological space $X$ is 
the intersection of an open subset $U$ of $X$ with a closed subset $Z$ of $X$.
When $X$ is a scheme, we can endow 
$F$ with the structure
of a {\it subscheme} of $X$ by considering $U$ as an open subscheme of $X$ and 
$F$ as the closed subscheme $Z \cap U$ of $U$ endowed with the reduced induced structure.
} of $X$ such that for
  all $s\in S$ and for all $i\le m$, $C_s$ does not contain
any  irreducible
component  of positive dimension of $(F_i)_s$; 
\item Let $A$ be a finite subset of $X$ such that $A\cap C=\emptyset$. 
\end{enumerate}
Then there exists $n_0>0$ such that for all $n\ge n_0$, there exists
a global section $f$ of $\cO_X(n)$ such that:
\begin{enumerate}[\rm (1)]
\item $C$ is a closed subscheme of $H_f$, 
\item for all $s \in S$ and for all $i\le m$, $H_f$ does not contain any irreducible component  of positive dimension of
$F_i\cap X_s$, and  
\item $H_f\cap A=\emptyset$. 
Moreover, 
\item if, for all $s \in S$,  $C$ does not contain any irreducible component  of
positive dimension of $X_s$, then there exists $f$ as above such that $H_f$ is a
hypersurface relative to $X \to S$. 
If in addition $C\cap \Ass(X)=\emptyset$, then there exists $f$ as above such that $H_f$ is a
locally principal hypersurface.
\end{enumerate}
\end{theorem}

We will first give a complete proof of Theorem \ref{exist-hyp}
in the case where $X \to S$ is projective in \ref{Proofreductiondimension2}, after a series of technical
lemmas. The proof of \ref{exist-hyp}
when $X\to S $ is only assumed to be quasi-projective 
is given in \ref{quasiproj}.
Theorem \ref{exist-hyp} will be generalized to the case where $S$  is not noetherian in
\ref{bertini-type-0}.

Theorem \ref{pro.reductiondimension2} below is the key to reducing the proof of the Moving Lemma
\ref{mv-1-cycle-local} (a)
 to the case of relative dimension 1.
This theorem is stated in a slightly different form in the introduction, and we note 
in \ref{lem.remcodim} (3) that the two versions are compatible.
 
\begin{theorem} \label{pro.reductiondimension2}
Let $S$ be an affine noetherian scheme of finite dimension, and let 
$X\to S$ be a quasi-projective morphism with a given very 
ample invertible sheaf $\cO_X(1)$ relative to $X \to S$. Assume that 
the hypotheses {\rm (i)}, {\rm (ii)}, and {\rm (iii)} in {\rm \ref{exist-hyp}} hold. 
Suppose further 
that 
\begin{enumerate}[\rm (a)]
\item $C\to S$ is
finite,  
\item $C\to X$ is a regular immersion 
and $C$ has pure codimension\footnote{By {\it pure} codimension $d$, we mean that
  every irreducible component of $C$ has codimension $d$ in $X$.} $d> \dim S$ in $X$, and
\item for all $s\in S$, $\codim(C_s, X_s)\ge d$.
\end{enumerate}
Then there exists $n_0>0$ such that for all $n\ge n_0$ there exists
a global section $f$ of $\cO_X(n)$ such that 
$H_f$ satisfies {\rm (1)}, {\rm (2)}, and {\rm (3)} in {\rm  \ref{exist-hyp}}, and such that
$H_f$ is a locally principal hypersurface,  
$C \to H_f$ is a regular immersion, and $C$ has pure
codimension $d-1$ in $H_f$.

Suppose now
that $\dim S= 1$. 
Then there exists a closed subscheme $Y$ of $X$ such that
$C$ is a closed subscheme of $Y$ defined by an invertible sheaf of ideals of $Y$ 
(i.e.,  $C$ corresponds to an effective Cartier divisor on $Y$). Moreover,  
for all $s\in S$ and all $i\le m$, any irreducible 
component $\Gamma$ of $F_i \cap X_s$ 
is such that $\dim(\Gamma \cap Y_s) \leq \max(\dim(\Gamma)-(d-1),0)$.
In particular, if $(F_i)_s$ has positive codimension in $X_s$ in 
a neighborhood of $C_s$, then $F_i \cap Y_s$ has dimension at most $0$ 
in a neighborhood of $C_s$. 
\end{theorem}
\proof 
The main part of Theorem \ref{pro.reductiondimension2}
is given a complete proof in the case where $X \to S$ is projective in \ref{Proofreductiondimension2-a}. The proof 
when $X\to S $ is only assumed to be quasi-projective 
is given in \ref{quasiproj}. We prove here the end of the statement of Theorem \ref{pro.reductiondimension2},
where we assume that 
$\dim S = 1$.
Apply Theorem \ref{pro.reductiondimension2}  
$(d-1)$
times, starting with 
$X':=V_+(f)$,  $F'_i:=F_i\cap V_+(f)$, and $C \subseteq X'$. 
Note that at each step Condition \ref{pro.reductiondimension2} (c) 
holds by Lemma \ref{lem.remcodim} (2). 
 \qed

\begin{lemma} \label{lem.remcodim}
Let $S$ be a noetherian scheme, and let $\pi: X\to S$ be a morphism of 
finite type. Let $C$ be a 
closed subset of $X$, with $C \to S$ finite. 
\begin{enumerate}[\rm (1)]
\item Let $s \in S$ be such that $C_s$ is not empty. Then the following are equivalent:
\begin{enumerate}[\rm (a)]
\item $\codim(C_s, X_s)\ge d$. 
\item Every point $x$ of $C_s$ is contained in an irreducible component of
$X_s$ of dimension at least equal to $d$ (equivalently, $\dim_x X_s\ge d$ for all $x \in C_s$)\footnote{
Recall that $\dim_x X_s$ is the infimum of $\dim U$, where $U$ runs through the
open neighborhoods of $x$ in $X_s$.} 
\end{enumerate}
\item Let ${\mathcal L}$ be a line bundle on $X$ with a global section $f$ defining a 
closed subscheme $H_f$ which contains $C$. Let $s\in S$. Suppose that 
$\codim(C_s, X_s)\ge d$. 
Then $\codim(C_s, (H_f)_s)\ge d-1$. 
\item  Assume that $C$ has codimension $d\ge 0$ in $X$
and that each irreducible component of $C$ dominates an irreducible
component of $S$ (e.g., when $C\to S$ is flat).  
Then for all $s\in S$, 
$\codim(C_s, X_s)\ge d$. 
In particular, if $X/S $ and $C$ satisfy the hypotheses of the version of {\rm \ref{pro.reductiondimension2}}
given in the introduction, then they satisfy the hypotheses of Theorem {\rm \ref{pro.reductiondimension2}} as stated above.
\end{enumerate}
\end{lemma}

\proof (1)  This is immediate since, $X_s$ being of finite type over
$k(s)$, $C_s$ is the union of finitely many closed points of $X_s$. 

(2) We can suppose $C_s$ is not empty. Let $x \in C_s$. 
Then $x$ is contained in an irreducible component $\Gamma'$ of $X_s$ 
of dimension at least equal to $d$. Consider an irreducible component 
$\Gamma $ of $\Gamma' \cap (H_f)_s$ which contains $x$. 
Since 
$\Gamma' \cap (H_f)_s$ is defined in $\Gamma'$ by a single equation,
we find that 
$\dim(\Gamma)  \ge \dim(\Gamma')- 1 \ge d - 1$, as desired. 

(3) 
Let $\xi$ be a generic point of $C$.
By hypothesis, $\pi(\xi)$ is a generic point of $S$ and $\xi$ is closed in $X_{\pi(\xi)}$. So 
$$\dim_{\xi} X_{\pi(\xi)}=\dim \cO_{X_{\pi(\xi)}, \xi}=\dim \cO_{X, \xi}\ge\codim_{\xi}(C, X)\ge d.$$ 
The set $\{ x\in X | \dim_x X_{\pi(x)} \ge d\}$ is closed (\cite{EGA}, IV.13.1.3). Since this set
contains the generic points of $C$, it contains $C$. Hence, when $C_s $ is not empty,
$\codim(C_s, X_s)\ge d$ by (1)(b). 
When $C_s=\emptyset$, $\codim(C_s,X_s)=+\infty$ by definition and the statement of (3) also holds.  

In the version of {\rm \ref{pro.reductiondimension2}}
given in the introduction, we assume that $C$ is irreducible, that $C \to S$ is finite and surjective, and that $C$ has codimension $d > \dim S$ in $X$.
It follows  then from (3) that (c) in Theorem {\rm \ref{pro.reductiondimension2}} is automatically satisfied.
\qed 

\begin{notation}  
\label{emp.notation} 
\label{emp.proofpro.reductiondimension2} 
We fix here some notation needed in the proofs of
\ref{exist-hyp} and \ref{pro.reductiondimension2}. 
Let $S = \Spec R $ be a noetherian affine scheme. 
Consider a 
projective morphism  $X\to S$. 
Fix a very ample sheaf $\cO_X(1)$ on $X$ relative to $S$. 
As usual, if $\F$ is any quasi-coherent sheaf on $X$ and $s \in S$, 
let $\F_s$ denote the pull-back of $\F$ to the fiber $X_s$ and, 
if $x\in X$, $\F(x):=\F_x\otimes k(x)$ (see \ref{zero.locus}).  

Let $C \subseteq X $ be a closed subscheme defined by an ideal sheaf
$\J$. 
For $n\ge 1$, set $\J(n):= \J \otimes \cO_X(n)$,
and for $s \in S$, let $ \J_s(n):= \J_s \otimes \cO_{X_s}(n)=\J(n)_s$.
Let $\oline{\J}_s$ denote the image of $\J_s\to \cO_{X_s}$. When $x \in C \cap X_s$, we note the following natural isomorphisms of $k(x)$-vector spaces:
$$ (\J(n)|_C)_s(x) \longrightarrow \J_s(n)/\J_s^2(n) \otimes k(x) \longrightarrow \J_s(n) \otimes k(x) \longrightarrow \J(n) \otimes k(x)$$
and 
$$\oline{\J}_s(n)/\oline{\J}_s^2(n) \otimes k(x) \longrightarrow \oline{\J}_s(n) \otimes k(x).$$
To prove Theorem~\ref{exist-hyp} and Theorem
\ref{pro.reductiondimension2}, we will show the existence of 
$f \in H^0(X, \J(n))$, for all $n$ sufficiently large,
such that the associated closed subscheme $ H_f \subset X$ satisfies 
the conclusions of the theorems. 
To enable us to use the results of the previous section
to produce the desired $f$, we define the following sets.

Let  $n$ be big enough such that $\J(n)$ is generated by its global sections.
Fix a system of generators $f_1, \dots, f_N$ of $H^0(X, \J(n))$. 
Let $s\in S$. Denote by $\overline{f}_{i,s}$ the image of $f_i$ in 
$H^0(X_s, \oline{\J}_s(n))$.
Let $F$ be a locally closed subset of $X$. 
\begin{quote}  
$\bullet$ 
Let $\Sigma_F(s)$ denote the set of $(\alpha_1,\ldots, \alpha_N)
\in k(s)^N$ such that the closed subset 
$V_+(\sum_{i=1}^N  \alpha_i \overline{f}_{i,s})$ in $X_s$, defined by the 
global section
$\sum_{i=1}^N  \alpha_i \overline{f}_{i,s}$ of $\cO_{X_s}(n)$,
 contains at least one irreducible 
component of $F_s$ of positive dimension. 
\end{quote}
For the purpose of \ref{pro.reductiondimension2}, we will also
consider the following set. 
\begin{quote} 
$\bullet$ Let $\Sigma_C(s)$ denote the set of $(\alpha_1,\ldots, \alpha_N)
\in k(s)^N$ for which there exists 
$x\in C \cap X_s$ such that the image of 
$\sum_{i=1}^N  \alpha_i (f_{i}|_{X_s})\in H^0(X_s, \J_s(n))$ vanishes
in $\J_s(n) \otimes k(x)$.
\end{quote}

To lighten the notation, we will not always explicitly use symbols to
make it clear that indeed  the sets $\Sigma_C(s)$ and $\Sigma_F(s)$
depend on $n$ and on $f_1, \dots, f_N$. We will use the fact that if 
$f \in H^0(X, \J(n))$ and $\overline{f}_{s}$ is its image in  
$\oline{\J}_s(n)$, then $V_+(f) \cap X_s = V_+(\overline{f}_{s})$.
\end{notation}

\begin{lemma} \label{remove-isolated-pts}  
Let $S$ be an affine noetherian
scheme, and let $X\to S$ be a morphism of finite type. 
Let $F$ be a locally closed subset of $X$. 
Let ${\bf F}$ be the union of the irreducible components of positive
dimension of $F_s$, when $s$ runs over all points of $S$. Then 
${\bf F}$ is closed in $F$.

Assume now that $X\to S$, $n$, $\J(n)$, and $\{f_1,\dots, f_N\}$ are as above in {\rm \ref{emp.notation}}.
Then
\begin{enumerate}[{\rm (1)}]
\item  
$\Sigma_{F}(s)=\Sigma_{\bf{F}}(s)$ for all $s \in S$. 
\item There exists a natural constructible subset $T_F$ 
of $\mathbb A^N_S$ such that for all $s\in S $, $\Sigma_F(s)$ is exactly 
the set of $k(s)$-rational points of $\mathbb A^N_{k(s)}$ 
contained in $(T_F)_s$, 
\end{enumerate}
\end{lemma}

\begin{proof}  Endow $F$ with the structure of a reduced subscheme of
$X$ and consider the induced morphism $g: F \to S$. Then the set of $x \in F$ 
such that $x$ is isolated in $g^{-1}(g(x))$ is open in $F$ (\cite{EGA}, IV.13.1.4).
Thus,  ${\bf F}$ is closed in $F$. 

(1) By construction, for all $s\in S$, 
$F_s$ and ${\bf F}_s$ have the same irreducible components of 
positive dimension, so $\Sigma_F(s)=\Sigma_{\bf F}(s)$ for all $s \in
S$. 

(2) By (1) we can replace $F$ by ${\bf F}$ and suppose that for all
$s\in S$, $F_s$ contains no isolated point. 
Endow $F$ with the structure of a reduced subscheme of $X$. 
Let $\cO_F(1):=\cO_X(1)|_F$. 
Consider the following data: the morphism of finite type $F \to S$, 
the sheaf $ \cO_F(n)$, and the sections $h_1,\dots, h_N$ in $H^0(F,
\cO_F(n))$, with $h_i:= f_i|_F$. 
We associate to this data, for each $s \in S$, the subset $\Sigma(s) $ as in \ref{sig}.
We claim that for each $s \in S$, we have $\Sigma_F(s)=\Sigma(s)$. 
For convenience, recall that 
$$
\begin{array}{c}
\Sigma(s)=\left\{ (\alpha_1, \ldots, \alpha_N)\in k(s)^N \ \Big| \ 
\sum_{i=1}^N \alpha_i h_{i,s} \ \text{\rm vanishes at some generic point of} 
\ F_s\right\}.
\end{array} 
$$
Let $f\in H^0(X, \J(n))\subseteq H^0(X, \cO_X(n))$ 
and let $h=f|_F\in H^0(F, \cO_F(n))$. Recall that $\bar{f}_s$ denotes the image of $f_s$ under the natural map $\J_s(n) \to  \cO_{X_s}(n)$.
Thus, $\bar{f}_s$ is nothing but the image of $f\in H^0(X, \cO_X(n))$ under the natural map $H^0(X, \cO_X(n)) \to H^0(X_s, \cO_{X_s}(n))$.
For any $s\in S$ and for any $x\in F_s$, 
we have 
$$x\in V_+(\bar{f}_s) \Longleftrightarrow \bar{f}_s(x)=0\in \cO_X(n)\otimes k(x)
\Longleftrightarrow h_s(x)=0 \in\cO_F(n)\otimes k(x).$$ 
Since we are assuming that $F_s$ does not have any irreducible
component of dimension $0$, $\Sigma_F(s)$ is equal to 
$$ \begin{array}{c}
 \left\{ (\alpha_1, \ldots, \alpha_N)\in k(s)^N \ \Big| \ 
\sum_{i=1}^N  \alpha_i \overline{f}_{i,s} \ \text{\rm vanishes at some generic point of } F_s \right\}. 
\end{array}
$$
Therefore, $\Sigma_F(s)=\Sigma(s)$. We can thus apply 
Proposition~\ref{constructible-conditions} to the above data 
$F \to S$, $ \cO_F(n)$, and the sections $h_1,\dots, h_N$, 
to obtain 
a natural constructible subset $T_F$ of $\mathbb A^N_S$ 
such that for all $s\in S$, 
$\Sigma_F(s)$ is exactly the set of $k(s)$-rational points 
of $\mathbb A^N_{k(s)}$ contained in $(T_F)_s$.
\end{proof}

\medskip
Our goal  now is to bound the dimension of $(T_F)_s$ 
so that Theorem \ref{globalize}  can be used 
to produce the desired $f  \in H^0(X, \J(n))$. 
Let $V/k$ be a projective variety over a field $k$, endowed with a
very ample invertible sheaf $\cO_V(1)$. Recall that the {\it Hilbert
  polynomial} $P_V(t)\in \mathbb Q[t]$ is the unique polynomial such that 
$P_V(n)=\chi(\cO_V(n))$ for all 
integers $n$ (where 
$\chi(\G)$ denotes as usual the Euler-Poincar\'e characteristic 
of a coherent sheaf $\G$). 
A finiteness result for the Hilbert polynomials of the fibers of a projective
morphism, needed in the final step of the 
proof of our next lemma, is recalled in \ref{Mumford}. 

\begin{lemma} \label{cor.constructible2} 
Let $S=\Spec R$ be an affine noetherian scheme and let $X\to S$ be a projective
morphism. Let $\cO_X(1)$ be a very ample invertible sheaf relative to $X \to S$.
Let $C$ be a closed subscheme of $X$ with ideal sheaf $\J$,
and let $F$ be a locally closed subset of $X$. Assume  that for all $s \in S$, 
no irreducible component of $F_s$ of positive dimension is contained in $C_s$. 

Let $c\in {\mathbb N}$. 
Then there exists $n_0\in {\mathbb N}$ such 
that for all $n\ge n_0$ 
and for any choice $\{f_1,\ldots, f_N\}$ of generators of $H^0(X, \J(n))$,  
the constructible subset $T_F\subseteq \mathbb A^N_S$ introduced
in {\rm \ref{remove-isolated-pts} (2)} satisfies $\dim (T_F)_s \leq N-c$ for all $s \in S$. 
\end{lemma} 

\proof Lemma \ref{remove-isolated-pts} (1) shows that we can suppose that the locally closed subset
$F$ is such that for all $s \in S$, $F_s$ has 
no isolated point. We now further reduce to the case where
$F$ is open and dense in $X$. 

Let $Z$ be the Zariski closure of 
$F$ in $X$. Then $F$ is open and dense in $Z$. 
Endow $Z$ with the induced structure of reduced closed subscheme. 
Denote by $\J\cO_Z$ the 
image of $\J$ under the natural homomorphism $\cO_X \to \cO_Z$. 
This sheaf   is the sheaf of ideals associated with the image of the
closed immersion $C \times_X Z \to Z$. 
The morphism of $\cO_X$-modules $\J \to \J\cO_Z$ is surjective with kernel ${\mathcal K}$. Since $\cO_X(1)$ is very ample,
we find that there exists $n_0>0$ such that for all $n \geq n_0$, 
$H^1(X, {\mathcal K}(n)) =(0)$, so that the natural map
$$H^0(X, \J(n)) \longrightarrow H^0(Z, \J\cO_Z(n))$$ 
is surjective. Fix $n \geq n_0$, and fix a system of generators $\{f_1,\ldots, f_N\}$ of $H^0(X,\J(n))$.
It follows that  the images of $f_1, \dots, f_N$ generate 
the $R$-module $H^0(Z, \J\cO_Z(n))$. Note that 
$C\times_X Z$ does not contain any irreducible component of $F_s$ for
all $s$. It follows that it suffices to prove 
the bound on the dimension of $(T_F)_s$ when $Z=X$, that is, when $F$
is open and dense in $X$. 
We need  the following fact:

\begin{lemma} \label{use.it}
Let $S$ be a noetherian scheme. Let $X\to S$ be a morphism of 
finite type. Then there exist an affine scheme $S'$ and a quasi-finite 
surjective morphism of finite type 
$S' \to S$ with the following properties: 
\begin{enumerate}[{\rm (a)}] 
\item $S'$ is the disjoint union of its irreducible components.
\item Let 
$X':=X\times_S S'$, 
and let $\Gamma_1, \dots, \Gamma_m$ be the irreducible components of $X'$
endowed with the induced structure of reduced closed schemes.  
Then for $i=1,\dots, m$, the fibers of $\Gamma_{i}\to S'$ are either empty or 
geometrically integral.
\item For each $s'\in S'$,
the irreducible components of $X'_{s'}$ are exactly the irreducible components  of the 
non-empty $(\Gamma_i)_{s'}$, $i=1, \dots, m$. 
\end{enumerate}
\end{lemma} 

\begin{proof} We proceed by noetherian induction on $S$. We can suppose $S$ is
reduced
and $X\to S$ is dominant. 
First consider the case $S=\Spec K$ for some 
field $K$. Then there exists a finite 
extension $L/K$ such that each irreducible component of 
$X_L$, endowed with the structure of reduced closed subscheme, 
is geometrically integral 
(see \cite{EGA}, IV.4.5.11 and IV.4.6.6). The lemma is proved with 
$S'=\Spec L$. 

Suppose now that the property is true for any strict closed subscheme $Z$ 
of $S$ and for the scheme of finite type $X\times_S Z\to Z$. If $S$ is reducible
with irreducible components $S_1,\dots, S_\ell$, then by the induction hypothesis 
we can find $S'_i\to S_i$ with the desired properties (a)-(c). Then
it is enough to take $S'$ equal to the disjoint union of the $S'_i$. 
Now we are reduced to the case $S$ is integral. 
Let $\eta$ be the generic point of $S$ and let $K=k(\eta)$. Let $L/K$
be a finite extension defined as in the zero-dimensional case above. 
Restricting $S$ to a dense open subset $V$ if necessary, 
we can find a finite  surjective morphism $\pi: U\to V$ with $U$ integral that extends 
$\Spec L\to \Spec K$. 
Let $X_1, \dots, X_r$ be the (integral) irreducible components of $X\times_V U$. 
Their generic fibers over $U$ are geometrically integral.

It follows from \cite{EGA}, IV.9.7.7, that there exists a dense open
subset $U'$ of $U$ such that $X_i \times_{V} U' \to U'$ has
geometrically integral fibers for all 
$i=1,\dots,r$. Restricting $U'$ further if necessary, we can suppose 
that the number of geometric irreducible components in the fibers of 
$X \times_{V} U' \to U'$ is constant 
(\cite{EGA}, IV.9.7.8). Note now that for each $y \in U'$, the irreducible 
components of $(X\times_V U')_y$ are exactly the fibers $(X_i)_y$, 
$i=1, \dots, r$. 
As $S\setminus \pi(U\setminus U')$ is open and dense in $S$, it contains
a dense affine open subset $V'$ of $S$. 
By induction hypothesis, there exists $T'\to (S\setminus V')_{\mathrm{red}}$ with 
the desired properties (a)-(c). Let $S'$ be the disjoint union 
of $\pi^{-1}(V')$ with $T'$. It is clear that $S'$ satisfies the
properties (a)-(c). 
\end{proof}

Let us now return to the proof of Lemma \ref{cor.constructible2}.
We now proceed to  prove that {\it it suffices to bound the dimension  
of $(T_F)_s$ for all $s\in S$ when all fibers of $X \to S$ are integral}. 
To prove this reduction, we use the fact that the formation of $T_F$
is compatible with any base change $S'\to S$  as in 
\ref{constructible-conditions} (a), and the fact that the dimension of
a fiber $(T_F)_s$ is invariant by finite field extensions in the sense of 
\ref{lem.construct} (b). Finally, the conditions that 
$C_s$ does not contain any irreducible component of $F_s$ is also 
preserved by base change. While making these reductions, 
care will be needed to keep track of the hypothesis that 
$f_1,\dots, f_N$ generate $H^0(X, \J(n))$. 

Let $g: S' \to S$ be as in Lemma \ref{use.it} with natural morphism $g': X \times_S S' \to X$.
Let $F'$ be the pre-image of $F$ in $X\times_S S'$. 
For any $s\in S$ and $s'\in S'$ lying over $s$, 
$\dim (T_F)_s=\dim (T_F)_{s'}$ and $(T_F)_{s'}=(T_{F'})_{s'}$ is the finite union of the 
$(T_{\Gamma_i\cap F'})_{s'}$. 
Increasing $n_0$ if necessary, we find using Fact~\ref{Mumford} (i)
that the natural map
$$H^0(X, \J(n))\otimes \cO(S')\longrightarrow H^0(X', g'^*\J(n))$$ is 
an isomorphism. Denote now by $\J\cO_{\Gamma_i}$ the 
image of $g'^{*}\J$ under the natural map $g'^{*}\J \to \cO_{X'} \to \cO_{\Gamma_i}$.
The morphism $g'^{*}\J \to \J\cO_{\Gamma_i}$ of $\cO_{X'}$-modules is surjective. 
Increasing $n_0 $ further if necessary, we find that 
$$H^0(X', g'^*\J(n)) \longrightarrow H^0(\Gamma_i, \J\cO_{\Gamma_i}(n))$$ is 
surjective for all $i=1,\dots, m$, where the twisting is done with the very ample sheaf
$\cO_{X'}(1):= g'^* \cO_X(1)$ relative to $X' \to S'$. It follows that the images 
of $f_1,\dots, f_N$ in $H^0(\Gamma_i, \J\cO_{\Gamma_i}(n))$ 
also form a system of generators of $H^0(\Gamma_i, \J\cO_{\Gamma_i}(n))$. Therefore, we can replace $X\to S$ with $\Gamma\to S'$ for
$\Gamma$ equal to some $\Gamma_i$. Now we are in the situation where
all fibers of $X\to S$ are integral.

By \ref{constructible-conditions} (b), if $F_s\ne\emptyset$, 
$\dim (T_F)_s$ is the dimension of the kernel of the natural map 
$$k(s)^N \to \cO_{F_s}(n)\otimes k(\xi)=\cO_{X_s}(n)\otimes k(\xi)$$
defined by the $\bar{f}_{i,s}$ and 
where $\xi$ is the generic point of $X_s$. This map is given
by sections in $H^0(X, \J(n))$, so it factorizes into a sequence of linear maps
$$ k(s)^N\to H^0(X, \J(n))\otimes k(s) \to H^0(X_s, \J_s(n))
\to H^0(X_s, \oline{\J}_s(n))\to \cO_{X_s}(n)\otimes k(\xi),$$ 
where the first one is surjective because $f_1,\dots, f_N$ generate
$H^0(X, \J(n))$, the composition of the second and the third 
is surjective
(independently of $s$) by \ref{Jn-barJn} (a) (after increasing $n_0$ if necessary so that \ref{Jn-barJn} (a)
can be applied), 
and the last one is 
injective because $X_s$ is integral. If $F_s=\emptyset$, then
$(T_{F})_s=\emptyset$.  
Therefore, in any case 
$$\dim (T_F)_s \le N - \dim_{k(s)} H^0(X_s, \oline{\J}_s(n)).$$ 

We now end the proof of Lemma \ref{cor.constructible2} by showing that after increasing $n_0$ if necessary,
we have $\dim_{k(s)} H^0(X_s, \oline{\J}_s(n)) \geq c$
for all $s\in S_F:=\{ s\in S \mid F_s\ne\emptyset\}$. We note that for
all $s\in S_F$, $\dim F_s>0$ (since $F_s$ has no isolated point)
so $C_s$ does not contain $F_s$ and,  thus, $C_s\ne X_s$. As $X_s$ is
irreducible, we have $\dim C_{s}<\dim X_{s}$. It follows that 
the Hilbert  polynomial  $P_{C_{s}}(t)$ of $C_{s}$ satisfies 
$\deg P_{C_{s}}(t)<\deg P_{X_{s}}(t)$. Since the set of all Hilbert 
polynomials $P_{X_{s}}(t)$ and $  P_{C_{s}}(t)$ 
with $s\in S$ 
is finite (\ref{Mumford} (iii)), and since such polynomials have positive 
leading coefficient (\cite{Har}, III.9.10), we can assume, increasing $n_0$ if necessary, that 
$$ P_{X_{s}}(n)-P_{C_{s}}(n)\ge c$$ 
for all $s\in S_F$. Using \ref{Mumford} (ii), and increasing $n_0$ further if necessary,  we find that 
$$H^i(X_s, \cO_{X_{s}}(n)) =(0)
=H^i(C_{s}, \cO_{C_{s}}(n))$$
for all $i\ge 1$ and for all $s\in S$. 
We have
$P_{X_{s}}(n) = \chi(\cO_{X_{s}}(n))$, and $P_{C_{s}}(n) = \chi(\cO_{C_{s}}(n))$ for all $n \geq 1$.
Therefore, using the above vanishings for $i>0$, we find that
for all $s\in S$, 
$$P_{X_s}(n)-P_{C_s}(n)= \dim H^0(X_{s}, \cO_{X_s}(n))-
\dim H^0(C_{s}, \cO_{C_s}(n)).$$
Hence, for all $s \in S_F$,
$$\dim H^0(X_{s}, \oline{\J}_s(n)) \geq \dim H^0(X_{s}, \cO_{X_{s}}(n))
-\dim H^0(C_{s}, \cO_{C_{s}}(n))\ge c$$ 
and the lemma is proved. 
\qed 

\medskip 
Assume now that $C \to S$ is  as in Theorem \ref{pro.reductiondimension2}, and let $\J$ denote the ideal sheaf of $C$ in $X$, 
as in \ref{emp.notation}.
In particular, $C \to S$ is finite, $C \to X$ is a regular immersion, $C$ has pure codimension $d$ in $X$, 
and for all $s \in S$, $\codim(C_s,X_s) \geq d$.
This latter hypothesis and \ref{lem.remcodim} (1.b) imply that
$C_s$ does not contain any isolated point of $X_s$. 
Therefore, for any $x\in C_s$, 
$(\oline{\J}_s)_x\ne 0$ and, hence, both $\oline{\J}_s(n)/\oline{\J}_s^2(n)
\otimes k(x)$ and $\J_s(n)/\J_s^2(n) \otimes k(x) $, are non-zero.  
In fact, as $C\to X$ is a regular immersion, 
$\J(n)/\J^2(n)$ is a rank $d$ vector bundle on $C$. 

\begin{lemma} \label{cor.constructible} 
Assume that $C \to S$ is as in Theorem {\rm \ref{pro.reductiondimension2}}, with $C$ of pure codimension $d$ in $X$  
and suppose $X\to S$ is projective. 
Keep the notation in {\rm \ref{emp.notation}}.
Let $n_0>0$ be an integer  such that for all $n \geq n_0$,  $\J(n)$ is generated by its global sections, and $H^1(X,\J^2(n)) =(0)$.
Then for all $n \geq n_0$, and for any system of generators $f_1, \dots, f_N$ of $H^0(X, \J(n))$,
there exists  a constructible subset  $T_C $ of $\mathbb A^N_S$ 
such that 
\begin{enumerate}[{\rm (i)}]
\item for all $s\in S$, $\Sigma_C(s)$ is exactly the set of 
$k(s)$-rational points of $\mathbb A^N_{k(s)}$ contained in $(T_C)_s$, and  
\item $\dim (T_C)_s \le N-d$.
\end{enumerate} 
\end{lemma} 

\proof Consider the data consisting of the morphism $C\to S$, the 
sheaf $\F:=\J(n)|_C$ on $C$, and the images of $f_1,\dots, f_N$ under the
natural map $H^0(X,\J(n) ) \to H^0(C, \F)$. To this data is associated 
in \ref{sig} a set $\Sigma(s)$ for any $s \in S$. As $C_s$ is finite, 
$\Sigma_C(s)$ is nothing but the set $\Sigma(s)$.
We thus apply Proposition \ref{constructible-conditions} to the above 
data to obtain a constructible set $T_C$ of $\mathbb A^N_S$ 
such that for all $s\in S$, 
$\Sigma_C(s)$ is exactly the set of $k(s)$-rational points 
of $\mathbb A^N_{k(s)}$ contained in $(T_C)_s$.

Our additional hypothesis implies that the images of the sections 
$f_1,\dots, f_N$ generate $H^0(X,\J(n)/\J^2(n))$, which we identify with $ H^0(C, \F)$. 
Since $C \to S$ and $S$ are affine, and $C_s$ is finite 
for each $s \in S$, we have isomorphisms 
$$H^0(C, \F) \otimes k(s) \simeq 
H^0(C_s,\F_s) \simeq \oplus_{x \in C_s} (\F_s)_x.$$
It follows that for each $x \in C_s$, the natural map 
$H^0(C_s,\F_s) \to (\F_s)_x$ is surjective.
As $(\F_s)_x$ is free of rank $d$ and the images of $f_1, \dots, f_N$
generate $H^0(C_s, \F_s)$, we find that the linear maps $k(s)^N \to \F(x)$ in  
\ref{constructible-conditions} (b) are surjective for all $s \in
S$. It follows that $\dim (T_C)_s \le N-d$ (the equality holds if
$C_s\ne\emptyset$). 
\qed

\begin{emp} \label{Proofreductiondimension2}
{\it Proof of Theorem {\rm \ref{exist-hyp}} when $\pi: X \to S$ is projective. } 
Let $\{F_1, \dots, F_m\}$ be the   locally closed 
subsets of $X$  given in (ii) of the statement of the theorem.
When $C$ does not contain any irreducible component  of
positive dimension of $X_s$  for all $s \in S$, we set $F_0:=X$ and argue below using the set 
$\{F_0, F_1, \dots, F_m\}$.
Let $A$ denote the finite set given in (iii). 
When  $C\cap \Ass(X)=\emptyset$, we 
enlarge $A$ if necessary by adjoining to it the finite set $\Ass(X)$. 

Let $A_0\subset X$ be the union of $A$ with the set of the generic points
of the irreducible components of positive dimension  of $(F_{1})_s, \dots, (F_{m})_s$, for all $s\in S$. 
When relevant, we also add to $A_0$  the generic points
of the irreducible components of positive dimension  of $(F_{0})_s$, for all $s\in S$.
Using \cite{EGA}, IV.9.7.8, we see that the number of points in 
$A_0\cap X_s$ is bounded when $s$ varies in $S$. 
We are thus in a position to apply Lemma~\ref{existence-fn} (a) with the set $A_0$.
Let $c:=1+\dim S$. 
Let $n_0$ be an integer satisfying simultaneously 
the conclusion of Lemma~\ref{existence-fn} (a) for $A_0$, and of Lemma~\ref{cor.constructible2} for $c$ and for each locally closed subset $F=F_i$, with $i=1,\dots, m$, and $i=0$ when relevant.

Fix now $n\ge n_0$, and fix  $f_1, \dots, f_N$ a system of generators
of $H^0(X, \J(n))$.
Increasing $n_0$ if necessary, we can assume using 
Lemma \ref{Jn-barJn} that for  all $s \in S$, the composition of the canonical maps 
$$ 
H^0(X, \J(n)) \otimes k(s) \longrightarrow H^0(X_s, \J_s(n)) 
\longrightarrow H^0(X_s, \oline{\J}_s(n))
$$ 
is surjective. 
Let $T_{F_1}, \dots, T_{F_m}$ be the constructible 
subsets of ${\mathbb A}_S^N$ pertaining to $\Sigma_{F_1}(s), \dots, 
\Sigma_{F_m}(s)$ and whose existence is proved in 
\ref{remove-isolated-pts} (2). When relevant, we also consider  $T_{F_0}$ and $\Sigma_{F_0}(s)$.
Since Lemma~\ref{cor.constructible2} is applicable 
for $c$ and for each $F=F_i$, we find that  for all $s\in S$ we have 
$$\dim (T_{F_i})_s\le N-c=N-1-\dim S.$$
It follows from \ref{lem.construct} (c) that $\dim T_{F_i}\le N-1$.

Let $\pi(A) := \{s_1, \dots, s_r\} \subseteq S$. Fix $s_j \in \pi(A)$, and 
for each $x\in A\cap X_{s_j}$, consider the 
hyperplane of $\mathbb A^N_{k(s_j)}$ defined by 
$\sum_i \alpha_i f_i(x)=0$. This is
indeed a hyperplane because otherwise $f_i(x)=0$ for all $i\le N$ at  
$x$, which would imply that $x\in C$, but 
$A\cap C=\emptyset$ by hypothesis. 
Denote by $T_{A_j}$   the finite union of
all such hyperplanes of $\mathbb A^N_{k(s_j)}$, for each $x\in A\cap X_{s_j}$.
The subset $T_{A_j}$ is 
pro-constructible in $\mathbb A^N_S$ (see \ref{pro-constr}).
It has dimension $N-1$, and its fibers $(T_{A_j})_s$ are constructible for each $s \in S$ (and $(T_{A_j})_s$ is empty if $s \neq s_j$).

We now apply Theorem \ref{globalize} to the set of 
pro-constructible subsets $T_{A_j}$, $j=1,\dots, r$ and $T_{F_i}$, $i=1,\dots, m$, and $i=0$ when relevant.
Our discussion so far implies that these pro-constructible subsets all satisfy Condition (1) in 
\ref{globalize} with $V=S$. Let $T=(\cup_j T_{A_j})\cup (\cup_i
T_{F_i})$. For each $s \in S$,  
the element $f_{s,n} \in H^0(X_s,\J_s(n))$
exhibited in Lemma~\ref{existence-fn} (a) gives rise to  a $k(s)$-rational 
point of $\mathbb A_{k(s)}^N$ not contained in $T_s$. So Condition
 (2) in \ref{globalize} is also satisfied by $T$. 
We can thus apply Theorem \ref{globalize} to find a  
section $(a_1,\ldots, a_N)\in R^N=\mathbb A_S^N(S)$ such that for all $s \in S$,
$(a_1(s), \ldots, a_N(s))$ is a $k(s)$-rational point of $\mathbb A^N_{k(s)}$ 
that is not contained in $T_s$. 

Let $f:=\sum_{i=1}^N a_if_i$ and consider the closed subscheme
$H_f \subset X$. As $f\in H^0(X, \J(n))$, $C$ is a closed
subscheme of $H_f$. By definition of $T_{F_i}$ and $T_{A_j}$, 
for all $s\in S$ and for all $0 \leq i\le m$, $H_f$ does not contain any 
irreducible component of $(F_i)_s$ of positive dimension and 
$H_f\cap A=\emptyset$. This proves the conclusions (1), (2), and (3) of \ref{exist-hyp}.

When the hypothesis of (4) is satisfied, we have included in our proof above conditions
pertaining to $F_0=X$, and we find then that  $H_f$ contains no irreducible 
component of $X_s$. It  is thus by definition a hypersurface relative to $X \to S$. If furthermore
$C\cap \Ass(X)=\emptyset$, as we enlarged $A$ to include $\Ass(X)$, 
we have $H_f\cap \Ass(X)=\emptyset$.  Hence, it follows from Lemma~\ref{hypersurfaces-properties} (4)
that $H_f$ is 
locally principal.  
This proves (4), and completes the proof of Theorem \ref{exist-hyp} when 
$X\to S$ is projective.
\qed 
\end{emp}

\begin{emp} \label{Proofreductiondimension2-a} 
{\it Proof of Theorem {\rm \ref{pro.reductiondimension2}} when $\pi: X \to S$ is projective. }
We assume now that $C \to S$ is finite. Thus $C_s$ is finite for each $s \in S$, and we find 
that 
$C_s$ does not contain any irreducible component of positive dimension of $X_s$. 
 Let $\{F_1, \dots, F_m\}$ be the   locally closed 
subsets of $X$  given in (ii) of \ref{exist-hyp}.
We set $F_0:=X$ and argue as in the proof of \ref{exist-hyp} above using the set 
$\{F_0, F_1, \dots, F_m\}$.
Let $A$ denote the finite set given in (iii) of \ref{exist-hyp}. 
We have that $C\cap\Ass(X)=\emptyset$: indeed, for all $x\in C$, 
 $\mathrm{depth}(\cO_{X,x})\ge d>0$, so that $x\notin \mathrm{Ass}(X)$.
We therefore enlarge $A$ if necessary by adjoining to it the finite set $\Ass(X)$. 
We define $A_0$ and $c:=1+\dim S$ exactly as in the proof of \ref{exist-hyp} in \ref{Proofreductiondimension2}.
Let $n_0$ be an integer satisfying simultaneously 
the conclusion of Lemma~\ref{existence-fn} (b) for $A_0$, of Lemma~\ref{cor.constructible}, and of Lemma~\ref{cor.constructible2} for $c$ and for each locally closed subset $F=F_i$, with $i=0,1,\dots, m$.

Fix now $n\ge n_0$, and fix  $f_1, \dots, f_N$ a system of generators
of $H^0(X, \J(n))$.
Increasing $n_0$ if necessary, we can assume using 
Lemma \ref{Jn-barJn} that for  all $s \in S$, the composition of the canonical maps 
$$ 
H^0(X, \J(n)) \otimes k(s) \longrightarrow H^0(X_s, \J_s(n)) 
\longrightarrow H^0(X_s, \oline{\J}_s(n))
$$ 
is surjective. 
Let $T_{F_0}, T_{F_1}, \dots, T_{F_m}$ be the constructible 
subsets of ${\mathbb A}_S^N$ pertaining to $\Sigma_{F_0}(s)$, $\Sigma_{F_1}(s), \dots$, 
$\Sigma_{F_m}(s)$, and whose existence is proved in 
\ref{remove-isolated-pts} (2). 
As in the proof \ref{Proofreductiondimension2}, we find that $\dim T_{F_i}\le N-1$ for each $i=0,\dots, m$.
Define now  $T_{A_j}$, $j=1,\dots, r$ as in the proof
\ref{Proofreductiondimension2}. Again,  $T_{A_j}$ is 
pro-constructible in $\mathbb A^N_S$,
it has dimension $N-1$, and its fibers $(T_{A_j})_s$ are constructible for each $s \in S$. 

Since Lemma \ref{cor.constructible} is applicable, we can also consider  the constructible
subset $T_C$  of $\mathbb A^N_S$ 
pertaining to $\Sigma_C(s)$. Since we assume that $d > \dim S$, we find from \ref{cor.constructible}
that $$ \dim (T_C)_s\le N-d\le N-(\dim S+1) $$ 
for all $s\in S$. Thus, it follows from  \ref{lem.construct} (c) that
$\dim T_C\leq N-1$. 

As in \ref{Proofreductiondimension2}, we set $T$ to be the union of the sets $T_{F_i}$, $i=0,\dots, m$,
and $T_{A_j}$, $j=1,\dots, r$. Lemma~\ref{existence-fn} (b)
implies that $\mathbb A^N_{k(s)}$ is not contained in 
$(T_C\cup T)_s$ because $(\oline{\J}_s(n)/\oline{\J}^2_s(n))\otimes
k(x)\ne 0$ for all $x\in C_s$ (see the paragraph before \ref{cor.constructible}). 
Applying Theorem~\ref{globalize}
to the pro-constructible subsets $T_C$, $T_{F_i}$, $i=0,\dots, m$, and $ T_{A_j}$, $j=1,\dots, r$, we find 
$(a_1, \dots, a_N)\in \mathbb A^N_S(S)$ such that 
for all $s \in S$,
$(a_1(s), \ldots, a_N(s))$ is  a $k(s)$-rational point of $\mathbb A^N_{k(s)}$ 
that is not contained in $(T_C\cup T)_s$.  

Let $f:=\sum_{i=1}^N a_if_i$ and consider the closed subscheme
$H_f \subset X$. As in \ref{Proofreductiondimension2}, we find that  $H_f$ satisfies the conclusions
(1), (2), and (3)  of
\ref{exist-hyp},  and that $H_f$ is a locally principal hypersurface. 

It remains to use the properties of the set   $T_C$ to show that $C$ is regularly immersed in 
$H_f$, and that $C$ is pure of codimension $d-1$ in $H_f$. 
Indeed, this is a local question.  
Fix $x \in C$. Let $I:= \J_x \subset \cO_{X,x}$ and let $g \in I$
correspond to the section $f$. 
Since the image of $g$ in $I/I^2 \otimes k(x)$ is non-zero by the
definition of $T_C$, the image 
of $g$ in the free $\cO_{C,x}$-module $I/I^2$ can be completed into 
a basis of $I/I^2$, and
it is then  well-known that $g$ belongs to a regular sequence
generating $I$. 
This concludes the proof of Theorem \ref{pro.reductiondimension2} when 
$X\to S$ is projective.  \qed
\end{emp} 

\begin{emp} \label{quasiproj}
{\it Proof of Theorems {\rm \ref{exist-hyp}} and {\rm \ref{pro.reductiondimension2}} when 
$X\to S$ is quasi-projective. }
Since $\cO_X(1)$ is assumed to be very ample relative to $X \to S$, there exists
a projective morphism 
$\overline{X} \to S$ with an open immersion $X \to \overline{X}$ of 
$S$-schemes, and a very ample sheaf $\cO_{\overline{X}}(1)$ relative to $\overline{X} \to S$ 
which restricts on $X$ to the given sheaf $\cO_X(1)$.

Let us first prove Theorem {\rm \ref{exist-hyp}}. We are given in {\rm \ref{exist-hyp}} (i) a closed subscheme $C$ of $X$. 
Let $\oline{C}$ be the scheme-theoretical closure of $C$ in
$\oline{X}$. We are given in {\rm \ref{exist-hyp}} (ii)
$m$ locally closed subsets $F_1, \dots, F_m$ of $X$.
Since $X$ is open in $\oline{X}$, each set $F_i$ is again locally closed in 
$\oline{X}$.  It is clear that the finite subset $A \subset X$ given in {\rm \ref{exist-hyp}} (iii) which does not intersect
$C$ is such that $A \subset \oline{X}$ does not intersect $\oline{C}$.
 
 We are thus in the position to apply Theorem \ref{exist-hyp}
 to the projective morphism $\overline{X} \to S$
 with the data $ \oline{C}$, $ F_1, \dots, F_m$, and $ A$.
When  $C$ satisfies the first
hypothesis of \ref{exist-hyp} (4), we set $F_0 :=X$ and add the locally closed subset $F_0$
to the list  $F_1,\dots, F_m$, as in the proof \ref{Proofreductiondimension2}.
When $C\cap \Ass(X)=\emptyset$, we replace $A$ 
by  $A \cup \Ass(X)$. 
We can then conclude that there exists $n_0>0$ such that
for any $n\ge n_0$, there exists a global section $f $ of $ \cO_{\overline{X}}(n)$ such that the  closed subscheme $H_f$ in $\oline{X}$ contains 
$\oline{C}$ as a closed subscheme and satisfies the conclusions (2), (3), and, when relevant, (4), of \ref{exist-hyp} for $\overline{X} \to S$.
 The restriction of $f$ to $\cO_X(n)$ defines the desired closed subscheme $H_f \cap X$ 
satisfying the conclusions of Theorem \ref{exist-hyp} for $X \to S$.

Let us now prove Theorem \ref{pro.reductiondimension2}, where
we assume that $C \to S$ is finite and, hence, proper. It follows that 
$\oline{C}=C$. 
We apply Theorem  \ref{pro.reductiondimension2} to the projective morphism $\oline{X} \to S$, 
and the data $C$, $F_1,\dots, F_m$, and $A$. 
We can then conclude that there exists $n_0>0$ such that
for any $n\ge n_0$, there exists a global section $f $ of $ \cO_{\overline{X}}(n)$ such that the  closed subscheme $H_f$ in $\oline{X}$ contains 
$\oline{C}$ as a closed subscheme and satisfies the conclusions (2), (3), and (4), of \ref{exist-hyp} for $\overline{X} \to S$.
The restriction of $f$ to $\cO_X(n)$ defines the desired closed subscheme $H_f \cap X$ 
satisfying the conclusions of Theorem \ref{pro.reductiondimension2} for $X \to S$.
\qed
\end{emp}

\end{section}

\begin{section}{Variations on the classical Avoidance Lemma}
\label{compute-coh}

In this section, we prove various assertions used in the 
proofs of Theorem~\ref{exist-hyp} and Theorem~\ref{pro.reductiondimension2}. 
The main result in this section is  Lemma \ref{existence-fn}.

\begin{facts} \label{Mumford}
Let $S$ be a noetherian scheme, and let $\pi: X \to S$ be a projective morphism. Let $\cO_X(1)$ be a very ample sheaf relative to $\pi$, 
and let ${\mathcal F}$ be any coherent sheaf on $X$.
\begin{enumerate}[\rm (i)]
\item Let $g:S'\to S$ be  a morphism of finite type, and consider the cartesian square
$$
\xymatrix{
X':=X \times_S S'    \ar@{>}[r]^{\hspace*{12mm} g'}  \ar[d]^{\pi'}  & X \ar[d]^{\pi} 
\\
S'  \ar@{>}[r]^{g} & S. \\  
}
$$ 
Then there exists a positive integer $n_0$ such that for all $n \geq n_0$, 
the canonical morphism $g^*\pi_*({\mathcal F}(n)) \longrightarrow \pi'_{*}g'^{*}({\mathcal F}(n))$ is an isomorphism.
\item There exists a positive integer $n_0$ such that for all $n \geq n_0$
and for all $s \in S$, $H^i(X_s, {\mathcal F}_s(n))=(0)$ for all $i>0$,  
and $\pi_*\F(n) \otimes k(s) \longrightarrow  H^0(X_s, {\mathcal F}_s(n))$ is an isomorphism.
\item The set of Hilbert polynomials $\{ P_{X_s}(t)\in \mathbb Q[t] \mid s
  \in S\}$ is finite. 
\end{enumerate}
\end{facts}

\proof The properties in the statements are local on the base, and we may thus assume that $S$ is affine.
In this case, there is no ambiguity in the definition of a projective morphism, as all standard definitions coincide when the target is affine
(\cite{EGA}, II.5.5.4 (ii)).
The proofs of (i) and (ii) when $X= {\mathbb P}^d_S$ and any coherent
sheaf $\F$ can be found, for instance, 
in \cite{MumL}, p.\ 50, (i), and  \cite{MumL}, p.\ 58, (i) (see also 
  \cite{Sernesi}, step 3 in the proof of Theorem 4.2.11). The
  statement (iii) follows from \cite{MumL}, p.\ 58, (ii). The general 
case follows immediately using the 
closed $S$-immersion $i: X \to {\mathbb P}^d_S$ defining $\cO_X(1)$. 
\qed

\begin{emp} \label{m-regular} Let us recall the definition and 
properties of $m$-regular sheaves needed in our next lemmas.
Let $X$ be a projective variety over a field $k$, 
with a fixed very ample sheaf $\cO_X(1)$. 
Let $\F$ be a coherent sheaf on $X$, and let $\F(n) := \F \otimes \cO_X(n)$.
Let $m\in \mathbb Z$. Recall (\cite{MumL}, Lecture 14, p. 99) that  $\F$ is called
\emph{$m$-regular} if $H^i(X, \F(m-i))=0$ for all $i\ge 1$. 

Assume that $\F$ is 
\emph{$m$}-regular. Then it is known (see, e.g., \cite{Sernesi}, Proposition 4.1.1)
that for all $n\ge m$, 
\begin{enumerate}[\rm (a)]
\item
$\F$ is $n$-regular, 
\item $H^i(X, \F(n))=0$ for all 
$i\ge 1$, 
\item $\F(n)$ is generated by its global sections, 
and 
\item The canonical homomorphism
$$ 
H^0(X, \F(n))\otimes H^0(X, \cO_X(1)) \longrightarrow H^0(X, \F(n+1))
$$ 
is surjective. 
\end{enumerate} 
\end{emp}

\begin{lemma} \label{Mumford2}
Let $S$ be a noetherian scheme, and let $\pi: X \to S$ be a projective 
morphism. Let $\cO_X(1)$ be a very ample sheaf relative to $\pi$, 
and let ${\mathcal F}$ be any coherent sheaf on $X$. 
Then there exists a positive integer $n_0$ such that for all $n \geq n_0$ and all $s\in S$, the sheaf ${\mathcal F}_s$ is $n$-regular on $X_s$.
\end{lemma}

\begin{proof} Let $r$ denote the maximum of $\dim X_s$, $s \in S$. This maximum 
is finite (\cite{EGA}, IV.13.1.7). Then 
$H^i(X_s,  {\mathcal F}_s(n))=(0)$ for all $i \geq r+1 $ and for any $n$. 
Using \ref{Mumford} (ii), there exists $n_1>0$ such that 
$H^i(X_s, {\mathcal F}_s(n))=(0)$ for all $s \in S$, for all $n \geq n_1$, 
and for all $i>0$. It follows that 
$ {\mathcal F}_s $ is $n$-regular for all $s \in S$ and for all 
$n\geq n_0:=r+n_1$. 
\end{proof}

We now discuss a series of lemmas needed in the proof of \ref{existence-fn}.
\begin{lemma} \label{Jn-barJn} 
Let $\pi : X\to S$ be a projective scheme over a noetherian scheme $S$. Let
$\cO_X(1)$ be a very ample sheaf relative to $\pi$. 
Let $C$ be a closed subscheme of $X$, with  
 sheaf of ideals $\J$  in $\cO_X$. 
Let $\oline{\J}_s $ denote the image of  $\J_s$ in $\cO_{X_s} $.
Then there exists $n_0 \in {\mathbb N}$ such that for all $n\ge n_0$ and for
all $s\in S$, 
\begin{enumerate}[\rm (a)]
\item The canonical map 
$\pi_*\J(n)\otimes k(s)\longrightarrow H^0(X_s, \oline{\J}_s(n))$ is 
surjective. 
\item The sheaf $\oline{\J}_s$ is $n$-regular.
\end{enumerate}
\end{lemma}

\proof 
By the generic flatness theorem \cite{EGA}, IV.6.9.3,  
there exist finitely many 
locally closed subsets $U_i$ of $S$ 
such that $S = \cup_i U_i$ (set theoretically), and such that when each 
$U_i$ is endowed with the structure of reduced subscheme of $S$, then 
$C_{U_i}:=C\times_S U_i\to U_i$ is flat. Refining each $U_i$ by an affine covering, we 
can suppose $U_i$ affine. 

Denote by $U$ one of these affine schemes $U_i$. 
Denote by $\mathcal K$ and $\J'$ the kernel and image
of the natural morphism $\J \otimes_{\cO_X} \cO_{X_U}  \to \cO_{X_{U}}$, with associated exact sequence of sheaves on $X_{U}$ 
$$ 0 \longrightarrow \mathcal K \longrightarrow \J\otimes_{\cO_X} \cO_{X_U} \longrightarrow \J'
\longrightarrow 0.$$
For all $n\in \mathbb Z$, we then have the exact sequence 
$$ 0 \to \mathcal K(n) \to \J(n)\otimes_{\cO_X} \cO_{X_U} \to 
\J'(n)\to 0.$$
Since $X_{U} \to U$ is projective, we can find $n_1$ such that 
$H^1(X_{U}, \mathcal K(n))=(0) $ for all $n \geq n_1$ (Serre Vanishing). 
Using \ref{Mumford} (ii), we find that by increasing $n_1$ if necessary,
we can assume that 
for all $n\ge n_1$ and for all $s\in U$, 
\begin{equation} \label{Jns}
H^0(X_{U}, \J'(n))\otimes k(s)\longrightarrow H^0(X_s, \J'(n)_s)
\end{equation}
is an isomorphism. The exact sequence 
$0\to \J'  \to \cO_{X_U} \to \cO_{C_U} \to 0$ induces 
an exact sequence $0\to \J'_s \to \cO_{X_s} \to \cO_{C_s} \to 0$ 
for all $s \in U$ because $C\times_S U\to U$ is flat. 
It follows that $\J'_s= \oline{\J}_s$. 
We can thus apply \ref{Mumford2} to the morphism $X_{U} \to U$ and the sheaf $\J'$ to obtain
that  $\oline{\J}_s$ is $n$-regular for all $n \geq n_1$  and all $s \in U$(after increasing $n_1$ further if necessary.)

For any $s \in U$ and for $n \geq n_1$, consider the commutative diagram: 
$$
\xymatrix{
H^0(X_{U}, \J(n)\otimes\cO_{U})\otimes k(s) 
\ar@{>>}[r] \ar[d]  \ar[dr] & H^0(X_{U}, \J'(n))\otimes k(s) \ar[d] \\
H^0(X_s, \J(n)_s)  \ar[r] & H^0(X_s, \oline{\J}_s(n)). \\  
}
$$
The top horizontal map is surjective because $H^1(X_{U}, \mathcal K(n))=(0)$, 
and the right vertical arrow is an isomorphism by the isomorphism \eqref{Jns}
above. Thus, 
the bottom arrow 
$H^0(X_s, \J(n)_s)  \to H^0(X_s, \oline{\J}_s(n))$ is surjective 
for all $n\ge n_1$ and all $s\in U$.

To complete the proof of  (b), 
it suffices to choose $n_0$ to be the maximum in the set of integers $n_1$ 
associated with each $U_i $ in the stratification.  To complete the proof of (a), 
we further increase $n_0$ if necessary to be able to  use the isomorphism in 
\ref{Mumford}\ (ii) applied to ${\mathcal F}=\J$ on $X\to S$. 
\qed

\begin{lemma} \label{fn-bis} 
Let $X$ be a projective variety over a field $k$ with a fixed very 
ample sheaf $\cO_X(1)$. Let $C$ be a closed subscheme of $X$. 
 Let  $\J$ denote the ideal sheaf of $C$ in $X$, 
and assume that $\J$ is $m_0$-regular for some $m_0 \geq 0$. 
Let $D$ be a finite set of closed points of $C$. Let $\{\xi_1, \dots, \xi_r\}$ be a finite subset of $X$ disjoint
from $C$.
\begin{enumerate}[{\rm (a)}] 
\item 
If ${\mathrm{Card}}(k)\geq r+{\mathrm{Card}}(D)$, then
for all $n\ge m_0$, there exists a section $f_n\in H^0(X, \J(n))$ 
such that $V_+(f_n)$ does not contain any $\xi_i$, 
and such that, for all $x\in D$ such that 
$(\J(n)/\J^2(n))\otimes k(x) \neq (0)$,
the image of $f_n$ in 
$(\J(n)/\J^2(n))\otimes k(x)$
is non-zero.
\item  
There exists an integer $n_0>0$ such that for all 
$n\ge n_0$, there exists a section $f_n\in H^0(X, \J(n))$ as in {\rm (a)}. 
\end{enumerate}
\end{lemma}

\proof 
It suffices to prove the lemma for the subset
of $D$ obtained by removing
from $D$ all points $x$ such that 
$(\J(n)/\J(n)^2)\otimes k(x)=0$. 
We thus suppose now that $(\J(n)/\J(n)^2)\otimes k(x)\ne 0$
for all $x\in D$. Note also that the natural map $(\J(n)/\J(n)^2)\otimes k(x) \to \J(n)\otimes k(x)$
is an isomorphism for all $x\in D$, and we will use the latter expression.

(a)  Let $x \in D$ and $n \geq m_0$. Consider the $k$-linear map 
$$H^0(X, \J(n))\longrightarrow \J(n)\otimes k(x)$$ 
and denote by $H_x$ its kernel. 
Since $\J(n)$ is generated
by its global sections (\ref{m-regular} (c)), this map is non-zero
and   $H_x \neq H^0(X, \J(n))$.

Let $B= \oplus_{j \geq 0} H^0(X, \cO_X(j))$. This is a 
graded $k$-algebra and $X\simeq \Proj B$. 
Let $\p_1,\dots, \p_r$ be the homogeneous prime ideals of $B$ defining 
$\xi_1, \dots, \xi_r$.
Let $J$ be the homogeneous ideal $\oplus_{j \geq 0} H^0(X, \J(j))$ of $B$.
Then $C$ is the closed subscheme of $X$ defined by $J$. By hypothesis,
for each $i \leq r$, $\p_i$  neither contains $J$ nor $B(1)$. 
Let $J(n):=H^0(X, \J(n))$. 
We claim that for each $i \leq r$, 
$J(n)\cap \p_i$ is a proper 
subspace of $J(n)$.
Indeed, 
if $J(n_0)\cap \p_i = J(n_0)$ for some $n_0 \geq m_0$, then 
the surjectivity of the map in \ref{m-regular} (d) implies that 
$J(n)\cap \p_i = J(n)$ for all $n \geq n_0$. 
This would imply $C=V_+(J)\supseteq V_+(\p_i)\ni \xi_i$.

We have  constructed above at most $r+{\mathrm{Card}}(D)$ 
proper 
subspaces of $H^0(X, \J(n))$.  
Since $r+{\mathrm{Card}}(D) \leq {\mathrm{Card}}(k)$ by hypothesis, the union 
of these proper subspaces is not equal to $H^0(X,\J(n))$ (\ref{union}).
Since any element $f_n$ in the 
complement of the union of these 
subspaces satisfies the desired properties, (a) follows.
 
(b)  Let $\J_D$ be the ideal sheaf on $X$ defining the structure of
reduced closed subscheme on $D$. 
Choose $m \geq 0$ large enough such that both $\J$ and 
$ \J \J_D$ are 
$m$-regular. 
As $H^1(X, (\J \J_D)(n))=(0)$ for $n \geq m$ by \ref{m-regular} (b), the map
$$H^0(X, \J(n))\longrightarrow H^0(X, \J(n){|_D})
=H^0(X, \J(n)/\J \J_D(n)))$$
is surjective for all $n\ge m$. Note now the isomorphisms
$$H^0(X, \J(n){|_D})
\longrightarrow \oplus_{x\in D} (\J(n){|_D})_x \longrightarrow \oplus_{x\in D} \J(n)\otimes k(x).$$ 
Let then $f\in H^0(X, \J(n))$ be a section such that
its image in 
$\J(n)\otimes k(x)$
is non-zero 
for each $x \in D$. 
Keep the notation introduced in (a). Then $I:=\oplus_{n\ge 0} H^0(X, \J^2(n))$
is a homogeneous ideal of $B$ and $J^2\subseteq I\subseteq J$. 
Hence $I\not\subseteq \p_i$ for all $i\le r$, 
since otherwise $J\subseteq \p_i$, which contradicts 
the hypothesis that $\xi_i\notin C$.

Lemma \ref{avoid} (a) 
below implies then the existence of $n_0 \geq 0$ such that for all $n \geq n_0$, 
there exists $x_n\in I(n)$ 
such that  $f_n:=f+x_n\notin\cup_{1\le i\le r}\p_i$. 
We have $f_n\in J(n)=H^0(X, \J(n))$ and 
for all $x\in D$, 
$f_n$ is non-zero in 
$\J(n)\otimes k(x)$.
\qed

\medskip
The following Prime Avoidance Lemma for graded rings is needed in the proof of \ref{fn-bis}.
This lemma is slightly stronger than 4.11 in \cite{GLL1}. For related statements, see \cite{SH}, Theorem A.1.2., or \cite{Bou}, III, 1.4, Prop. 8, page 161.
We do not use the statement \ref{avoid} (b) in this article.
\begin{lemma}\label{avoid} 
Let $B=\oplus_{n\ge 0}B(n)$ be a graded ring. 
Let $I=\oplus_{n\ge 0}I(n)$ be a homogeneous ideal of $B$.
Let $\p_1,\dots,\p_r$ be homogeneous prime ideals of $B$ not containing $B(1)$
and not containing $I$. 
\begin{enumerate}[{\rm (a)}]
\item Then there exists an integer $n_0 \geq 0$ such
that for all $n\ge n_0$ and for all $f\in B(n)$, we have 
$$f+I(n)\not\subseteq \cup_{1\le i\le r} \p_i.$$    
\item Let $k$ be a field with $\mathrm{Card}(k)>r$, and assume that $B$ is a $k$-algebra. 
If $I$ can  be generated by elements of degree at most $ d$, then 
in {\rm (a)} we can take $n_0=d$. 
\end{enumerate} 
\end{lemma}

\proof  We can suppose
that there are no inclusion relations between $\p_1, \dots, \p_r$. 
 
(a) Let $i\le r$ and set $I_i:=I\cap (\cap_{j\ne i}\p_j)$. 
We first observe that there exists $n_i\ge 0$ such that for 
all $n \ge n_i$, we have $I_i(n)\not\subseteq \p_i$. Indeed, 
as $I_i\not\subseteq \p_i$ and $I_i$ is homogeneous, we can find a 
homogeneous element $\alpha$ in $ I_i\setminus \p_i$. Let $t\in B(1)\setminus \p_i$. 
Set $n_i:=\deg \alpha$. Then for all $n\ge n_i$, we have 
$t^{n-n_i}\alpha \in I_i(n)\setminus \p_i$. 

Let $n_0:=\max_{1\le i\le r} \{ n_i\}$. Let $n\ge n_0$ and let $f\in B(n)$. 
If $f\notin\cup_i \p_i$, then clearly $f+I(n)\not\subseteq \cup_{1\le i\le r} \p_i$. 
Assume now that $f \in\cup_i \p_i$, and for each $j$ such that $f \in \p_j$, choose $t_j\in I_j(n)\setminus \p_j$. Then we easily verify that 
$$f+\sum_{\p_j\ni f} t_j \ \in (f+I(n))\setminus \cup_{1\le i\le r} \p_i.$$ 
 
 (b) Let $n\ge d$. For each $j \leq r$, let us show that $I(n)\not\subseteq \p_j$.
Suppose by contradiction that $I(n)\subseteq \p_j$,  and choose $t\in B(1)\setminus \p_j$. Then
$t^{n-e}I(e)\subseteq I(n)\subseteq \p_j$ for all $1 \leq e \leq d$. Hence, $I(e)\subseteq \p_j$, and
then $I\subseteq \p_j$ because $I$ can be generated by the union of the $I(e)$, $1 \leq e \leq d$. Contradiction.

Let $f\in B(n)$, and suppose that $f+I(n)\subseteq \cup_{1\le i\le r} \p_i$.
Then 
$$ I(n) = \cup_{1\le i \le r} \left((-f+\p_i)\cap I(n)\right). $$  
If $(-f+\p_i)\cap I(n)$ is not empty, pick $c_i \in ((-f+\p_i)\cap I(n))$, and let $W_i:= -c_i +((-f+\p_i)\cap I(n))$. 
The reader will easily check that $W_i$ is  a $k$-subspace
of the $k$-vector space $I(n)$. Moreover, we claim that $W_i \neq I(n)$. 
Indeed, if $W_i = I(n)$, then $I(n) = c_i + W_i = (-f+\p_i)\cap I(n)$. But then $f\in \p_i$,
which implies that  $I(n)\subseteq \p_i$, a contradiction.
Therefore, the $k$-vector space $I(n)$ is a finite union 
of at most $r$ proper $k$-affine subspaces, and this is also a contradiction (\ref{union}). 
\qed 

\begin{lemma} \label{union} 
Let $V$ be a vector space over a field $k$. For $i=1,\dots, m$, let
$v_i\in V$ and let $V_i$ be a proper subspace of $V$.
If 
$\mathrm{Card}(k)\ge m+1$,
then $$V\ne (v_1 +V_1) \cup \ldots \cup (v_m +V_m).$$ 
If  $\mathrm{Card}(k)\ge m$, then 
$V\ne V_1 \cup \ldots \cup  V_m.$
\end{lemma}

\begin{proof} Assume that
$\mathrm{Card}(k)\ge m+1$,
and that $V= (v_1 +V_1) \cup \ldots \cup (v_m +V_m)$. We claim then that 
 $V=  V_1 \cup \ldots \cup V_m$. Indeed, fix $x \in V_1$, and let $y \in V \setminus V_1$.
Since $\mathrm{Card}(k^*) \geq m$, we can find at least $m$ elements of the form 
$v_1+(x+\lambda y)$ with $\lambda \in k^*$, and 
  $$v_1+(x+\lambda y)\in V\setminus (v_1+V_1)\subseteq \cup_{2\le i\le m} (v_i+V_i).$$
Thus there  exist an index $i$ and distinct $\lambda_1 $, $\lambda_2$ in $k^*$ such that
$v_1+x+\lambda_1 y$ and $ v_1+x+\lambda_2 y$ both belong to $v_{i}+V_i$. It follows that 
$(\lambda_1-\lambda_2)y\in V_i$ and, thus, $y\in V_i$.
Hence, $V=  V_1 \cup \ldots \cup V_m$. The second 
statement of the lemma is well-known and can be found for instance in \cite{BBS}, Lemma 2.   
\end{proof}
\medskip 

Our final lemma in this section is a key ingredient in the proofs of
Theorem~\ref{exist-hyp} and Theorem~\ref{pro.reductiondimension2}
and is used to insure that 
Condition (2) in Theorem~\ref{globalize} holds for 
$n$ big enough uniformly in $s\in S$. 

\begin{lemma}  \label{existence-fn}
Let $S$ be a noetherian affine scheme, and let $X\to S$ be
projective. Let $\cO_X(1)$ be a very ample sheaf on $X$ relative to 
$X\to S$. Let $C:=V(\J)$ be a closed subscheme of $X$. 
Let $A_0$ be a subset of $X$ disjoint from $C$ and such that there
exists $c_0\in \mathbb N$ with
$\mathrm{Card} (A_0\cap X_s) \le c_0$ for all $ s\in S$. 
Let $\oline{\J}_s$ denote the image of  $\J_s $ in $\cO_{X_s} $. 
Then there exists $n_0 \geq 0$ such that for all $s \in S$ and
for all $n \geq n_0$,
\begin{enumerate}[{\rm (a)}] 
\item There exists $f_{s,n}\in H^0(X_s, \oline{\J}_s(n))$ 
whose zero locus $V_+(f_{s,n})$ in $X_s$ does not contain any
point of $A_0$.
\item Suppose that $C\to S$ is finite. Then  there exists $f_{s,n}\in H^0(X_s, \oline{\J}_s(n))$
as in {\rm (a)} such that  the image of $f_{s,n}$ in 
$(\oline{\J}_s(n)/\oline{\J}_s^2(n))\otimes k(x)$ is non-zero for all $x\in C$ with
$(\oline{\J}_s(n)/\oline{\J}^2_s(n))\otimes k(x)\neq (0)$. 
\end{enumerate} 
\end{lemma}

\proof 
When $C\to S$ is
finite, we increase $c_0$ if necessary so we can assume that $\mathrm{Card}(C_s)\le c_0$ for all $s\in S$.  
Let 
$$Z_0:=\{ s\in S \mid {\mathrm{Card}}(k(s)) \le 2c_0\}.$$
Lemma \ref{finiteresidue} shows that $Z_0$ is a finite set.

Let $n_0$ be such that Lemma \ref{Jn-barJn} applies. 
Fix $n>n_0$. It follows from \ref{Jn-barJn} (b) that $\oline{\J}_s$ is $n$-regular for all $s \in S$.
Let $s \in S \setminus Z_0$. Then ${\mathrm{Card}}(k(s)) > {\mathrm{Card}}(A_0) + {\mathrm{Card}}(C_s)$.   
Parts (a) and (b) both follow from Lemma \ref{fn-bis} (a) applied to $\oline{\J}_s$, with $D$ empty 
in the proof of (a), and $D = C_s$ in the proof of (b).
For the remaining finitely many points $s\in Z_0$, we increase $n_0$ if necessary so that we can 
use Lemma \ref{fn-bis} (b) for each $s \in Z_0$.
\qed 

\end{section}

\begin{section}{Avoidance lemma for families} \label{hyperf}

We present in this section further applications of our method. Our first result below is a generalization of Theorem~\ref{exist-hyp}, where the noetherian hypothesis on the base has been removed. 

\begin{theorem} \label{bertini-type-0}
Let $S$ be an affine scheme, and let $X\to S$ be a 
quasi-projective and finitely presented morphism. Let
 $\cO_X(1)$ be a very ample sheaf relative to $X \to S$. Let 
\begin{enumerate}[{\rm (i)}]
\item $C$ be a closed subscheme of $X$, 
finitely presented over $S$;
\item $F_1, \dots, F_m$ be 
subschemes of $X$ of finite presentation over $S$;
\item $A$ be a finite subset of $X$ such that $A\cap C=\emptyset$. 
\end{enumerate} 
Assume that for all $s \in S$, $C$ does not contain any
irreducible component of positive dimension of $(F_i)_s$ and of $X_s$. 
Then there exists $n_0>0$ such that for all $n\ge n_0$, there exists 
a global section $f$ of $\cO_X(n)$ such that: 
\begin{enumerate}[\rm (1)]
\item The closed subscheme $H_f$ of $X$ is a hypersurface 
that contains $C$ as a closed subscheme; 
\item \label{HF} For all $s \in S$ and for all $i\le m$, $H_f$ does not contain 
any irreducible component of positive dimension of $(F_i)_s$; and
\item \label{add-A} $H_f\cap A=\emptyset$. 
\end{enumerate}
Assume in addition that $S$ is noetherian, and 
that $C\cap\Ass(X)=\emptyset$. Then there exists such a hypersurface 
$H_f$ which is locally principal. 
\end{theorem}

\proof The last statement when $S$ is noetherian is immediate from the 
main statement of the theorem: simply apply the main statement of the 
theorem with $A$ replaced by $A\cup\Ass(X)$. The fact that $H_f$ is 
locally principal when $H_f \cap \Ass(X) = \emptyset$ is noted in 
\ref{hypersurfaces-properties} (4). 

Let us now prove the main statement of the theorem. 
First we add $X$ to the set of subschemes $F_i$. Then the 
property of $H_f$ being a hypersurface results from 
\ref{bertini-type-0} (2). 
Our main task is to reduce to the case $S$ 
is noetherian and of finite dimension, in order to then apply
Theorem~\ref{exist-hyp}. 
Using \cite{EGA}, IV.8.9.1 and IV.8.10.5, we find the existence of  an affine scheme $S_0$ of 
finite type over $\mathbb Z$, and of  a morphism $S\to S_0$ such that all the objects of Theorem 
\ref{bertini-type-0} descend to $S_0$. More precisely,
there exists   a quasi-projective 
scheme $X_0\to S_0$ such that $X$ is isomorphic to $X_0\times_{S_0} S$.
We will denote by $p : X\to X_0$ the associated `first projection' morphism.  
There also exists a very ample sheaf  $\cO_{X_0}(1)$ relative to $X_0 \to S_0$ whose pull-back to 
$X$ is $\cO_{X}(1)$. There exists a closed 
subscheme $C_0$ of $X_0$
such that $C$ is isomorphic to $C_0 \times_{S_0} S$. 
 Finally, there exists 
subschemes $F_{1,0}, \dots, F_{m,0}$ of $X_0$ such that $F_i$ is
isomorphic to $F_{i,0} \times_{S_0} S$. Let $A_0:=p(A)$. 

Since $S_0$ is of finite type over $\mathbb Z$, $S_0$ is noetherian and 
of finite dimension. The data $X_0, C_0, \{ (F_{1,0}\setminus C_0), 
\dots, (F_{m,0}\setminus C_0) \}, A_0$ 
satisfy the hypothesis of Theorem~\ref{exist-hyp}. 
Let $n_0> 0$ and, for all $n\ge n_0$, 
an $f_0\in H^0(X_0, \cO_{X_0}(n))$ be given 
by Theorem~\ref{exist-hyp} with respect to these data. 
Let $H_{f}$ be the closed subscheme of $X$ define by the canonical
image of $f_0$ in $H^0(X, \cO_X(n))$. Then $H_f=H_{f_0}\times_{S_0} S$ 
contains $C$ as a closed subscheme and $H_f\cap A=\emptyset$. 
It remains to check Condition (2) of \ref{bertini-type-0}. 
Let $\xi$ be the generic point of an irreducible component 
of positive dimension of $F_{i,s}$. Let $s=p(s_0)$ and let
$\xi_0=p(\xi)$. Then an open neighborhood of $\xi$ in $(F_i)_s$
has empty intersection with $C$. As $C=C_0\times_{S_0} S$,
this implies that the same is true for $\xi_0$ in $(F_{i,0})_s$. 
Hence $\xi_0$ is the generic point of an irreducible component of
$(F_{i,0}\setminus C_0)_s$ of positive dimension. Thus 
$\xi_0\notin H_{f_0}$ and $\xi\notin H_f$.
\qed

\begin{remark} 
The classical Avoidance Lemma states that if $X/k$ is a quasi-projective
scheme over a field, $C \subsetneq X$ is a closed subset of positive codimension, 
and $\xi_1,\dots, \xi_r$ are points of $X$ not contained in $C$, then there exists a 
hypersurface $H$ in $X$ such that $C \subseteq H$ and $\xi_1, \dots, \xi_r \notin H$.

Let $S$ be a noetherian scheme, and let $X/S$ be a quasi-projective scheme. 
One may wonder whether it is possible 
to strengthen Theorem \ref{bertini-type-0}, the Avoidance Lemma for Families, 
by strengthening its Condition (\ref{HF}). 
The following example shows that  Theorem \ref{bertini-type-0}
does not hold if Condition (\ref{HF}) is replaced by the stronger 
Condition (\ref{HF}'): For all $s \in S$,
 $H_f$ does not contain any irreducible component of $F_s$.
 
 Let $S=\Spec R$ be a Dedekind scheme such that $\Pic(S)$ is not a torsion group (see, e.g., 
\cite{Gol}, Cor.\ 2).
Let ${\mathcal L}$ be an 
invertible sheaf on $S$ of infinite order.
Consider as in \ref{threesections}
the   scheme $X={\mathbb P}(\cO_S \oplus {\mathcal L})$ with its natural projective morphism $X \to S$.
Let $F$ be the union of the two horizontal sections $C_0$ and $C_{\infty}$. If Theorem \ref{bertini-type-0} with Condition (\ref{HF}')
holds, then there   exists a hypersurface $H_f$ which is a 
finite quasi-section  
of $X \to S$ (as defined in  \ref{def.finite-qs}), and which does 
not meet $F$. Proposition 
\ref{threesections} shows that this can only happen when ${\mathcal L}$ has finite order. 
\end{remark} 
\begin{remark} \label{ex.avoidance}
Let $k$ be any field, and let $X/k$ be an irreducible proper scheme over $k$. 
Let $C \subsetneq X$ be a closed subscheme, 
and let $\xi_1,\dots, \xi_r$  be points of $X$ not contained in $C$.
We may ask whether an Avoidance Lemma holds for $X/k$ in the following senses: (1) Does there exist a 
line bundle ${\mathcal L}$ on $X$ and a section $f \in {\mathcal L}(X)$ such that the closed subscheme $H_f$ 
contains $C$  and $\xi_1, \dots, \xi_r \notin H_f$? We may also ask (2) whether  there exists a codimension $1$ subscheme $H$ of $X$  such that  
 $H$ contains $C$  and $\xi_1, \dots, \xi_r \notin H$.

The answer to the first question is negative, as there exist proper schemes $X/k$ with $\Pic(X)= (0)$.
For instance, a normal proper surface $X/k$ over an uncountable field $k$  with $\Pic(X)= (0)$ is constructed in 
\cite{Sch}, section 3. 

The answer to the second question is also negative when $X/k$ is not smooth. Recall the example of Nagata-Mumford
(\cite{Art}, pp.\ 32-33). Consider the projective plane ${\mathbb P}^2_k/k$, and fix an elliptic curve $E/k$ in it, with origin $O$.
Assume that $E(k)$ contains a point $x$ of infinite order. Fix ten distinct multiples $n_ix$, $i=1,\dots, 10$.
Blow up ${\mathbb P}^2_k$ at these ten points to get a scheme $Y/k$.
Since $x$ has infinite order, any codimension $1$ closed subset of $Y$ intersects the strict transform
$F$ of $E$ in $Y$. Now $F$ has negative self-intersection on $Y$ by construction, and so there exists  an algebraic space $Z$ and a morphism $Y \to Z$ which contracts $F$.
The algebraic space $Z$ is not a scheme, since the image $z$ of $F$ in $Z$
cannot be contained in an open affine of $Z$.
It follows from \cite{LM}, 16.6.2, that there exists a scheme $X$ with a finite surjective morphism $X \to  Z$. The finite set consisting of the preimage of $z$ in $X$
meets every codimension $1$ subscheme $H$ of $X$.
\end{remark}

\begin{remark} Let $S$ be an affine scheme and let $X \to S$ be projective and smooth. Fix a very ample invertible sheaf $\cO_X(1)$ relative to $X\to S$ as in Theorem \ref{bertini-type-0}.
It is not possible in general to find $n>0$ and a global section $f \in \cO_X(n)$ such that $H_f \to S$ is smooth.
Examples of N.\ Fakhruddin illustrating this point can be found in \cite{Poo}, 5.14 and 5.15.
M.\ Nishi (\cite{Nishi}, and \cite{DH}, Remarks, page 80, (b)) gave an example of a nonsingular cubic surface $C$ in ${\mathbb P}^4_k$
which is not contained in any nonsingular hypersurface of ${\mathbb P}^4_k$.
Our next two corollaries are examples of weaker `theorems of Bertini-type for families', where for instance smooth is replaced by Cohen-Macaulay.
\end{remark}

Recall that a locally noetherian scheme $Z$ is $(\mathrm{S}_\ell)$ for some integer $\ell\ge 0$ 
if for all $z\in Z$, the depth of $\cO_{Z,z}$ is at least equal 
to $\min\{\ell, \dim\cO_{Z,z}\}$ (\cite{EGA}, IV.5.7.2). 

\begin{corollary}\label{bertini-cor1} 
Let $S$ be an affine scheme, and let $X\to S$ be a 
quasi-projective and finitely presented morphism. 
Let $C$ be a closed subscheme of $X$ 
finitely presented over $S$.  
Assume that for all $s \in S$, $C$ does not contain any
irreducible component of   positive 
dimension of  $X_s$.
 Suppose that for some $\ell \geq 1$, $X_s$ is $(\mathrm{S}_\ell)$ for
 all $s \in S$. Let $\cO_X(1)$ be a very ample sheaf relative to $X \to S$.

Then there exists $n_0>0$ such that for all $n\ge n_0$, there exists 
a global section $f$ of $\cO_X(n)$ 
such that $ H_f \subset X$ is a hypersurface   containing 
$C$ as a closed subscheme, and 
the fibers of $H_f\to S$ are $(\mathrm{S}_{\ell-1})$. 
In particular, if the fibers of $X\to S$ are Cohen-Macaulay, then
the same is true of the fibers of  $H_f\to S$. Moreover:
\begin{enumerate}[\rm (a)]
\item Assume that $X_s$ has no isolated point 
for all $s \in S$. 
If $X\to S$ is flat, then $H_f\to S$ can be assumed to be flat and locally principal.
\item Assume that  $S$ is noetherian and
that $\Ass(X) \cap C= \emptyset$. 
Then  $H_f \to S$ can be assumed to be  locally principal.
\end{enumerate}
\end{corollary}

\proof We apply Theorem \ref{bertini-type-0} to $X\to S$ and $C$, with 
$F_1=X$, $m=1$ and with $A= \emptyset$.
Let $H:= H_f$ be a 
hypersurface in $X$ as given by \ref{bertini-type-0}. 
For all $s\in S$, $H_s$ does not contain any irreducible component of 
$X_s$ of positive dimension. Since $X_s$ is $(\mathrm{S}_\ell)$ with $\ell\ge 1$, $X_s$ has no embedded
points. At an isolated point of $X_s$ contained in $H$, $H_s$ is trivially $(\mathrm{S}_{k})$ for any $k \geq 0$.
At all other points, it follows that $H_s$ is locally generated everywhere by a regular 
element and, thus,  $H_s$ is $(\mathrm{S}_{\ell-1})$. 

The statement (a) follows from \ref{hypersurfaces-properties} (3).
For (b), we apply Theorem \ref{bertini-type-0} and 
Lemma~\ref{hypersurfaces-properties} (4) 
to $X\to S$ and $C$, with  $F_1=X$ and with $A=\Ass(X)$. 
\qed

\begin{corollary} \label{generic-S1} 
Let $S$ be an affine irreducible scheme of dimension $1$. Let 
$X\to S$ be quasi-projective and flat of finite presentation. 
Assume that its  generic fiber is $(\mathrm{S}_1)$, and  
that it does not contain any isolated point. Let 
$\cO_X(1)$ be a very ample sheaf relative to $X \to S$. 
Then there exists $n_0>0$ such that for all $n\ge n_0$, there exists 
a global section $f$ of $\cO_X(n)$  
 such that 
$H_f \subset X$ is  a locally principal hypersurface, flat over $S$. 
\end{corollary}

\proof Consider the set $M$ of all $x \in X$ such that 
$X_s$ is not $(\mathrm{S}_1)$ at $x$ (or, equivalently, the set of all $x\in X$ such that $x$ is contained in an embedded 
component of $X_s$). 
Then this set is constructible (\cite{EGA}, IV.9.9.2 (viii), and even closed since $X\to S$ is flat, \cite{EGA}, IV.12.1.1 (iii)).
Since the generic fiber is $(\mathrm{S}_1)$, the image of $M$ in $S$  must be finite because $S$ has dimension $1$. 
Therefore, there are only finitely many $s \in S$ such that the fiber $X_s$ is not $(\mathrm{S}_1)$.
Let $M'$ denote the set of associated points in the fibers that are not $(\mathrm{S}_1)$. This set is finite.
Apply now Theorem \ref{bertini-type-0} with the very ample sheaf $\cO_X(1)$ and 
with $F_i=C=\emptyset$ and $A=M'$, to find a hypersurface $H_f$ which does not intersect $M'$.
This hypersurface is locally principal and flat over $S$  
by \ref{hypersurfaces-properties} (3). Indeed, by construction, $(H_f)_s$ does not contain any 
irreducible component of $X_s$ of positive dimension. Our hypothesis on the generic fiber 
having no isolated point implies that $X_s$ has no isolated point
for all $s$, since the set of all points that are isolated in their fibers is open (\cite{EGA}, IV.13.1.4). 
\qed 

\smallskip
We discuss below one additional application of Theorem \ref{globalize}.

\begin{proposition} \label{generic_smoothness} 
Let $S=\Spec R$ be an affine scheme and let $\pi:X\to S$ be 
projective and finitely presented. 
Let $C$ be a closed subscheme of $X$, finitely presented over $S$. 
Let $\cO_X(1)$ be a very ample sheaf relative to $X \to S$. 
Let $Z \subset S$ be a finite subset. Suppose that  
\begin{enumerate}[{\rm (i)}] 
\item $\pi:X\to S$ is smooth at every point of  $\pi^{-1}(Z)$;
\item for all $s\in Z$, $C_{s}$ is smooth and 
$\codim_x(C_s, X_s)>\frac{1}{2}\dim_x X_s$ for all $x\in C_s$;
\item for all $s\in S$, $C_s$ does not contain any irreducible
component  of positive dimension of $X_s$, and for all $s \in Z$, $X_s$ has no isolated point.
\end{enumerate} 
Then there exists an integer $n_0$ such that for all $n \geq n_0$, there exists
a global section $f$ of $\cO_X(n)$ such that $H_f$ is a hypersurface containing $C$ as a closed subscheme 
and such that $H_f\to S$ is smooth in an open neighborhood of $\pi^{-1}(Z)\cap H_f$. 
\end{proposition}

\begin{proof} 
By arguing as in the proof of Theorem~\ref{bertini-type-0} 
(using also \cite{EGA}, IV.11.2.6, or \cite{EGA}, IV.17.7.8), 
we find that it suffices to prove the proposition 
in  the case where $S$ is noetherian and has finite dimension.

Let $\J$ be the ideal sheaf defining $C$. 
As in the proof in \ref{Proofreductiondimension2} of Theorem~\ref{exist-hyp}, 
when $m=1$ and $F_1=X$, there exists $n_0$ such that 
for any $n\ge n_0$ and for any choice of generators 
$f_1, \dots, f_N$ of $H^0(X, \J(n))$, the associated constructible
subset $T_X\subseteq \mathbb A^N_S$ (see  \ref{remove-isolated-pts}) has dimension at most $ N -1$. 
For each $s\in Z$, 
we can apply Lemma~\ref{constructible-sm} to the following data: the morphism $X_s \to \Spec k(s)$, $C_s\subseteq X_s$
defined by the ideal sheaf $\oline{\J}_s   \subseteq \cO_{X_s}$ (notation as in \ref{emp.notation}), and the sections 
$\overline{f}_{i,s}$,  image of $f_i$ in 
$H^0(X_s,\oline{\J}_s(n))$, $i=1,\dots, N$.
We denote by $T_{Z,s}$ the constructible subset 
of $\mathbb A^N_{k(s)}$ associated in Lemma~\ref{constructible-sm} (1) with this data.
Note that $T_{Z,s}$ is then a pro-constructible subset 
of $\mathbb A^N_S$.

Let $n$ and $ f_1, \dots, f_N$ be as above. 
Consider the finite union $T$ of $T_X$ and the 
pro-constructible subsets $T_{Z,s}$, $s\in Z$. It is pro-constructible
with constructible fibers over $S$. 
Lemma~\ref{constructible-sm}   shows that 
Conditions (1) and (2) in \ref{globalize} (with $V=S$) are
satisfied, after increasing $n_0$ if necessary. Then by 
Theorem \ref{globalize}, 
there exists a  
section $(a_1,\ldots, a_N)\in R^N=\mathbb A_S^N(S)$ such that for all $s \in S$,
$(a_1(s), \ldots, a_N(s))$ is a $k(s)$-rational point of $\mathbb A^N_{k(s)}$ 
that is not contained in $T_s$. 
Let $f:=\sum_{i=1}^N a_if_i$ and consider the closed subscheme
$H_f \subset X$. As $f\in H^0(X, \J(n))$, $C$ is a closed
subscheme of $H_f$. By definition of $T_{X}$, we find that
for all $s\in S$, $H_f$ does not contain any 
irreducible component of $X_s$ of positive dimension, so that $H_f$ is a hypersurface relative to $X \to S$.
By definition of $T_{Z,s}$, we find that  $(H_f)_s$ is smooth for all $s\in Z$. 
Since $X \to S$ is flat in a neighborhood $U$ of $\pi^{-1}(Z)$, and since 
$X_s$ does not contain any isolated point when $s \in Z$, we find that 
$U \cap H_f \to S $ is flat at every point of $\pi^{-1}(Z)$ 
(\ref{hypersurfaces-properties} (3)).  
So $H_f\to S$ is smooth in a neighborhood
of $\pi^{-1}(Z)$ by the openness of the smooth locus. 
\end{proof}

\newcommand{\sing}{\mathrm{sing}}

\begin{lemma} \label{constructible-sm} 
Let $S$ be a noetherian scheme. Let 
$X\to S$ be a quasi-projective morphism.
Let $\cO_X(1)$ be a very ample sheaf relative to $X \to S$. Let $C$ be a closed subscheme of $X$ defined by a sheaf 
of ideals $\J$. 
Let $n\ge 1$ and let $f_1, \dots, f_N \in H^0(X, \J(n))$. 
\begin{enumerate}[{\rm (1)}] 
\item Define for any $s\in S$ 
$$\Sigma_{\sing}(s):=
\left\{ (a_1,\dots, a_N) \in k(s)^N \mid V_+(\sum_i a_i \overline{f}_{i,s})\subseteq X_{s} 
\text{\ is not smooth over } k(s)\right\}. $$ 
Then there exists a constructible subset $T_{\sing}$ of $\A^N_S$ such that
for all $s\in S$, the set of $k(s)$-rational points of $\A^N_S$ contained in $T$ is equal to $\Sigma_{\sing}(s)$.
Moreover, 
$T_{\sing}$ is compatible with base changes $S'\to S$ 
as in Proposition~{\rm \ref{constructible-conditions} (a)}. 
\item Let $k$ be a field and assume that $S=\Spec k$. Suppose also that $C$ and $X$ are projective 
and smooth over $k$, and that 
$\codim_x(C, X)> \frac{1}{2}\dim_x X$ for all $x\in C$. 

Then  there exists $n_0$ such that for any 
$n\ge n_0$ and for any choice of a system of generators 
$f_1, \dots, f_N$ of $H^0(X, \J(n))$, we have $\dim T_\sing\le N-1$, and
there exists 
$(a_1, \dots, a_N)\in k^N$ such that $V_+(\sum_i a_if_i)$ is smooth
and does not contain any irreducible component of $X$.
\end{enumerate}
\end{lemma}

\begin{proof} (1) 
Write $\A^N_S=\Spec \cO_S[u_1, \dots, u_N]$. 
Consider  the natural projections 
$$p: \A^N_X\to X, \quad 
q: \A^N_X\to \mathbb A^N_S.$$
With the appropriate identifications, consider the global section 
$\sum_{1\le i\le N} u_i f_i$ of $p^*({\J}(n))$,
which defines the closed subscheme 
$Y:=V_+(\sum_{1\le i\le N} u_i {f}_i )$ of $\A^N_X$.
Consider 
$$\mathcal S: =\{ y\in Y \mid Y_{q(y)} \ \text{is not smooth at $y$ over }
k(q(y))\}.$$
Set $T_{\sing}:=q(\mathcal S)$. Clearly $T_{\sing}$ satisfies all the requirements of the lemma except
for the constructibility, which we now prove. By Chevalley's theorem, 
it is enough to show that $\mathcal S$ is constructible. 
The subset $\mathcal S$ is the union for 
$0\le d\le \max_{z\in \A^N_S} \dim Y_z$ (\cite{EGA}, IV.13.1.7) of the subsets
$$\mathcal S_d:=\{ y\in Y \mid \dim_y Y_{q(y)} \le d <
\dim_{k(y)} (\Omega^1_{Y/\A^N_S}\otimes k(y))\}.$$
Thus it is enough to show that $\mathcal S_d$ is constructible. 
The set $\{ y\in Y \mid \dim_y Y_{q(y)} \le d \}$ is open by
Chevalley's semi-continuity theorem (\cite{EGA}, IV.13.1.3). 
On the other hand, for any coherent sheaf $\F$ on $Y$, the subset
$\{ y\in Y \mid \dim_{k(y)}(\F\otimes k(y)) \ge d+1\}$ is closed
in $Y$. 
So $\mathcal S_d$ is constructible and (1) is proved. 

(2) By the compatibility with base changes, the dimension of $T_{\sing}$ 
can be computed over an algebraic closure of $k$. Now over an infinite field, 
Theorem (7) of \cite{KA}, p.\ 787, in the simplest possible case, 
where the subvariety is smooth and equal to the open subset in Theorem (7), 
implies that the generic point of $\A^N_k$ is not contained in $T_\sing$
for all $n$ big enough and for all systems of generators 
$\{f_1,\dots, f_N\}$ of $H^0(X, \J(n))$. 
As $T_\sing$ is constructible, we have $\dim T_\sing \le N-1$. 

Let $\Gamma_1,\dots, \Gamma_m$ be the connected components of $X$. 
They are irreducible and 
$$H^0(X, \J(n))=\oplus_{1\le i\le m} H^0(\Gamma_i, \J(n)|_{\Gamma_i}).$$ 
By hypothesis, $\dim C\cap \Gamma_i < \frac{1}{2} \dim \Gamma_i$. So by 
\cite{KA}, loc. cit., when $k$ is infinite, and by 
\cite{Poo2}, Theorem 1.1 (i) when $k$ is finite, increasing $n_0$ 
if necessary, for any $n\ge n_0$, there exists 
$g_i\in H^0(\Gamma_i, \J(n)|_{\Gamma_i})$ such that 
$V_+(g_i)\subset \Gamma_i$ is smooth and of dimension $\dim\Gamma_i -1$. 
Let $f:=g_1\oplus \cdots \oplus g_m\in H^0(X, \J(n))$. 
Then $V_+(f)$ is a smooth subvariety of $X$ not containing any 
irreducible component of $X$. Let $f_1,\dots, f_N$ be any system 
of generators of $H^0(X, \J(n))$, and write 
$f=\sum_i a_if_i$ with $a_i\in k$. Then 
$(a_1,\dots, a_N)\in k^N$ is the desired point. 
\end{proof}

\end{section}

\begin{section}{Finite quasi-sections} \label{finite-qs} 

Let $X\to S$ be a surjective morphism. We call a closed subscheme 
$T$ 
of $X$ a \emph{finite quasi-section}
when $T \to S$ is finite and surjective (\ref{def.finite-qs}). 
We establish in \ref{quasisections} the existence
of a finite quasi-section for certain types of projective morphisms.
The existence of quasi-finite quasi-sections locally on $S$ for flat or smooth morphisms is discussed in \cite{EGA}, IV.17.16.

When $S$ is integral noetherian of dimension $1$ 
and $X\to S$ is proper,
the existence of a finite quasi-section $T$ is well-known and easy to establish.
It suffices to take $T$ to be the Zariski closure of a closed point of the 
generic fiber of $X\to S$. Then $T\to S$ have fibers of dimension $0$ (see, e.g., \cite{Liubook}, 8.2.5), 
so it is quasi-finite 
and proper and, hence, finite. When $\dim S>1$, 
the process of taking the closure of a  closed point of the generic 
fiber does not always produce a closed subset {\it finite} over $S$, as the 
simple example below shows.

\begin{example} \label{easy}  Let $S=\Spec A$ with $A$ a noetherian integral domain, and let 
$K=\Frac(A)$. Let $X=\mathbb P^1_A$. 
Choose coordinates and write $X = {\rm Proj \ } A[t_0,t_1]$.
Let $P\in X_K(K)$ be given as $(a:b)$, with $a, b\in A\setminus \{ 0 \}$. 
When $(bt_0-at_1)$ is a 
prime ideal in $A[t_0, t_1]$, then $T:=V_+(bt_0-at_1)$ is the
Zariski closure of $P$ in $X$. When in addition $aA+bA\ne A$, $T$ is not 
finite over $S$. For a concrete example with $S$ regular of dimension $2$, 
take $k$ a field and $A=k[t, s]$,  with $a=t$, and $ b=s$. (Note that when 
$\dim(A)=1$ and $aA+bA\ne A$, the ideal $(bt_0-at_1)$ is never prime in $A[t_0,t_1]$).

More generally, to produce $K$-rational points on the generic fiber of ${\mathbb P}^n_S \to S$ for some $n>1$ whose closure is not finite over $S$, 
we can proceed as follows.
Let $T \to S$ be the blowing-up of $S$ with respect to a coherent sheaf of ideals $I$,
and choose $I$ so that $T\to S$ is not finite.
Then $T\to S$ is a projective morphism, and we can choose $T\to {\mathbb P}^n_S$ to be a closed immersion over $S$ for some $n>0$.
Let $\xi$ denote the generic point of the image of $T$ in $X:={\mathbb P}^n_S$.
Then $\xi$ is a closed point of the generic fiber of $X\to S$, and the closure of $\xi$ in $X$ is not finite over $S$.

The composition $\mathbb P^d_T= {\mathbb P}^d_S \times T \to T \to S$ is an example of a projective morphism which 
does not have any finite quasi-section. In this example, one irreducible fiber has dimension greater than $d$. 
\end{example}

Before turning to the main theorem of this section, let us note here
an instance of interest in arithmetic geometry where the closure of a rational point of the generic 
fiber is a section. 

\begin{proposition} \label{closurequasisection}
Let $S$ be a noetherian regular integral scheme, with function 
field $K$. Let $X\to S$ be a proper morphism such that no geometric 
fiber $X_{\bar{s}}$ contains a rational curve. Then 
any $K$-rational point of the generic fiber of $X\to S$ 
extends to a section over $S$.
\end{proposition} 

\proof Let $T$ be the (reduced) Zariski closure of a rational
point of the generic fiber of $X\to S$. 
Consider the proper birational morphism $f: T\to S$.
Denote by $E:=E(f)$ the exceptional set of $f$, that is,
the set of points $x \in T$
such that $f$ is not a local isomorphism at $x$. 
Suppose $E\ne\emptyset$. 
Since $S$ is regular, by van der Waerden's purity theorem
(\cite{EGA}, IV.21.12.12 or \cite{Liubook}, 7.2.22),
$E$ has pure codimension $1$ in $T$. Let $\xi$ be a generic
point of $E$ and let $s=f(\xi)$. 
Using the dimension formula (\cite{EGA}, IV.5.5.8, \cite{Liubook}, 8.2.5) and because $S$ is 
regular hence universally catenary (\cite{EGA}, IV.5.6.4), we find 
$$ \mathrm{trdeg}_{k(s)} k(\xi)=\dim \cO_{S,s}-1.$$  
Let $\overline{T}\to T$ be the normalization of $T$ and let $\eta$ 
be a point of $\overline{T}$ lying over $\xi$. Then by Krull-Akizuki, 
$\cO_v:=\cO_{\overline{T}, \eta}$ is a discrete valuation ring. It has 
center $s$ in $S$. As $k(\eta)$ is algebraic (even finite) 
over $k(\xi)$, we have $\mathrm{trdeg}_{k(s)} k(\eta)=\dim
\cO_{S,s}-1$.  
So $\cO_v$ is a prime divisor of $K(S)$ in the sense of 
\cite{Abh}, Definition 1. It follows from 
a theorem of Abhyankar (\cite{Abh}, Proposition 3) 
that $k(\eta)$ is the function field of a ruled variety of positive
dimension over $k(s)$. One can also prove this result in a more geometric 
flavor as in \cite{Liubook}, Exercise 8.3.14 (a)-(b) (the 
hypothesis that the base scheme is Nagata is not needed in 
our situation as the local rings which intervene are all regular). 
So $\overline{T}_{\bar{s}}$ contains a rational curve. 
As $\overline{T}_{\bar{s}}\to T_{\bar{s}}$ is integral, 
the image of such a curve is a rational curve in $T_{\bar{s}}$. 
It follows that 
$X_{\bar{s}}$ contains a rational curve, and this is a contradiction.
So $E$ is empty and $T\to S$ is an isomorphism.
 \qed 

\begin{theorem} \label{quasisections}
Let $S$ be an affine scheme  and let $\pi:X\to S$ be a projective, finitely 
presented  
morphism. Suppose that all fibers of $X\to S$ 
are of the same dimension $d\ge 0$. Let $C$ be a finitely presented closed subscheme of $X$, 

with $C \to S$ finite but not necessarily surjective.
Then there exists a finite quasi-section $T \to S$ of finite presentation which contains $C$. 
Moreover:
\begin{enumerate}[\rm (1)]
\item Assume that $S$ is noetherian. If $C$ and $X$ are both irreducible, then there exists 
such a quasi-section with $T$ irreducible.
\item   
If $X\to S$ is flat with Cohen-Macaulay fibers (e.g., if
$S$ is regular and $X$ is Cohen-Macaulay),  
then there exists 
such a quasi-section with $T\to S$ flat. 
\item  
If $X \to S$ is flat and a local complete intersection morphism\footnote{Since the morphism $X\to S$ is flat, 
it is a local complete intersection morphism if and  only if every fiber is a local complete intersection morphism (see, e.g., \cite{Liubook}, 6.3.23).},  
then there exists such a quasi-section with $T\to S$ flat and a local complete intersection morphism.
\item 
Assume that $S$ is noetherian. Suppose that $\pi:X \to S$ has fibers pure 
of the same dimension, and that $C \to S$ is unramified. 
Let $Z$ be a finite subset of $S$ (such as the set of generic points of $\pi(C)$), and suppose that there
exists  an open subset $U$ of $S$ containing $Z$
such that $X \times_S U \to U$ is smooth. Then there exists 
such a quasi-section $T$ of $X \to S$ and an open set $V \subseteq U$ containing $Z$
such that $T \times_S V \to V$ is \'etale.
\end{enumerate}
\end{theorem}
\proof
To prove the first conclusion of the theorem, it suffices to show that $X/S$ has a finite quasi-section $T$ of finite presentation.
Then $T \cup C$ is a finite quasi-section which contains $C$.
 If $d=0$, then $X\to S$ itself is finite. Suppose $d\ge 1$. 
It follows from Theorem \ref{bertini-type-0}, with $ A=\emptyset$ and 
$F=\emptyset$, that there exists a hypersurface $H$ in $X$.
By definition of a hypersurface, for all $s \in S$, 
$H_s$ does not contain any irreducible component of $X_s$ of positive dimension.
Lemma \ref{hypersurfaces-properties}(1) and our hypotheses show that every fiber $H_s$ has dimension $d-1$.
Lemma \ref{hypersurfaces-properties}(2) shows that 
$H/S $ is also finitely presented. Repeating this 
process another $d-1$ times produces the desired quasi-section. 

(1) Since $X$ is assumed irreducible and since the fibers of $X \to S$ are all not empty by hypothesis, 
we find that $X\to S$ is surjective and that $S$ is irreducible. When $d=0$, $X\to S$ is then an irreducible finite quasi-section, and contains $C$
as a closed subscheme. Assume now that $d \geq 1$. 
Then we can find a hypersurface $H_f$ which contains $C$ as a closed subscheme (\ref{bertini-type-0}).  
Since $S$ is noetherian, we can use \ref{components-dom} below and the assumption that $C$ is irreducible to  find 
an irreducible component $\Gamma$ of $H_f$ which contains
(set-theoretically) $C$, dominates $S$, and such that all fibers of $\Gamma \to S$ have 
dimension $d-1$. Let $\J_C$ and  $\J_{\Gamma}$ denote the sheaves of ideals in $\cO_X$ defining
$C$ and $\Gamma$, respectively. Then some positive power $\J_{\Gamma}^m$ is contained in $\J_C$, and we endow
the irreducible closed set $\Gamma$ with the structure of scheme given by 
the structure sheaf $\cO_X/\J_{\Gamma}^m$. By construction, the scheme $\Gamma$ is irreducible and contains
$C$ as a closed  subscheme. If $d-1>0$, we 
repeat the process
with $\Gamma \to S$. 

(2) When $d=0$, the statement is obvious. Assume now that $d >0$. 
Since $X_s$ has no embedded point for all $s \in S$, we  find that for each $i \geq 0$, 
the set $X_i$ of all $x \in X$ such that every irreducible component of $X_{\pi(x)}$ passing through $x$ has dimension $i$ is open in $X$
(\cite{EGA}, IV.12.1.1 (ii), using here that $X \to S$ is flat). 
Moreover, since $X_s$ is Cohen-Macaulay for all $s$, the irreducible components of $X_s$ passing 
through a given point $x$ have the same dimension.  
We find that $X$ is the disjoint union of the open sets $X_i$. 
Each $X_i \to S$ is of finite presentation, since each $X_i$ is open and closed in $X$ (\cite{EGA}, IV.1.6.2 (i)).

Consider now $X_0 \to S$, which is clearly quasi-finite of finite presentation and flat. 
Since  $X \to S $ is projective, $X_0 \to S$ is then also finite (\cite{EGA} IV.8.11.1). We apply Corollary \ref{bertini-cor1} (a) to the 
finitely presented scheme $X':=(X \setminus X_0) \to S$ and the finite quasi-section
$C' := C \times_X X'$. We obtain a hypersurface $H'$ containing $C'$, and using the same method as in the proof of the first statement of the 
theorem, we obtain a finite flat quasi-section $T'$ of $X' \to S$ containing $C'$. 
Then $T:= T' \cup X_0$ is the desired finite flat quasi-section. 

To prove (3), we proceed as in (2), and remark that the hypersurface $H'$  obtained from 
Corollary \ref{bertini-cor1} (a) is flat and locally principal, so that its fiber ${H'}_s$ is
l.c.i.\ over $k(s)$ when $X_s$ is. By hypothesis, $X_0/S$ has only l.c.i\ fibers, and (3) follows.

(4) 
When $d=0$, $X \to S$ is the desired
finite quasi-section, since it is \'etale over the given open subset $U$ of $S$. Assume now that $d>0$. 
By hypothesis, $C \to S$ is finite and unramified, so that for each $s \in S$, $C_s \to \Spec k(s)$ is smooth. Moreover,
since we are assuming that the fibers are pure of dimension $d$, Condition (iii) in \ref{generic_smoothness} is satisfied.
We can therefore apply Proposition~\ref{generic_smoothness} with $Z$, to find a hypersurface $H_f$ of $X\to S$ containing $C$ as
a closed subscheme, with $H_f$
smooth over an open neighborhood $W$ of $Z$ in $S$. 
For all $s\in S$, $X_s$ is pure of dimension $d$ and
$(H_f)_s$ is a hypersurface
in $X_s$. Thus, $(H_f)_s$ is pure of dimension $d-1$ for all $s\in S$. Therefore, the above discussion 
can be applied  to the morphism $H_f \to S$, which induces a smooth morphism $H_f\times_S W\to W$, to produce a hypersurface $H_{f_2} $ of $H_f\to S$
containing $C$ as a closed subscheme, with $H_{f_2}$
smooth over an open neighborhood $W_2$ of $Z$ in $S$.
Thus, we obtain  
the desired  finite quasi-section after $d$ such steps. 
\qed 

\begin{lemma}\label{components-dom}
Let $S$ be affine, noetherian, and irreducible, with generic point $\eta$. Let $\pi: X\to S$ be a
morphism of finite type. 
For each irreducible component $\Delta $ of $X$, suppose that 
$\Delta \to S$   
has generic fiber of positive dimension.  
Let $\mathcal L$ be an invertible sheaf on $X$ with a global section $f$, and assume that
$H:=H_f \subset X$ is a hypersurface relative to $X \to S$.
Then:
\begin{enumerate}[{\rm (1)}]
\item Each irreducible component $\Gamma$ of $H$ dominates $S$. 
\item Assume in addition that for some $d\ge 1$, 
the fibers of each morphism $\Delta \to S$ 
all have dimension $d$.
Then $X\to S$ is equidimensional 
of dimension $d$  and $H\to S$ is 
equidimensional of dimension $d-1$. 
\end{enumerate}
\end{lemma}

\proof (1) Apply 
\cite{EGA}, IV.13.1.1, to each morphism $\Delta \to S$
to find that for all $s\in S$, 
all irreducible components of $X_s$ have 
positive dimension. 
Let $\Gamma$ be an irreducible component of $H$. Let $Z$ denote 
the Zariski closure of $\pi(\Gamma)$ in $S$. We need to show that $Z=S$. 
Let us first show by contradiction that $\codim(\Gamma, X_Z)>0$. Otherwise, $\Gamma$ 
contains an irreducible component $T$ of $X_Z$. Let $t$ be the generic
point of $T$. Since $T_{\pi(t)}$ is irreducible and 
dense in $T$, it is an irreducible component of $X_{\pi(t)}$. In particular, 
$T_{\pi(t)}$ has positive dimension, and is contained in $\Gamma$. This contradicts 
our   hypothesis that $H$
is a hypersurface. 
 
Every irreducible component of $X_Z$ is contained in an irreducible component of $X$, and 
every irreducible component of $X$ has non-empty 
generic fiber. Thus, if  $Z\ne S$, then $\codim(X_Z, X)>0$, and 
$$\codim(\Gamma, X)\ge \codim(\Gamma, X_Z)+\codim(X_Z, X)\ge 2.$$
This is a contradiction with the inequality   $\codim(\Gamma, X)\leq 1$, which follows from
 Krull's Principal Ideal Theorem. Hence, $Z=S$. 

(2) Let us first show that $X\to S$ is equidimensional of dimension $d\ge 1$.
The definition of {\it equidimensional} is found in \cite{EGA}, IV.13.3.2. 
We use \cite{EGA}, IV.13.3.3
to prove our claim. Indeed, our hypotheses imply that the image under $X \to S$ of each irreducible component $\Delta$ of $X$
is   $S$, and that the generic fibers of all induced morphisms $\Delta \to S$
have equal dimension $d$. Let $x \in \Delta$, and let $s \in S$ be its image. 
We have $\dim_x \Delta_s \leq \dim \Delta_s=d$.
We find using \cite{EGA}, IV.13.1.6, that  $\dim_x \Delta_s \geq \dim \Delta_\eta=d$.
Our claim follows immediately.

Let now
$\Gamma$ be an irreducible component of $H$. We know from (1) that $\Gamma \to S$ 
is dominant. Thus, using \cite{EGA}, IV.13.3.3, 
to show that $H \to S$ is equidimensional of dimension $d-1$, it suffices to 
show that $\Gamma \to S$ is equidimensional of dimension $d-1$.
Since $\Gamma_\eta$ is an irreducible component of the hypersurface
$H_\eta\cap \Delta_\eta$ of  $\Delta_\eta$, we have
$ \dim \Gamma_\eta= \dim \Delta_\eta - 1 = \dim X_\eta -1$.

Let $s \in S$ be such that $\Gamma_s $ is not empty. 
Our hypothesis 
that $H$ is a hypersurface 
implies that $\Gamma_s$ does not 
contain any irreducible component of $X_s$ of positive dimension.  
Thus, $\dim\Gamma_s\le \dim X_s-1=d-1$. By \cite{EGA}, IV.13.1.6, 
$\Gamma_s$ is equidimensional of dimension $d-1$. It follows that  $\Gamma\to S$ is equidimensional of dimension $d-1$. 
\qed 

\begin{remark} Let $S$ be an affine integral scheme. The scheme $X:= {\mathbb P}^1_S \sqcup S$ is an $S$-scheme in a natural way, and every irreducible component 
of $X$ dominates $S$. 
Any proper closed subset of $S$ defined by a principal ideal is a hypersurface $H_0$ of $S$. Thus, there exist 
hypersurfaces $H:=H_1 \sqcup H_0$ of $X$ such that 
the irreducible component $H_0$  of $H$ does not dominate $S$. As \ref{components-dom} (a) shows,
this cannot happen when every irreducible component of $X$ has a generic fiber of positive dimension.
\end{remark}

\begin{remark} Let $S$ be a noetherian affine scheme. A variant of Theorem \ref{quasisections} can be obtained
when the morphism $\pi:X\to S$ is only assumed to be quasi-projective, but satisfies the following 
additional condition: There exists a scheme $\overline{X}$ with a projective morphism $\pi': \overline{X} \to S$ 
having all fibers of dimension $d > 0$, 
 and  an open $S$-immersion $X \to \overline{X}$ with dense image and $\dim (\overline{X}\setminus X)<d$. 
Keeping all other hypotheses of Theorem \ref{quasisections}
in place, its conclusions then also hold under the above weaker hypotheses on $\pi:X\to S$. 
The proof of this variant is similar to the proof of Theorem \ref{quasisections}, and consists in applying Theorem  \ref{bertini-type-0} $d$ times, 
starting with 
the data $\overline{X}$, $C$, $F:=\overline{X}\setminus X$, and the finite set $A$ containing the generic points of $F$. 
\end{remark}

\begin{remark} Let $S$ be  an  affine scheme, and let $ X \to S$ be a {\it smooth}, projective, and surjective, morphism.
We may ask whether $X \to S$ always admits a {\it finite \'etale} quasi-section. 
(The existence of a quasi-finite \'etale quasi-section is proved in \cite{EGA}, IV.17.16.3 (ii).)
The answer to the above question is known in two cases of arithmetic interest.

  First, let $S$ be a smooth affine geometrically irreducible curve over a finite field.
Let $X \to S$ be a  smooth  and surjective morphism, with geometrically irreducible generic fiber.
Then $X/S$ has  a   finite   \'etale quasi-section (\cite{Tam}, Theorem (0.1)).

Let now $S = \Spec {\mathbb Z}$. The answer to this question in this case is negative, 
as examples of K. Buzzard \cite{Buz} show.
Indeed, a positive answer to this question 
over $S = \Spec {\mathbb Z}$ would imply that any smooth, projective, surjective, morphism $X \to \Spec {\mathbb Z}$  has a generic fiber 
which has a ${\mathbb Q}$-rational point. The hypersurface $X/S$ in ${\mathbb P}^7_{\mathbb Z}$ defined by the quadratic form $f(x_1,\dots, x_8)$
associated with the $E_8$-lattice is smooth over $S$ because the determinant of the associated symmetric matrix is $\pm 1$, and the generic fiber
of $X/S$  has no ${\mathbb R}$-points because $f$ is positive definite.

Let $S= \Spec \cO_K$, where $K$ is a number field. Let $L/K$ denote the extension maximal with the property that 
the integral closure $\cO_L$ of $\cO_K$ in $L$ is unramified over $\cO_K$. Does the above question have a positive 
answer if $L/K$ is infinite? Obviously, if it is possible to find such a $K$ and $L$ where $L \subset {\mathbb R}$, then 
the example of Buzzard would still show that the answer is negative. We do not know if examples of such $K$ exist. 
\end{remark}

Some conditions on the dimension of the fibers of a projective morphism 
$X \to S$ 
are indeed necessary for a finite quasi-section to 
exist, as the following proposition shows.

\begin{proposition}\label{fqs-dim} 
Let $X$ and $ S$ be irreducible noetherian  schemes. Let $\pi : X\to S$ be a  
proper morphism, and suppose that $\pi $ has a finite quasi-section 
$T$. 
\begin{enumerate}[\rm (a)]
\item Assume that $\pi: X\to S$ is generically finite. Then $\pi$ is finite. 
\item Assume  that
the generic fiber of $X\to S$ has dimension $1$. If $X$ is regular, then for all 
$s\in S$, $X_s$ has an irreducible component of dimension $1$. 
\end{enumerate}
\end{proposition}

\proof (a) Since $\pi$ is generically finite and $X$ is irreducible, the generic fiber 
of $X\to S$ is reduced to one point, namely, the generic point of $X$. 
Since $T\to S$ is surjective, $T$ meets the generic fiber of $X\to S$, 
and so it contains the generic point of $X$. Thus, $T=X$ set-theoretically. 
Since $X_{\mathrm{red}}\subseteq T$, we find that is $X_{\mathrm{red}}$ is finite over $S$.
Since $X$ is then quasi-finite and proper, 
it is finite over $S$. 

(b) Let $\Gamma$ be an irreducible component of $T$ which surjects onto $S$. 
Let us first show that $\codim(\Gamma, X)=1$. 
Let $Y$ be an irreducible closed subset 
of $X$ of codimension $1$ 
which contains $\Gamma$. Since the generic fiber of $X\to S$ has dimension $1$,
the generic fibers of  $\Gamma \to S$ and $Y \to S$ are both irreducible 
and $0$-dimensional. Hence, these generic fibers are equal. Therefore, $\Gamma=Y$ 
and $\Gamma$ has codimension $1$ in $X$. 
Since $X$ is regular, $\Gamma$ is then the support of a Cartier divisor on $X$. 
By hypothesis, for all $s \in S$, $\Gamma_s$ is not empty, and has dimension $0$.
Thus, for all $s\in S$ and all $t\in \Gamma_s$, we have 
$0=\dim_t \Gamma_s\ge \dim_t X_s -1$.   
It follows that the irreducible components of $X_s$ which intersect 
$\Gamma_s$ all have dimension at most $ 1$. Since every irreducible component of $X_s$ has dimension
at least $1$ (\cite{EGA}, IV.13.1.1), (b) follows. 
\qed 

\begin{example} As the following example shows, it is not true in general
in Proposition~\ref{fqs-dim} (b) that for all $s \in S$, all irreducible components of $X_s$ have dimension $1$.
Let $S$ be regular of dimension $d\ge 2$. 
Fix a section $T$ of $\mathbb P^1_S\to S$. 
Let $x_0$ be a closed point of $\mathbb P^1_S$ not contained in $T$ and lying over a point
$s\in S$ with $\dim_{s} S=d$. 
Let $X\to \mathbb P^1_S$ be the blowing-up of $x_0$. Then $X$ is regular, and $X\to S$ has 
the preimage of $T$ as a section (and thus it has a finite quasi-section). 
However, $X_{s}$ consists of the union of a projective line and
the exceptional divisor $E$ of $X\to \mathbb P^1_S$, which has dimension $d$. So
$\dim X_{s}=d\ge 2$. 

Our next example shows that some regularity assumption on $X$ is necessary in
\ref{fqs-dim} (b).
Let $k$ be any field, $R:=k[ t_1, t_2]$, and $B:= R[u_0, u_1, u_2]/(t_1u_2-t_2u_1)$.
Consider the induced projective morphism
$$X:=\Proj ( B) \longrightarrow S:=  \Spec R = \mathbb A^2_{k}.$$ 
The scheme $X$ is singular at the point $P$ corresponding to the homogeneous ideal $(t_1,t_2,u_1,u_2)$ of $B$.
The fibers of $X \to S$ are isomorphic to $\mathbb P^1_{k(s)}$ if $s \neq (0,0)$.
When $s = (0,0)$, then $X_s$ is isomorphic to $\mathbb P^2_{k(s)}$.
The morphism $X \to S$ has a finite section $T$, corresponding to the homogeneous ideal
$(u_1,u_2)$. As expected in view of the proof of \ref{fqs-dim} (b), any section of $X\to S$, and in particular the section $T$, 
contains the singular point $P$.
\end{example}

We conclude this section with two applications of Theorem \ref{quasisections}.

\begin{proposition} \label{splitting} Let $A$ be a 
commutative ring. Let $M$ be a 
projective $A$-module of finite presentation with constant rank $r> 1$.
Then there exists an $A$-algebra $B$, finite and faithfully flat over $A$, with $B$ a local complete intersection over $A$,
such that $M\otimes_A B $ 
is isomorphic to a direct sum of projective $B$-modules of rank $1$.
\end{proposition}

\proof 
 Let $S := \Spec A$. 
Let ${\mathcal M}$ denote the locally free $\cO_S$-module of rank $r$ associated with $M$.
Let $X := {\mathbb P}({\mathcal M})$. Then the natural map $X \to S$ is projective, smooth, and its fibers all have dimension $r-1$.
We are thus in a position to apply Theorem \ref{quasisections} (3) to obtain the existence of a finite flat quasi-section $f:T \to S$ as in \ref{quasisections} (3). 
In particular, $T=\Spec B$ for some finite and faithfully flat $A$-algebra $B$, with $B$ a local complete intersection over $A$. 
Moreover, the existence of an $S$-morphism $g:T \to X$ corresponds to the existence of an $\cO_T$-invertible sheaf ${\mathcal L}_1$
and of a surjective morphism $f^*{\mathcal M} \to {\mathcal L}_1$. Let ${\mathcal M}_1$ denote the kernel of this morphism.
The $\cO_T$-module ${\mathcal M}_1$ is locally free of rank $r-1$, and $f^*{\mathcal M} \cong {\mathcal L}_1 \oplus {\mathcal M}_1$. We may thus proceed as above and use Theorem \ref{quasisections} (3)
another $r-2$ times to obtain the conclusion of the corollary. \qed

\begin{remark} The proposition strengthens,  in the affine case, the classical splitting lemma for vector bundles
(\cite{F-L}, V.2.7). When 
$A$ is of finite type over an algebraically closed field $k$ and is regular, 
it is shown in \cite{Sum}, 3.1, that it is possible to find a finite faithfully flat {\it regular} $A$-algebra $B$ over which $M$ splits.

We provide now an example of a commutative ring $A$ with a finitely generated projective module $M$
which is not free and such that it is not possible to find a finite \emph{\'etale} $A$-algebra $B$ which splits $M$ into a
direct sum of rank $1$ projective modules. For this, we exhibit a ring $A$
such that the \'etale fundamental group of $\Spec A$ is trivial and such that $\Pic(A) = (0)$.
Then, if a projective module $M$ of finite rank is split over a finite \emph{\'etale} $A$-algebra $B$, it must be split over $A$.  Since 
$\Pic(A)=(0)$, we find then that $M$ is a free module. 
Let $n >2$ and consider the algebra
$$A:=\mathbb C[x_1,\dots, x_{2n}]/(x_1^2+ \dots + x_{2n}^2 -1).$$
This ring is regular, and it is well-known that it is a UFD, so that $\Pic(A)=(0)$
(see, e.g., \cite{Swan}, Theorem 5). It is shown in \cite{ST}, Theorem 3.1 (use $p=2$), that
for each $n>2$, there exists a projective module $M$ of 
rank $n-1$ which is not free. Let now $X:=\Spec A$. The \'etale fundamental group of $X$ is trivial 
if the topological fundamental group of $X({\mathbb C})$ is trivial (use \cite{SGA1}, XII, Corollaire 5.2).
The topological fundamental group of $X({\mathbb C})$ is trivial because there exists a retraction $X({\mathbb C}) \to S^{2n-1}$,
where $S^{2n-1}$ is the real sphere in ${\mathbb R}^{2n}$ given by the equation $x_1^2+ \dots + x_{2n}^2 =1$ (see, e.g., \cite{Wood}, section 2).
It is well-known that the fundamental group of $S^{2n-1}$ is trivial for all $n \geq 2$. Hence, the module $M$ cannot be split after a finite \'etale base change.
\end{remark}

Let $S$ be a scheme and let $U\subseteq S$ be an open subset. 
Given a family  $C \to U$ of stable curves over $U$, conditions are known (see, e.g.,  \cite{dJ-O})
to insure that  this family extends to a family of stable curves over $S$.
It is natural to consider the analogous problem of extending a given family $D \to Z$ of stable curves over a closed subset $Z$ of $S$. 
For this, we may use the existence of finite quasi-sections in
appropriate moduli spaces, as in the proposition below. 

Let $\overline{\mathcal M}:=\overline{\mathcal M}_{g,S}$ be 
the proper Deligne-Mumford stack of 
stable curves of genus $g$ over $S$ (see \cite{D-M}, 5.1).
Our next proposition uses the following statement: {\it Over $S=\Spec \mathbb Z$, the stack $\overline{\mathcal M}_{g,S}$ admits a coarse  moduli space
$\overline{M}_{g,\mathbb Z}$ 
which is a projective scheme over $\Spec \mathbb Z$}.  
Such a statement is found
in an appendix in GIT \cite{Mum}, page 228, with a sketch of proof. See also \cite{Kol}, 5.1, for another brief proof.

\begin{proposition} \label{extension.stable.curve}
Let $S$ be a 
noetherian affine scheme. 
Let $Z$ be a
closed subscheme of $S$,
and let $D\to Z$ be a stable
curve of genus $g\ge 2$. Then there exist a finite surjective 
morphism $S'\to S$  mapping each irreducible component
of $S'$ onto an irreducible component of $S$, a finite surjective 
morphism $Z' \to Z$, a closed $S$-immersion $Z' \to S'$,
and a stable curve $\mathcal D\to S'$ of genus $g$  
with a morphism $D\times_Z Z' \to \mathcal D$ 
such that 
the diagram below commutes and the top 
square in the diagram is cartesian: 
$$
\xymatrix{
D\times_Z Z' \ar[d] \ar@{^{(}->}[r] & {\mathcal D} \ar[d] \\
Z' \ar@{->>}[d] \ar@{^{(}->}[r] & S' \ar@{->>}[d]\\
Z \ar@{^{(}->}[r] & S. \\
}
$$ 
\end{proposition}

\proof 
Let $\overline{\mathcal M}:=\overline{\mathcal M}_{g,S}$ be 
the proper Deligne-Mumford stack of 
stable curves of genus $g$ over $S$ (see \cite{D-M}, 5.1). 
We first construct a finite surjective morphism 
$X\to \overline{\mathcal M}$ such that $X$ is a scheme, projective
over $S$ and with constant 
fiber dimensions over $S$. 
It is known that over $\mathbb Z$, the coarse  moduli space
$\overline{M}_{g,\mathbb Z}$ 
of $\overline{\mathcal M}$ is a projective scheme
and that its fibers over $\Spec {\mathbb Z}$  are all 
geometrically irreducible of  the same
dimension $3g-3$. Let 
$\overline{M}:=\overline{M}_{g,\mathbb Z}\times_{\Spec{\mathbb Z}} S$.
Then we have a canonical morphism 
$\overline{\mathcal M}\to \overline{M}$ which is proper and 
a universal homeomorphism 
(hence quasi-finite). 
By construction, the
$S$-scheme $\overline{M}$ is projective with constant fiber
dimension. 

Since  $\overline{\mathcal M}$ is a noetherian 
separated Deligne-Mumford stack, 
there exists a (representable) finite surjective morphism from a scheme 
$X$ to $\overline{\mathcal M}$ (\cite{LM}, 16.6).  
The composition $X \to\overline{\mathcal M}\to \overline{M}$ is a 
finite (because proper and 
quasi-finite) surjective morphism of schemes.  
Thus $X\to S$ is projective since $S$ is affine and 
$\overline{M}\to S$ is projective. 
So $X\to S$ is projective and all its fibers have the same dimension. 

The curve $D\to Z$ corresponds to an element in 
the set $\overline{\mathcal M}(Z)$, which in turn corresponds to a 
finite morphism $Z\to \overline{\mathcal M}$. 
So $Z':=Z\times_{\overline{\mathcal M}} X$ is a scheme, finite surjective over $Z$ and 
finite over $X$. Let $Z_0$ denote the schematic image of $Z'$ 
in $X$. It is finite over $S$.  

To be able to apply Theorem \ref{quasisections} (1), 
we note the following. 
Let $T$ be  the disjoint union of the
reduced irreducible components of $S$. 
Replacing if necessary $S$ with $T$ and $D\to Z$ with
$D\times_S T\to Z\times_S T$,  we easily reduce the proof 
of the proposition to the case where
$S$ is irreducible.  Once $S$ is assumed irreducible, we use the fact that $\overline{M}\to S$ 
is proper with irreducible fibers to find that $\overline{M}$ is also irreducible. Replacing
$X$ by an irreducible component of $X$ which dominates  $\overline{M}$, we can
suppose that $X$ is irreducible.

Theorem \ref{quasisections} (1) can then be applied to the morphism $X \to S$ 
 and to  
each irreducible component of $Z_0$. We obtain 
a finite quasi-section $S_0$ of $X/S$ containing (set-theoretically) 
$Z_0$ and such that each irreducible component of $S_0$ maps onto $S$.
Modifying the structure of closed subscheme on $S_0$ as in the 
proof of \ref{quasisections} (1), we can suppose that $Z_0$ is 
a subscheme of $S_0$.   

Because $S_0$ is affine, it is clear that there exists 
a scheme $S'$, finite and faithfully flat (and even l.c.i.) over $S_0$, 
and a closed immersion $Z' \to S'$ making the following diagram commute:
$$
\xymatrix{
Z'  \ar@{->>}[d] \ar@{^{(}->}[r] & S' \ar@{->>}[d] & \\
Z_0 \ar@{^{(}->}[r] & S_0 .\\
}
$$
As $S'\to S_0$ is flat,  each irreducible component of $S'$
maps onto an irreducible component of $S_0$, hence onto $S$. 

The stable curve $\mathcal D\to S'$ whose existence is asserted in the statement of Proposition \ref{extension.stable.curve} corresponds to 
the element of $\overline{\mathcal M}(S') $ given by the composition of the finite 
morphisms $S' \to S_0\to X\to \overline{\mathcal M}$. 
\qed

\begin{remark}\label{smooth_cover_mg} 
Consider the finite surjective $S$-morphism $X \to \overline{\mathcal M}$ introduced at the beginning of
the proof of \ref{extension.stable.curve} above.
If we can find such a  cover 
$X\to \overline{\mathcal M}$ such that $X\to S$ is flat with Cohen-Macaulay fibers
(resp., with l.c.i.\ fibers), then using Theorem~\ref{quasisections}\ (2) and (3),  we can further require in the statement of Proposition 
\ref{extension.stable.curve} 
that $S'\to S$ be finite and faithfully flat (resp., l.c.i.).

When some prime number $p$ is invertible in $\cO_S(S)$, then it is proved
in \cite{dJ-P}, 2.3.6.(1) and 2.3.7, that there exists  
such an $X$ which is even smooth over $S$. Therefore, in this case,
we can find a morphism $S'\to S$ which is finite,
faithfully flat, and l.c.i. 
\end{remark}
\end{section}

\begin{section}{Moving lemma for 1-cycles}
\label{mv-1c}

We review below the basic notation needed to state our moving lemma.
Let $X$ be a noetherian scheme. Let ${\mathcal Z}(X)$ denote the free 
abelian group on the set of closed integral subschemes of $X$.
An element of ${\mathcal Z}(X)$ is called a \emph{cycle}, and if $Y$ is an integral closed subscheme of $X$,
we denote by $[Y]$ the associated element in ${\mathcal Z}(X)$.

Let ${\mathcal K}_X$ denote the sheaf of 
meromorphic functions 
on 
$X$ (see \cite{K}, top of page 204 or 
\cite{Liubook}, Definition 7.1.13). Let $f\in \mathcal K_X^*(X)$. 
Its associated  
principal Cartier divisor is denoted by $\dv(f)$ and defines a cycle on $X$:
$$[\dv(f)]=\sum_{x} \ord_x(f_x)[\overline{\{x\}}]$$ 
where $x$ ranges through the points of codimension $1$ in $X$,
and $\ord_x: \mathcal K_{X,x}^* \to \mathbb Z$ is defined, for a regular element of $g \in \cO_{X,x}$,
to be the length of the $\cO_{X,x}$-module $\cO_{X,x}/(g)$. 
 
A cycle $Z$ is \emph{rationally equivalent to $0$} 
or \emph{rationally trivial}, 
if there are finitely many integral closed subschemes $Y_i$ and
non-zero rational functions $f_i$ on $Y_i$
such that 
$Z=\sum_i [\dv(f_i)]$. Two cycles $Z$ and $Z'$ are 
\emph{rationally equivalent} in $X$ if $Z-Z'$ is rationally 
equivalent to $0$. 
We denote by $\CH(X)$ the quotient of $\mathcal Z(X)$ by
the subgroup of rationally trivial cycles. 

A 
morphism of schemes of finite type $\pi : X\to Y$ induces by {\it push forward of cycles} 
a group homomorphism $\pi_*: {\mathcal Z}(X) \to
{\mathcal Z}(Y)$. If $Z $ is any closed integral subscheme of $X$, then 
$\pi_*([Z]):= [k(Z):k(\overline{\pi(Z)})] [\overline{\pi(Z)}]$, with the convention 
that  $[k(Z):k(\overline{\pi(Z)})]=0$ if the extension $k(Z)/k(\overline{\pi(Z)})$ is not 
finite.

\begin{emp} \label{emp.proper}
Let $S$ be a noetherian scheme which is universally catenary and equidimensional at every point (for instance, 
$S$ is regular). Assume that both $X\to S$ and $Y\to S$ are morphisms of finite type,
and let 
$\pi:X\to Y$ be a {\it proper} morphism of $S$-schemes. Let $C$ and $C'$ be 
two cycles on $X$ which are rationally equivalent. 
Then $\pi_*(C)$ and $\pi_*(C')$ are rationally equivalent on $Y$
 (\cite{Th}, Note 6.7, or Proposition 6.5 and  3.11). We denote by $\pi_*: \CH(X) \to \CH(Y)$ the induced morphism.
 For an example showing that the hypotheses on $S$ are needed for $\pi_*: \CH(X) \to \CH(Y)$ to be well-defined, 
 see \cite{GLL1}, 1.3.
\end{emp}

We are now ready to state the main  theorem  
of this section.
Recall that the support of a horizontal $1$-cycle $C$ in a scheme $X$ over 
a Dedekind scheme $S$
is 
a finite quasi-section (\ref{ConditionT*}).
 The definitions of {\it Condition} (T) and of {\it pictorsion} are given 
in {\rm (\ref{ConditionT})} and {\rm (\ref{ConditionT*})}, respectively.

\begin{theorem} \label{mv-1-cycle-local} 
Let $R$ be a Dedekind domain, and let $S:=\Spec R$. 
Let $X \to S$ be a flat and quasi-projective morphism, with $X$
integral. 
Let $C$ be a horizontal $1$-cycle on $X$. 
Let $F$ be a 
closed subset of $X$.
Assume that for all $s \in S$, $F\cap X_s$ and $\Supp(C) \cap X_s$ have
positive codimension in $X_s$. 
Assume in addition that
either
\begin{enumerate}[\rm (a)]
\item $R$ is pictorsion and the support of $C$ is contained in the regular locus of $X$, or
\item $R$ satisfies Condition \emph{(T)}.  
\end{enumerate}
Then some positive multiple $mC$ of 
$C$ is rationally equivalent
to a horizontal $1$-cycle $C'$ on $X$ whose support does not meet
$F$.
Under the assumption {\rm (a)}, if furthermore $R$ is  
semi-local, then we can take $m=1$.

Moreover, if $Y \to S$ is any separated morphism of finite type and $h: X \to Y$
is any $S$-morphism, then $h_*(mC)$ is rationally equivalent to $h_*(C')$ on $Y$.
\end{theorem}

The proof of Theorem \ref{mv-1-cycle-local} is postponed to 
\ref{mv-1-cycle-local.Proof}. We first briefly introduce below needed facts about contraction morphisms.
We then discuss several statements needed in the proof of \ref{mv-1-cycle-local} (b) when $S$ is
not excellent.

\begin{proposition} \label{contraction}
Let $R$ be a Dedekind domain, 
and $S:=\Spec R$.
Let
$X\to S$ be a projective morphism of relative dimension $1$, 
with $X$ integral. Let $C$ be an effective 
Cartier divisor on $X$, flat over $S$. Then 
\begin{enumerate}[\rm (a)]
\item There exists $m_0 \geq 0$ such that the invertible
sheaf $\cO_X(mC)$ is generated by its global sections
for all $m\geq m_0 $. 
\item The morphism
$X':=\Proj \left(\oplus_{m\ge 0} H^0(X, \cO_X(mC))\right) \longrightarrow S$
is  projective, with $X'$ integral, and the canonical 
morphism $u : X\to X'$ is projective,  with
$u_*\cO_X=\cO_{X'}$ and connected fibers.
\item
For any vertical prime divisor 
$\Gamma$ on $X$, $u|_{\Gamma}$ is constant if 
$\Gamma\cap \Supp C=\emptyset$, and is finite otherwise.
\item Let $Z$ be the
union of the vertical prime divisors of $X$ disjoint from $\Supp C$.
Then $u$ induces an isomorphism $ X\setminus Z\to X'\setminus u(Z)$.\end{enumerate} 
\end{proposition}
\proof In \cite{BLR}, Theorem 1 in 6.7, a similar statement is proved, with 
$R$  local, and $X$   normal. (The normality is not assumed in \cite{Em} and \cite{Pie}.
A global base is considered in \cite{Liubook}, 8.3.30.) We leave it to the reader to check that the proof
of  \cite{BLR}, 6.7/1,  can be used {\it mutatis mutandis} to prove \ref{contraction}.
Part (a) follows from  the first part of the proof of  6.7/1. Part (b) follows from 6.7/2. 
Part (c) follows from the second part of the proof of 6.7/1. We now give a proof of (d).
The morphism $u$ is birational because it induces 
an isomorphism $X_{\eta} \to X'_{\eta}$ over the generic point $\eta$ of $S$, 
since $C_\eta$ is ample, being effective of positive degree.
It follows that $Z$ is the union of finitely many prime divisors of $X$. 
As $u$ has connected fibers, it follows from (c) that $Z=u^{-1}(u(Z))$.
The restriction $v : X\setminus Z\to X'\setminus u(Z)$ of $u$ is thus
projective and  quasi-finite. Therefore, $v$ is finite and, hence, affine. 
As $\cO_{X'\setminus \pi(Z)}=v_*\cO_{X\setminus Z}$, $v$ is an isomorphism.
\qed 
\medskip 

Let $K$ be a field of characteristic $p>0$. Let $K':=K^{p^{-\infty}}$ be the perfect closure of $K$.
Let $n \geq 0$ and set $q:= p^n$. Let $K^{1/q}$ denote the extension of $K$ in $K'$ generated by the $q$-th roots 
of all elements of $K$. Let $i: K \to K^{1/q}$ denote the natural inclusion, and let 
$\rho: K^{1/q}\to K$ be defined by   
$\lambda\mapsto \lambda^{q}$. The composition $F:= \rho \circ i: K \to K$ is the $q$-th Frobenius morphism of $K$. 
By definition, given a morphism $Y \to \Spec K$, the morphism $Y^{(q)} \to \Spec K$ is the base change $(Y \times_{\Spec K, F^*} \Spec K) \to \Spec K$. It follows that we have a natural isomorphism of $K$-schemes:
\begin{equation} \label{Frob}
 Y^{(q)}\simeq (Y\times_{\Spec K, i^*} {\Spec (K^{1/q})})\times_{\Spec (K^{1/q}), \rho^*}\Spec K.
\end{equation} 
\begin{lemma}\label{normalization-smooth} 
Let $K$ be a field of characteristic $p>0$.
Let $Y \to \Spec K$ be a morphism of finite type, with  $Y$  integral  of dimension 
$1$. Then there exists $n\geq 0$ such that 
the normalization of $(Y^{(p^n)})_{\mathrm{red}}$ is smooth over $K$. 
\end{lemma} 
\begin{proof}  
The normalization 
$Z$ of $(Y_{K'})_{\mathrm{red}}$ is regular and, hence, smooth over the perfect closure $K'$. There 
is a finite sub-extension $L/K$ of $K'$ such that the curve $Z$ and the morphism 
$Z\to (Y_{K'})_{\mathrm{red}}$ are defined over $L$. 
This implies that the normalization of $(Y_L)_{\mathrm{red}}$ is 
$Z_{/L}$, hence smooth over $L$. Let $q=p^n$ be such that $L\subseteq
K^{1/q}$. As $Z_{/L}\to Y_L$ is finite and induces an isomorphism on the
residue fields at the generic points, the same is true for 
$$Z_{K^{1/q}}\longrightarrow (Y_L)_{K^{1/q}}=Y_{K^{1/q}}.$$ 
Using $\rho: K^{1/q}\to K$ and \eqref{Frob}, 
$$(Z_{K^{1/q}})_K \longrightarrow (Y_{K^{1/q}})_K\simeq Y^{(q)}$$ 
is finite and  induces an isomorphism on the
residue fields at the generic points. As the left-hand side is smooth, this
morphism is the normalization of $(Y^{(q)})_{\mathrm{red}}$. 
\end{proof} 

\begin{lemma}\label{Chow-radicial} Let $S$ be a universally catenary 
noetherian scheme which is equidimensional at every point. 
Let $\pi: X\to X_0$ be a finite surjective morphism of $S$-schemes of 
finite type, with induced homomorphism of Chow groups 
$\pi_* : \CH(X)\to  \CH(X_0)$. Then
\begin{enumerate}[{\rm (1)}] 
\item The cokernel of $\pi_*$ is a torsion group. 
\item If $\pi$ is a homeomorphism, then the kernel of $\pi_*$ is
also a torsion group. 
\end{enumerate}
\end{lemma}

\begin{proof} 
Our hypotheses on $S$ allow us to use \ref{emp.proper}, so that the morphism $\pi_* : \CH(X)\to  \CH(X_0)$
is well-defined.

(1) Let $Z_0$ be an integral closed subscheme on $X_0$, and let $Z$ be an irreducible component 
of $\pi^{-1}(Z_0)$ whose image in $X_0$ is $Z_0$. When $Z$ is endowed with the reduced induced structure, $Z \to Z_0$ is finite and surjective, 
and $\pi_*[Z]=[k(Z): k(Z_0)][Z_0]$. 
Hence, the cokernel of $\mathcal Z(X)\to \mathcal
Z(X_0)$ is torsion, and the same holds  for the corresponding homomorphism of Chow 
groups. 

(2) Let $W_0$ be an integral closed subscheme of $X_0$. Since $\pi$ is a homeomorphism,
$W:=\pi^{-1}(W_0)$ is irreducible, and we endow it with the reduced induced structure. 
The induced morphism $\pi: W \to W_0$ is finite and surjective between integral noetherian schemes. Let 
$f\in k(W_0)$ be a non-zero rational function.  Using for instance \cite{Liubook}, 7.1.38, we find that
$$\pi_*([\mathrm{div}_W(\pi^*f)])=[k(W): k(W_0)][\mathrm{div}_{W_0}(f)]. 
$$  
This implies that for every integer multiple $r $ of $[k(W): k(W_0)]$, 
$r[\mathrm{div}_{W_0}(f)]=\pi_*(D_r)$ for some principal
cycle $D_r$ on $X$. 

Now let $Z$ be any cycle on $X$ such that 
$\pi_*Z$ is principal on $X_0$. Then 
for a suitable integer $N$, $N\pi_*Z=\pi_*(D)$ for some principal 
cycle $D$ on $X$.  Since $\pi$ is a homeomorphism,   $\pi_* : \mathcal Z(X)\to \mathcal Z(X_0)$
is injective. Therefore, $NZ=D$ in $\mathcal Z(X)$, and the class of $NZ$ is trivial in $\mathcal A(X)$. 
\end{proof}

For our next proposition, recall that 
a normal scheme $X$ is called {\it ${\mathbb Q}$-factorial} if every Weil divisor $D$ on $X$ is such that some 
positive integer multiple of $D$ is the cycle associated with a Cartier divisor on $X$. 
 
\begin{proposition}\label{normalization-finite} 
Let $S$ be a Dedekind scheme with generic point $\eta$. 
Let $X\to S$ be a dominant morphism of finite type, with $X$ integral.
Suppose that the normalization of $X_\eta$ is smooth over $k(\eta)$. 
Then  \begin{enumerate}[\rm (a)]
\item The normalization morphism $\pi: X'\to X$ is finite. 
\item If $X$ is normal, then the following properties are true. 
\begin{enumerate}[{\rm (1)}] 
\item The completion $\widehat{\cO}_{X,x}$ is normal for all $x\in X$. 
\item The locus $\mathrm{Reg}(X)$ of regular points of $X$ is open in $X$. 
\item If $\dim X_\eta=1$ and $S$ 
satisfies {\rm Condition (T)}, then $X$ is $\mathbb Q$-factorial. 
\end{enumerate}
\end{enumerate}
\end{proposition}

\begin{proof} When $S$ is assumed to be excellent, then $X$ is also excellent and most of the statements in the proposition
follow from this property. The statement \ref{normalization-finite}(b)(3) can be found in \cite{MB1}, 3.3.
We now give a proof of \ref{normalization-finite} without assuming that $S$ is excellent.

(a) We can and will assume that $X$ is affine. 
As $\pi_\eta: X'_\eta\to X_\eta$ is finite, 
there exists a factorization $X'\to X''\to X$ with 
$X''\to X$ finite and birational, and such that 
$X'_\eta\to X''_\eta$ is an isomorphism (simply 
take generators of $\cO_{X'_{\eta}}(X'_\eta)$ which belong to  $\cO_{X'}(X')$).
Replacing $X$ with $X''$, we can suppose that $X_\eta$ is smooth. 
The smooth locus of $X\to S$ is open and contains $X_\eta$, 
so it contains an open set of the form $X_V:=X\times_S V$ for some dense open subset 
$V$ of $S$. So $X'_V=X_V$ and we find that $\pi_*\cO_{X'}/\cO_X$ is 
supported on finitely many closed fibers 
$X_{s_1},\dots, X_{s_n}$.  

To show that $\pi$ is finite, it is enough to show 
that the normalization morphism of $X\times_S \Spec(\cO_{S,s_i})$ is finite for
all $s_i$. Therefore, we can suppose that $S=\Spec R$ for some discrete valuation ring
$R$. Let $\widehat{R}$ be the completion of $R$. As $X_\eta$ is smooth, 
the normalization morphism $\widehat{\pi}: (X_{\widehat{R}})'\to X_{\widehat{R}}$ 
is an isomorphism on the generic fiber. It is finite because $\widehat{R}$
is excellent. By \cite{Liubook}, 8.3.47 and 8.3.48, 
$\widehat{\pi}$ descends to 
a finite morphism $Z\to X$ over $R$. By faithfully flat descent, this
implies that $Z$ is normal and, thus, isomorphic to $X'$, and that $X'\to X$ is
finite and $X'_{\widehat{R}}=(X_{\widehat{R}})'$ is normal.

(b) Suppose now that $X$ is normal with smooth generic fiber. 
 To prove (1), let $x\in X$ with image $s\in S$. Then $\cO_{X,x}$ is also the 
local ring of $X\times_S\Spec\cO_{S,s}$ at $x$. To prove that its
completion is normal, we can thus suppose that $S$ is local. We can
even restrict to $s$ closed in $S$ as $X_V$ is regular. 
Let $R=\cO_{S,s}$. We saw above that $X_{\widehat{R}}$ is normal. 
As $\widehat{\cO}_{X,x}$ is 
also the completion of $\cO_{X_{\widehat{R}},x}$ 
(see, e.g., \cite{Liubook}, 8.3.49(b)), it is normal because
$X_{\widehat{R}}$ is excellent (\cite{EGA}, IV.7.8.3 (vii)). 

(2) We have $\mathrm{Reg}(X)\supseteq X_V$ and $\mathrm{Reg}(X)\cap
X_s=\mathrm{Reg}(X\times_S \Spec\cO_{S,s})$ for all $s\in S\setminus
V$. As $S\setminus V$ consists of finitely many closed points of
$S$, $\mathrm{Reg}(X)$ is open by \cite{EGA}, IV.6.12.6 (ii). 

(3) The statement of (3) is proved in \cite{MB1}, Lemme 3.3, provided that the 
singular points of $X$ are isolated,
and that \cite{MB1}, Th\'eor\`eme 2.8, holds when $A=\cO_{X,x}$. 
In our case, the singular points of $X$ are isolated by (2).
Th\'eor\`eme 2.8 in \cite{MB1} is proved under the hypothesis that  $A$ is excellent, but the proof in \cite{MB1} only 
uses the fact that the completion of $A$ is normal (in step 2.10). So in our case,  this property is satisfied by (1). 
\end{proof}

\begin{emp} \label{mv-1-cycle-local.Proof}
{\it Proof of Theorem {\rm \ref{mv-1-cycle-local}} when   {\rm (a)} holds.}  
It suffices to prove the theorem in the case where
the given $1$-cycle is the cycle associated with an integral closed subscheme of $X$ finite over $S$. 
We will denote again by $C $ this integral closed subscheme.
As in the proof of Theorem 2.3  
 in \cite{GLL1}, 
we reduce the proof of \ref{mv-1-cycle-local} to the case where $C \to X$ is 
a regular immersion\footnote{The hypothesis that $C\to X$ is a regular immersion
is equivalent to the condition that  $C \to X$ is a local complete intersection morphism
(see, e.g., \cite{Liubook}, 6.3.21).} 
as follows.

Proposition 3.2 in \cite{GLL1}  shows the existence of a finite birational
morphism $D\to C$ such that the composition $D\to C \to S$ is an l.c.i.\ 
morphism. Since $C$ is affine, 
there exists for some $N \in {\mathbb N}$ a closed immersion 
$D \to  C\times_S \mathbb P^N_S\subseteq X\times_S \mathbb P^N_S$.
Note that since $C$ is contained in the regular locus of $X$, 
then $D$ is contained in the regular locus of $X\times_S \mathbb P^N_S$.
We claim that it suffices to prove  the theorem 
for the $1$-cycle $D$ and the closed subset ${\bf F}:= F \times_S \mathbb P^N_S$ in the 
scheme $X\times_S \mathbb P^N_S$. 
Indeed, let  $D'$ be a horizontal $1$-cycle whose existence is asserted by the theorem in this case, with $mD$ rationally equivalent to $D'$. In particular, $\Supp(D') \cap {\bf F} = \emptyset$.
Consider the projection $p : X\times_S \mathbb P^N_S\to X$, 
which is a projective morphism. 
Then $p_*(D) = C$ because $D\to C$ is birational. 
It follows from \ref{emp.proper} that $mC=p_*(mD)$ 
is rationally equivalent to the horizontal $1$-cycle 
$C':= p_*(D')$ on $X$. Moreover, $\Supp(C') \cap F = \emptyset$.
Since $D/S$ is l.c.i., each local ring $\cO_{D,x}$, $x \in D$,
is an absolute complete intersection ring, and 
the closed immersion $D \to  X\times_S \mathbb P^N_S$ is a regular immersion 
(\cite{EGA}, IV.19.3.2). 
Finally, consider a morphism $h:X \to Y$ as in the last statement of the theorem.
Apply this statement to $mD$, $D'$, and to the associated morphism $h':X\times_S \mathbb P^N_S \to Y\times_S \mathbb P^N_S$.
Since the projection $Y\times_S \mathbb P^N_S \to Y $ is proper, we find as desired that $h_*(mC)$ is rationally equivalent 
to $h_*(C')$ on $Y$.

Let us now assume that $C \to X$ is a regular immersion.
Let $d$ denote the codimension of $C$ in $X$. 
If $d>1$, we can apply Theorem \ref{pro.reductiondimension2} (as stated in the introduction since $C$ is integral)
and obtain a closed subscheme $Y$ of $X$ such that $C$ is the support of
a Cartier divisor on $Y$ and such that $F \cap Y_s$ is finite for all $s\in S$.
Clearly, $C$ is also the support of a Cartier divisor on 
$Y_{\red}$, and on any irreducible component of $Y_{\red}$ passing
through $C$. Thus, we are reduced to proving the theorem when $X$ 
is integral of dimension $2$ and $F$ is quasi-finite over $S$.
Note that after this reduction process, we cannot and do not assume anymore 
that $C$ is contained in the regular locus of $X$.

When $d=1$, we do not apply \ref{pro.reductiondimension2}, but we note that in this case too $F \cap X_s$ is finite for all $s\in S$. Indeed, since $C\to S$ is finite, the generic point of $C$ is a closed
point in the generic fiber $X_{\eta} $ of $X \to S$. Since the codimension of $C$ in $X$ is $d=1$, and since the generic fiber
is a scheme of finite type over a field, we find that one irreducible component of $X_{\eta} $ has dimension $1$.
Since $S$ is a Dedekind scheme and $X\to S$ is flat with $X$ integral, we find that all fibers are equidimensional 
of dimension $\dim X_{\eta}=1$ (\cite{Liubook}, 4.4.16).
Hence, our hypothesis on $F$ implies that  
$F \to S$ is quasi-finite.
\end{emp}

\begin{emp} \label{endofproof}
Since $X/S$ is quasi-projective and $X$ is integral, there exists an integral  scheme  $\overline{X}$ with a projective morphism $\overline{X} \to S$
and an $S$-morphism $X \to \overline{X}$ which is an open immersion.
Let $\overline{F}$ be the Zariski closure of $F$ in $\overline{X}$. 
The closed subscheme  $\overline{F}$ is finite over $S$ because $F \to S$ is quasi-finite and $S$ has dimension $1$. Recall that by definition, a horizontal $1$-cycle on $X$ is 
finite over $S$. Hence,  $C$ is closed in $\overline{X}$. Since $C$ is the support of a Cartier divisor on $X$, 
we find that $C$ is also the support of a Cartier divisor
on $\overline{X}$.  
We are thus in a situation where we can consider  the contraction morphism  $u : \overline{X}\to X'$ associated to 
${C}$ in \ref{contraction}. Let $Z$ denote the union of the irreducible 
components $E$ of the fibers of $\overline{X}\to S$ such that 
$E \cap \Supp({C}) = \emptyset$. Let $U= X\setminus (Z\cap X)$. 
Then  $\Supp C \subseteq U$,
and $u|_U$ is an isomorphism onto its image. 
Let $F'=u(\overline{F}\cup Z\cup (\overline{X}\setminus X))\cup u(\Supp(C))$.
Then $X'\setminus F'\subseteq u(U)$, and $F'$ is finite over $S$. 
We endow $F'$ with the structure of a reduced closed subscheme of $X'$. 

Now suppose that $R$ is $\pictorsion$. 
Then $\Pic(F')$ is 
a torsion group by hypothesis.
So, fix $n>0$ such that 
$\cO_{X'}(nC)|_{F'}$ is trivial. 
Since $C$ meets every irreducible component of every fiber of $X'\to S$, 
the sheaf $\cO_{X'}(C)$ is relatively ample for $X'\to S$
(\cite{EGA}, III.4.7.1).
Let $\mathcal I$ denote the ideal sheaf of $F'$ in $X'$. Then there exists 
a multiple $m$ of $n$ such that 
$H^1(X', \mathcal I \otimes \cO_{X'}(mC))=(0) $. 
It follows that a trivialization of $\cO_{X'}(mC)|_{F'}$ 
lifts to a section $f\in H^0(X', \cO_{X'}(mC))$. 

Recall that by definition, $\cO_{X'}(mC)$ is a subsheaf of $\mathcal K_{X'}$.
We thus consider $f\in H^0(X', \cO_{X'}(mC)) \subseteq \mathcal K_{X'}(X')$ as 
a rational function. The support of the divisor ${\rm div}_{X'}(f)+mC$ is  
disjoint from $F'$ by construction. In particular, it is contained in $u(U)$ 
and is horizontal, and ${\rm div}_{X'}(f)$ has also its support contained in $u(U)$. 
Considering the pull-back of the divisors under $\overline{X}\to X'$
shows that the divisor $C':={\rm div}_{\overline{X}}(f)+mC$ is contained in $U$, disjoint from $F$, horizontal and linearly equivalent to $mC$ on $\overline{X}$. 

When $R$ is semi-local, the set $F' \subset X'$ is a finite set of 
points. Thus we may apply Proposition 6.2 of \cite{GLL1}
directly to 
the Cartier divisor whose support is $u(C)$ to find a Cartier divisor 
$D$ linearly equivalent to $u(C)$ and whose support does not meet $F'$.

It remains to prove the last statement of the theorem, which pertains to 
the morphism $h: X \to Y$.
To summarize, in the situation of \ref{mv-1-cycle-local} (a), when $C$ 
is integral, we found a closed integral 
subscheme $W$ of $X$ containing $C$, 
a projective scheme $\overline{W}/S$ containing $W$ as a dense open subset,
and $m\ge 1$ (with $m=1$ when $R$ is semi-local) such that $mC$ is 
rationally equivalent on $\overline{W}$ to some horizontal $1$-cycle $C'$ 
contained in $W$. The morphism $h : X\to Y$  in  the statement of \ref{mv-1-cycle-local} 
induces an $S$-morphism $h : W\to Y$. 
Our proof now proceeds as in \cite{GLL1}, proof of Proposition 2.4(2). For the convenience of the reader, 
we recall the main ideas of that proof here.

Let $g$ be  the function on $W$ such that $[\dv_W(g)] = mC - C'$,
Let $\Gamma\subseteq \overline{W}\times_S Y$ be the schematic closure of the graph 
of the rational map $\overline{W}\dasharrow Y$ induced by $h:W \to Y$. Let 
$p: \Gamma \to \overline{W} $ and $q: \Gamma\to Y$ be the associated projection maps over $S$.
Since $\Gamma$ is integral and its generic point maps to the generic 
point of $W$, 
the rational function $g$ on $W$ induces a rational function, again denoted by $g$, on $\Gamma$.
As $p : p^{-1}(W)\to W$ is an isomorphism,
we let $p^*(C)$ and $p^*(C')$ denote the preimages of $C$ and $C'$ in $p^{-1}(W)$; they are 
closed subschemes of $\Gamma$.
Since $g$ is an  invertible function in a neighborhood of $\overline{W} \setminus W$,
$[\dv_\Gamma(g)] =mp^*(C)-p^*(C')$,
and $p^*(mC)$ and $p^*(C')$ are rationally equivalent on $\Gamma$. 
Then, as $q$ is 
proper and $S$ is universally catenary, $q_*p^*(mC)$ and $q_*p^*C'$ are rationally equivalent in $Y$.
Since $h_*C =q_*p^*C$ and $h_*C'= q_*p^*C'$, we find that $h_*(mC)$ is rationally equivalent to $h_*(C') $ in $Y$. 
\qed 
\end{emp}

\noindent 
{\it Proof of Theorem {\rm \ref{mv-1-cycle-local}} when   {\rm (b)} holds.} 
It suffices to prove the theorem in the case where
the given $1$-cycle is the cycle associated with an integral closed subscheme of $X$ finite over $S$. 
We will denote again by $C $ this integral closed subscheme. 
By hypothesis, $X \to S$ is flat, so all its fibers are of the same dimension $d$.

Since $C$ is not empty,  $C_s$ is not empty, and thus has positive codimension in $X_s$ by hypothesis.
Therefore, we find that $d \geq 1$. Moreover, $C$ does not contain any  irreducible component of   positive 
dimension of $F_s$ and of $X_s$. If $d>1$, we fix a very ample invertible sheaf $\cO_X(1)$ on $X$ and apply Theorem \ref{exist-hyp} 
to $X \to S$, $C$, and $F$, to find that there exists  $n>0$ 
and a global section $f$ of $\cO_X(n)$ such that 
the closed subscheme $H_f$ of $X$ is a hypersurface 
that contains $C$ as a
closed subscheme, 
and such that for all $s \in S$, $H_f$ does not contain 
any irreducible component of positive dimension of $F_s$.
  Using Lemma \ref{components-dom} and the assumption that $C$ is integral, 
we can find 
an irreducible component $\Gamma$ of $H_f$ which contains $C$,  and such that all fibers of $\Gamma \to S$ have 
dimension $d-1$. If $d-1>1$, we repeat the process with $\Gamma$ endowed with the reduced induced structure, $C$, and $F \cap \Gamma$. 

It follows that we are reduced to proving the theorem when $X \to S$ 
has fibers of dimension $1$ and $X$ is integral. In this case, 
$F \to S$ is quasi-finite. We now reduce to the case where $X$ is normal and $X\to S$ has a 
smooth generic fiber. Let $K$ denote the function field of $S$. When $K$ has positive characteristic
$p>0$, consider the homeomorphism $\pi: X\to X^{(p^n)}$ with $n$ as in \ref{normalization-smooth},
so that the normalization of the reduced generic fiber of $X^{(p^n)}$ is smooth over $K$.
Applying \ref{Chow-radicial} to $\pi_*: {\mathcal A}(X) \to {\mathcal A}(X^{(p^n)})$, we find that it suffices to prove 
\ref{mv-1-cycle-local} 
for 
$X^{(p^n)}$, $\pi(C)$, and $\pi(F)$. So we can suppose that the normalization of 
the reduced generic fiber $X_{K}$ is smooth. 

Let $\pi: X'\to X$ be the normalization morphism. 
By \ref{normalization-finite} (a), this morphism is finite. 
Using \ref{Chow-radicial} applied to $\pi_*: {\mathcal A}(X') \to {\mathcal A}(X)$, 
we see that it is enough to prove \ref{mv-1-cycle-local}  for $X'$, $\pi^{-1}(C)$, and $\pi^{-1}(F)$. 
Replacing $X$ with $X'$ if necessary, we can now suppose that $X$ is normal, and that   
$X_K$ is smooth over $K$. 

We can now apply \ref{normalization-finite} (3) and we find that 
$X$ is $\mathbb Q$-factorial. So there exists 
an integer $n>0$ such that
the effective Weil divisor $nC$ is associated to a Cartier divisor on
$X$. We are thus reduced to the case where 
$C$ is a Cartier divisor on $X$, and the statement then follows from the end of the proof \ref{endofproof} of Case (a).
\qed

\smallskip
We show in our next theorem that 
in Rumely's Local-Global Principle as formulated in \cite{MB1}, 1.7, the 
hypothesis that the base scheme $S$ is excellent can be removed.
\begin{theorem} \label{RumelyLG} 
Let $S$ be a Dedekind scheme satisfying Condition {\rm (T)}. 
Let $X \to S$ be a separated surjective morphism of finite type. Assume that $X$ is irreducible and that the generic fiber of $X \to S$
is geometrically irreducible. Then $X \to S$ has a finite quasi-section.
\end{theorem}
\proof 
In \cite{MB1}, the hypothesis that $S$ is excellent is only used in 3.3 (which relies on 2.8) 
and, implicitly, in 2.5. The removal of the hypothesis that $S$ is excellent in 2.5 is addressed in  
\ref{lem.torsiondegreed} (2). To prove the 
Local-Global Principle, it is enough to prove it for integral 
quasi-projective schemes of relative dimension $1$ over $S$ 
(\cite{MB1}, 3.1). 

Assume that $S$ is not excellent. Consider a finite $S$-morphism $X \to X^{(p^n)}$ 
such that the normalization of the reduced generic fiber of $X^{(p^n)} \to S$ is smooth (\ref{normalization-smooth}).
Clearly, $X^{(p^n)} \to S$ has a finite quasi-section if and only if $X\to S$ has one. 
Similarly, since $(X^{(p^n)})_{\mathrm{red}} \to X^{(p^n)}$ is a finite $S$-morphism, 
$(X^{(p^n)})_{\mathrm{red}} \to S$ has a finite quasi-section if and only if $X^{(p^n)}\to S$ has one.
We also find from \ref{normalization-finite}  (a) that the normalization morphism  $X' \to (X^{(p^n)})_{\mathrm{red}}$ is finite, and again 
 $(X^{(p^n)})_{\mathrm{red}} \to S$ has a finite quasi-section if and only if $X' \to S$ has one.  
Thus we are reduced to the case where $X$ is normal and the generic fiber of $X \to S$ is smooth.
We now proceed as in the proof of \ref{normalization-finite} (3) to  remove 
the `excellent' hypothesis in \cite{MB1}, 2.8, and in \cite{MB1}, 3.3. 
\qed

\smallskip The following proposition is needed to produce the examples below which  conclude this section.
\begin{proposition} \label{threesections}
Let $S$ be a noetherian irreducible scheme. 
Let ${\mathcal L}$ be an invertible sheaf over $S$,  and consider the 
 scheme $X:= \mathbb P({\mathcal O}_S \oplus {\mathcal L})$,
with the associated projective\footnote{This morphism is projective by definition, see \cite{EGA}, II.5.5.2. It is then also proper (\cite{EGA}, II.5.5.3).} morphism $\pi: X \to S$. Denote by $C_0$ and $C_\infty$
the images of the two natural sections of $\pi$ obtained from the projections
${\mathcal O}_S \oplus {\mathcal L} \to {\mathcal O}_S$ and ${\mathcal O}_S \oplus {\mathcal L} \to {\mathcal L}$.
Suppose that there exists a finite flat quasi-section $g:Y\to S$ of $X \to S$ of degree $d$ which does not meet $F:= C_0 \cup C_\infty$.
Then ${\mathcal L}^{\otimes d}$ is trivial in $\Pic(S)$.
\end{proposition}
\proof  Let $X':= X \times_S Y$, with projection $\pi':X' \to Y$.
Clearly, $\pi'$ corresponds to the natural projection $\mathbb P({\mathcal O}_{Y} \oplus g^*{\mathcal L}) \to Y$.
We find that the morphism $\pi'$ has now three pairwise disjoint sections, corresponding to three 
homomorphisms from ${\mathcal O}_{Y} \oplus g^*{\mathcal L}$ to lines bundles, 
two of them being the obvious projection maps.

We claim that three such pairwise disjoint sections
can exist only if $\mathcal L':=g^*{\mathcal L}$ is the trivial invertible sheaf. 
Let $\mathcal N\subset \cO_{Y}\oplus {\mathcal L}'$ be the submodule 
corresponding to the third section (\cite{EGA}, II.4.2.4). 
For any $y\in Y$, $\mathcal N\otimes k(y)$ is different from 
$\mathcal L'\otimes k(y)$ (viewed as a submodule of $(\cO_{Y}\oplus {\mathcal L}')\otimes k(y)$) because in the fiber above $y$, 
the section defined by $\mathcal N$ is disjoint from the section defined by the projection to $\cO_{Y}$, 
so the image of 
$\mathcal N\otimes k(y)$ in the quotient $k(y)$ is non zero. Therefore 
the canonical map $\mathcal N\to \cO_{Y}\oplus {\mathcal L}' \to \cO_{Y}$ is surjective and, hence, it is an  isomorphism. 
Similarly, the canonical map $\mathcal N\to \mathcal L'$ is an isomorphism.
Therefore $\mathcal L'\simeq \cO_{Y}$. It is known (see, e.g., \cite{Gur}, 2.1) that since $Y \to S$ is finite and flat, the kernel of 
the induced map $\Pic(S) \to \Pic(Y)$ is killed by $d$.
\qed

\begin{example}\label{ex.extrahyp}
Let $R$ be  any Dedekind domain
and let $S= \Spec R$. 
Our next example shows that Theorem \ref{mv-1-cycle-local}
can hold  only if  $R$ has the property that $\Pic(R')$ is a torsion group for all
Dedekind domains $R'$ finite over $R$.

Indeed, choose an invertible sheaf ${\mathcal L}$ over $S$,  and consider the 
 scheme $X:= \mathbb P({\mathcal O}_S \oplus {\mathcal L})$,
with the associated smooth projective morphism $\pi: X \to S$. 
Let $C_0$ and $C_\infty$ be as in \ref{threesections}, and 
let $C:= C_0 + C_\infty$. Let $F:= \Supp(C)$. 

If Theorem \ref{mv-1-cycle-local} 
holds,
then a multiple of $C$ can be moved, and there exists a horizontal $1$-cycle $C'$ of $X$ such that $\Supp(C') \cap F = \emptyset$.
Hence, we find the existence of an integral subscheme $Y$ of $X$, finite and flat over $S$,
and disjoint from $F$.   Thus,  \ref{threesections} implies that ${\mathcal L}$ is a torsion element in $\Pic(S)$,
and for 
Theorem \ref{mv-1-cycle-local} 
to hold, it is necessary that $\Pic(S)$ be a torsion group.
Repeating the same argument starting with any invertible sheaf ${\mathcal L}'$ over any $S'$ (which is regular, 
and finite and flat over $S$)
and considering the map $\mathbb P({\mathcal O}_{S'} \oplus {\mathcal L}') \to S' \to S$, 
we find that  for 
Theorem \ref{mv-1-cycle-local}  
to hold, it is necessary that $\Pic(S')$ be a torsion group.
\end{example}

\begin{remark} An analogue of Theorem \ref{mv-1-cycle-local} cannot be expected to hold
when $S$ is assumed to be a smooth {\it proper} curve over a field $k$, even when $k$ is a finite field.
Indeed, suppose that $X \to S$ is given as in  \ref{mv-1-cycle-local} with both $X/k$ and $S/k$ smooth and proper.
Then any ample divisor $C$ on $X$ will have positive intersection number $(C \cdot D)_X$ with any curve $D$ on $X$.
Such an ample divisor then cannot be contained in a fiber of $X \to S$, and thus must be finite over $S$.
Set $F=\Supp(C)$. The conclusion of Theorem \ref{mv-1-cycle-local} cannot hold in this case: it is not possible to find on $X$ a divisor
rationally equivalent to $C$ which does not meet the closed set $F=\Supp(C)$. 
\end{remark}

\begin{example}\label{ex.extrahyp2}
Keep the notation introduced in Example \ref{ex.extrahyp}, and   
choose a non-trivial line bundle $\mathcal L$ of finite order $d>1$ in $\Pic(S)$.
Let $X:=\mathbb P(\cO_S\oplus \mathcal L)$. Let $F:= C_0 \cup C_\infty$. 
Theorem \ref{mv-1-cycle-local} (with the appropriate hypotheses on $S$) implies that a multiple $mC_0$ of $C_0$ can be moved away from $F$.
We claim that $C_0$ itself cannot  be moved away from $F$. Indeed, otherwise, 
there exist finitely many finite quasi-sections $Y_i \to S$ in $X \setminus F$ such that the greatest common divisor
of the degrees $d_i$ of $Y_i \to S$ is $1$ (because $C_0\to S$ has degree $1$). 
Hence, as in  \ref{ex.extrahyp}, we find that the order of $\mathcal L$
divides $d_i$ for all $i$. Since  $\mathcal L$ has order $d>1$ by construction, 
we have obtained a contradiction. 
In fact, we find that $dC_0$ is the smallest positive multiple of $C_0$ that could possibly be moved away from $F$.
\end{example}

\end{section}

\begin{section}{Finite morphisms to 
{${\mathbb P}^d_S$}.}
\label{finite-pd} 

Let $X \to S$ be an affine morphism of finite type, with $S=\Spec R$. 
Assume that $S$ is irreducible 
with generic point $\eta$, and let $d:= \dim X_\eta$. When $R=k$ is a field, 
the Normalization Theorem of E. Noether states that there exists a finite $k$-morphism $X \to {\mathbb A}_k^d$.
When $R$ is not a field, no {\it finite} $S$-morphism  $X \to {\mathbb
  A}_S^d$ may exists in general, 
even when $X\to S$ is surjective and $S$ is noetherian. 

When $R=k$ is a field, a stronger form of the Normalization Theorem
that applies to graded rings  (see, e.g., \cite{Eis}, 13.3) implies that every projective variety $X/k$ of dimension $d$
admits a finite $k$-morphism  $X \to {\mathbb P}_k^d$. Our main theorem in this section, Theorem \ref{theorem.finiteP^n} below, 
guarantees the existence of a finite $S$-morphism $X \to {\mathbb P}_S^d$ 
when $X\to S$ is projective with $R$ $\pictorsion$ (\ref{ConditionT*}), and $d:= \max\{\dim X_s, s\in S\}$. 
A converse to this statement is given in \ref{conversepictorsion}. We end this section 
with some remarks and examples of pictorsion rings.

\begin{theorem} \label{theorem.finiteP^n} \label{morphism-to-Pd}
Let $R$ be a $\pictorsion$ ring, and let $S:=\Spec R$.
Let $X \to S$ be a projective morphism, 
and set
$d:= \max\{\dim X_s, s\in S\}$.
Then  there exists a finite $S$-morphism
$r: X \to {\mathbb P}^d_S$.
If we assume in addition that  $\dim X_s= d$  for all $s\in S$,
then $r$ is surjective.
\end{theorem}

\proof 
Identify $X$ with a closed subscheme of a projective space $P:=\mathbb P^N_S$.
Assume first that $X \to S$ is of finite presentation.
We first apply Theorem \ref{bertini-type-0} to the projective
scheme $P\to S$ with $\cO_P(1)$, $C=\emptyset$, $m=1$, and $F_1=X$, to find $n_0>0$
and  $f_0\in H^0(P, \cO_P(n_0))$ such that $X\cap H_{f_0}\to S$
has all its fibers of dimension $\le d-1$
(use \ref{hypersurfaces-properties}  (1)).
We apply again Theorem \ref{bertini-type-0}, this time to $P\to S$ and $\cO_P(n_0)$,
$C=\emptyset$, $m=1$,  and $F_1=X\cap H_{f_0}$. We find an integer $n_1$ and a section
$f_1\in H^0(P, \cO_P(n_0n_1))$ such that
$(X\cap H_{f_0})\cap H_{f_1}$ has fibers over $S$ of maximal
dimension $d-2$.
We continue this process $d-2$ additional times, to find a sequence of
homogeneous polynomials $f_0,\dots, f_{d-1}$ such that the closed
subscheme $Y:=X\cap H_{f_{d-1}}\cap \ldots \cap H_{f_0}$ has all
its fibers of dimension
at most $0$ and, hence, is finite over $S$ since it is projective
(see \cite{EGA}, IV.8.11.1).
Note that replacing $f_i$ by a positive power of $f_i$ does not change
the topological properties of the closed set $H_{f_i}$.
So we can suppose $f_1, \dots, f_{d-1}\in H^0(P, \cO_P(n))$ for
some $n>0$.

Since $S$ is $\pictorsion$, $\Pic(Y)$ is a torsion group.
So there exists $j\ge 1$ such that $\cO_P(nj)|_Y\simeq \cO_Y$.
Let $e\in H^0(Y, \cO_P(nj)|_Y)$ be a basis.
As $Y$ is finitely presented over $S$, both $Y$ and  $e$ can be defined on  
some noetherian subring $R_0$ of $R$ (\cite{EGA}, IV.8.9.1(iii)).
By Serre's Vanishing Theorem on $\mathbb P^N_{R_0}$ applied with the very ample sheaf $\cO_{\mathbb P^N_{R_0}}(nj)$, we can find  $k>0$  
such that $e^{\otimes k}$ lifts to a section
$f_d \in H^0(P, \cO_P(njk))$. It follows that $H_{f_d}\cap Y=\emptyset$. 

We have constructed $d+1$ sections 
$f_0^{jk}, \dots, f_{d-1}^{jk}, f_d$ in $H^0(P, \cO_P(njk))$,
whose zero loci on $X$ have empty intersection.
The restrictions to $X$ of these sections define
a morphism $r : X\to \mathbb P^d_S$. 
Since $X\to S$ is of finite presentation and  $\mathbb P^d_S\to S$ is separated of finite type, the  morphism $r : X\to \mathbb P^d_S$
is also of finite presentation (\cite{EGA} IV.1.6.2 (v), or \cite{EGA1}, I.6.3.8 (v)).
By a standard argument (see e.g. \cite{Ked}, Lemma 3), 
the morphism $r : X\to \mathbb P^d_S$ is finite.
When $\dim X_s=d$, as
$X_s\to \mathbb P^d_{k(s)}$ is finite, it is also surjective.

Let us consider now the general case where $X \to S$ is not assumed to be of finite presentation. 
The scheme $X $, as a closed subscheme of $\mathbb P^N_S$, corresponds to a graded ideal $J \subset R[T_0,\dots, T_N]$.
Then $X$ is a filtered intersection of subschemes $X_\lambda \subset\mathbb P^N_S$ defined by finitely generated graded subideals of $J$.
Thus each  natural morphism $f_\lambda: X_\lambda \to S$ is of finite presentation. 
The points $x$ of $X_\lambda$ where the fiber dimension dim$_x(f_\lambda^{-1}(f_\lambda(x))$ is greater than $d$ form a closed subset $E_\lambda$
of $X_\lambda$ (\cite{EGA}, IV.13.1.3). 
Since the fibers of $Z \to S$ are of finite type over a field, the corresponding set for $X \to S$ is nothing but $\cap_\lambda E_\lambda$, and by hypothesis, the former set is empty.
Since $\cap_\lambda E_\lambda$ is a filtered intersection of closed subsets in the quasi-compact space $\mathbb P^N_S$, we find that $E_{\lambda_0} $ is empty for some $\lambda_0$. We can thus apply the statement of the 
theorem to the morphism of finite presentation $f_{\lambda_0}: X_{\lambda_0} \to S$ and find a finite $S$-morphism $X_{\lambda_0} \to \mathbb P^d_S$.
Composing with the closed immersion $X \to X_{\lambda_0}$ produces the desired finite morphism $X \to \mathbb P^d_S$.
\qed

\begin{remark} \label{rem.flat}
Assume that the morphism $r: X \to {\mathbb P}^d_S$ obtained in the above theorem is finite and surjective. 
When $S$ is a noetherian regular scheme  and  $X$ is
Cohen-Macaulay and irreducible, then $r$ is also flat. 
Indeed, $\mathbb P^d_S$ is regular  since $S$ is. Since $r$ is finite and surjective, $X $ is irreducible, and $ \mathbb P^d_S$ is universally catenary,
we find that $\dim \cO_{X,x}= \dim \cO_{\mathbb P^d_S,r(x)}$ for all $x \in X$ (\cite{Liubook}, 8.2.6). 
We can then use 
\cite{A-K}, V.3.5, 
to show that $r$ is flat. 
\end{remark}

Let us note here one class of   projective morphisms $X \to \Spec R$ which satisfy the conclusion of \ref{theorem.finiteP^n} 
without a pictorsion hypothesis on $R$.
\begin{proposition}\label{NoPictorsionHyp}  Let $R$ be a noetherian ring of dimension $1$ with $S:=\Spec R$ connected. 
Let $\mathcal E$ be any locally free $\cO_S$-module of rank $r \geq 2$.
Then there exists a finite $S$-morphism ${\mathbb P}(\mathcal E) \to {\mathbb P}^{r-1}_S$ of degree $r^{r-1}$.
\end{proposition} 

\proof Let $S$ be any scheme. Recall that given any locally free sheaf of rank $r$ of the form 
$\mathcal L_1 \oplus \dots \oplus \mathcal L_r$, with $\mathcal L_i$ invertible for $i=1,\dots, r$, there exists a finite $S$-morphism
$${\mathbb P}(\mathcal L_1 \oplus \dots \oplus \mathcal L_r)
\longrightarrow {\mathbb P}(\mathcal L_1^{\otimes d} \oplus \dots
\oplus \mathcal L_r^{\otimes d})$$
of degree $d^{r-1}$, defined on local trivializations by raising the
coordinates to the $d$-th tensor power.
 
Assume now that $S= \Spec R$ is connected, and recall that when $R$ is noetherian of dimension $1$, any  locally free $\cO_S$-module $\mathcal E$ of rank $r$ is 
isomorphic to a locally free $\cO_S$-module of the form $\cO_S^{r-1} \oplus \mathcal L$, where $\mathcal L$ is some invertible 
sheaf on $S$ (\cite{Ser}, Proposition 7)\footnote{When $k$ is an algebraically closed field and $R$ is a finitely generated regular $k$-algebra 
of dimension $2$, conditions on $R$ are given  in \cite{M-S}, Theorem 2, which ensure that such an isomorphism also exists for any locally free sheaf $\mathcal E$ on $\Spec R$.}, and $\wedge^r \mathcal E \cong \mathcal L$.
Consider the morphism ${\mathbb P}(\mathcal E) \to {\mathbb
  P}(\cO_S^{r-1} \oplus \mathcal L^{\otimes r})$ of degree $r^{r-1}$
described above. 
We claim that ${\mathbb P}(\cO_S^{r-1}\oplus \mathcal
L^{\otimes r})$ is isomorphic to $ \mathbb P^{r-1}_S$.
Indeed, 
we find  that $ \mathcal
L^{\oplus r}$ is isomorphic to $\mathcal \cO_S^{r-1}\oplus \mathcal L^{\otimes r}$ using the result quoted above, 
and ${\mathbb P}(\mathcal L^{\oplus r})$
is $S$-isomorphic to ${\mathbb P}(\cO_S^{r}) = \mathbb P^{r-1}_S$ (\cite{EGA}, II.4.1.4).
\qed
 
\smallskip
To prove a converse to Theorem \ref{theorem.finiteP^n} in \ref{conversepictorsion}, we will need the following proposition.

\begin{proposition} \label{ProjSpace} Let $S$ be a connected noetherian scheme. Let $\mathcal E$ be a locally free sheaf of rank $n+1$. 
Consider the natural projection morphism $\pi: {\mathbb P}(\mathcal E) \to S$.
\begin{enumerate}[\rm (a)]
\item 
Any invertible sheaf on ${\mathbb P}(\mathcal E)$ is isomorphic to a sheaf of the form $\cO_{{\mathbb P}(\mathcal E)}(m) \otimes \pi^*(\mathcal L)$, where $m \in {\mathbb Z}$  and $\mathcal L$
is an invertible sheaf on $S $.
\item Assume that $S=\Spec R$ is affine, and let  $f: {\mathbb P}^n_S \to {\mathbb P}^n_S$ be a finite morphism. 
Then $f^*(\cO_{\mathbb P^n_S}(1))$ is isomorphic to a sheaf of the form $\cO_{\mathbb P^n_S}(m) \otimes \pi^*(\mathcal L)$, where 
$\mathcal L$
is an invertible sheaf on $S $ of finite order and $m>0$.
\end{enumerate}
\end{proposition}
\proof (a) This statement is well-known (it is  found for instance in \cite{Mum}, page 20,
or in   
\cite{EGA}, II.4.2.7). Since we did not find a complete proof in the literature, let us sketch a proof here.
Let ${\mathcal M}$ be an invertible sheaf on  ${\mathbb P}(\mathcal E)$. For each $s \in S$, 
the pull-back  ${\mathcal M}_s $ of ${\mathcal M}$ to the fiber over $s$ is of the form $\cO_{{\mathbb P}_{k(s)}^n}(m_s)$ for some integer $m_s$. 
Using the fact that the Euler characteristic of ${\mathcal M}_s $  is locally constant on $S$ (\cite{MilAV}, 4.2 (b)), we find that 
$m_s$ is locally constant on $S$. Since $S$ is connected, $m_s=m$ for all $s\in S$. The conclusion follows as in \cite{MilAV}, 5.1.

(b) Let $\cO(1):= \cO_{{\mathbb P}_S^n}(1)$. 
Using (a), we find that $f^*(\cO(1))$ is isomorphic to a sheaf of the
form $\cO(m) \otimes \pi^*(\mathcal L)$, where  $\mathcal L$ is an
invertible sheaf on $S $. We have $m>0$ because over each point $s$,
$f^*(\cO(1))_s$ is ample, being the pull-back by a finite morphism of the ample sheaf $\cO(1)_s$,
and is isomorphic to $\cO(m)_s$. 

Write $M:= H^0(S,\mathcal L)$, and 
identify $H^0(\mathbb P^n_S, (\pi^*\mathcal L) (m))$ with $M \otimes_R R[x_0, x_1, ..., x_n]_m$, where $R[x_0, x_1, ..., x_n]_m$ denotes 
the set of homogeneous polynomials of degree $m$. The section of $\cO(1)$ corresponding to $x_i \in R[x_0, x_1, ..., x_n]_1$ pulls
back to a section of $f^*(\cO(1))$ which we identify with an element $F_i \in M \otimes_R R[x_0, x_1, ..., x_n]_m$.

Since $M$ is  locally free of rank $1$, there is a cover $\cup_{j=1}^t D(s_j) $ of $ S$ by special affine open subsets such that 
$M \otimes_R R[1/s_j]$ has a basis $t_j$.
Hence, for each $i\le n$, we can write $F_i = t_j \otimes G_{ij}$ with $G_{ij} \in R[1/s_j][x_0,\dots, x_n]_m$. 
Denote the resultant of $G_{0j}, \dots, G_{nj}$ by $\text{Res}(G_{0j}, \dots, G_{nj})$ (see \cite{Jou}, 2.3). 
We claim that $\text{Res}(G_{0j},\dots, G_{nj}) $ is a unit in $ R[1/s_j]$.
Indeed, over $D(s_j) := \Spec R[1/s_j]$, the restricted morphism $f_{S_j} : {\mathbb P}^n_{S_j} \to {\mathbb P}^n_{S_j}$
is given by the global sections of $\cO(m)_{\mid D(s_j)}$ corresponding to $G_{0j}, \dots, G_{jn} \in R[1/s_j][x_0,\dots, x_n]_m$. Since these global sections
generate the sheaf  $\cO(m)_{\mid D(s_j)}$, we find that the hypersurfaces $G_{0j}=0$, $\dots$, $G_{nj}=0$ cannot have a common point and, thus, that 
$\text{Res}(G_{0j},\dots, G_{nj}) \in R[1/s_j]^*$.

For $j=1,\dots, t$, consider now  
$$r_j:=  \text{Res}(G_{0j}, \dots, G_{nj})t_j^{\otimes (n+1)m^{n}} \in M^{\otimes (n+1)m^{n}} \otimes R[1/s_j].$$
Since $\text{Res}(G_{0j},\dots, G_{nj}) \in R[1/s_j]^*$, the element $r_j$ is a basis for $M^{\otimes (n+1)m^{n}} \otimes R[1/s_j]$.
We show now  that $M^{\otimes (n+1)m^{n}}$ is a free
$R$-module of rank $1$ by showing that the elements $r_j$ can be glued to produce 
a basis $r$ of $M^{\otimes (n+1)m^{n}}$ over $R$. Indeed, over $D(s_j) \cap D(s_k)$, we note that there exists $a \in R[1/s_j, 1/s_k]$
such that $at_j=t_k$. Then from $F_i= t_j \otimes G_{ij} = t_k \otimes G_{ik}$, we conclude that $G_{ij}=aG_{ik}$,
so that $$\text{Res}(G_{0j},\dots, G_{nj})= a^{(n+1)m^{n}}\text{Res}(G_{0k},\dots, G_{nk})$$ (\cite{Jou}, 5.11.2). We thus find that $r_j$ is equal to $r_k$ 
when restricted to $D(s_j) \cap D(s_k)$, as desired.
\qed

\begin{example} \label{example.infiniteorder}
Let $S$ be a connected noetherian affine scheme.  
Assume that $\Pic(S)$ contains an element $\mathcal L$ of infinite order. 
Suppose that $\mathcal L$ can be generated by $d+1$ sections for some
$d\ge 0$. We construct in this example a projective morphism $X_{\mathcal L} \to S$, with fibers of dimension $d$, and such that there exists no
finite $S$-morphism $X_{\mathcal L} \to {\mathbb P}^d_S$. Let $\cO(1):= \cO_{{\mathbb P}_S^d}(1)$. 

Using $d+1$ global sections in $\mathcal L(S)$ which generate $\mathcal L$, define a closed $S$-immersion $i_1: S \to {\mathbb P}^d_S$,
with $i^*_1({\mathcal O}(1)) = \mathcal L$. Consider also the closed
$S$-immersion $i_0: S \to {\mathbb P}^d_S$ given by $(1:0:\ldots:
0)\in \mathbb P^d_S(S)$,
so that $i^*_0({\mathcal O}(1) ) = \cO_S$. Consider now  the scheme $X_{\mathcal L}$ obtained by gluing
two copies of ${\mathbb P}^d_S$ over the closed subschemes $\text{Im}(i_0)$ and $\text{Im}(i_1)$ (\cite{Ana}, 1.1.1).
It is noted in \cite{Ana},  1.1.5, that under our hypotheses, the resulting gluing is endowed with a natural morphism $\pi: X_{\mathcal L} \to S$
which is separated and of finite type.
Recall that the scheme $X_{\mathcal L}$ is endowed with two natural closed immersions $\varphi_0: X_0= {\mathbb P}^d_S \to X_{\mathcal L}$
and $\varphi_1: X_1= {\mathbb P}^d_S \to X_{\mathcal L}$ such that
$\varphi_0 \circ i_0 = \varphi_1 \circ i_1$. 
Moreover, the $S$-morphism $(\varphi_0,\varphi_1): {\mathbb P}^d_S \sqcup {\mathbb P}^d_S \to X_{\mathcal L}$ is finite and surjective. 
 Since ${\mathbb P}^d_S \to S$ is proper, we find that $X_{\mathcal L} \to S$ is also proper.

Suppose that there exists a finite $S$-morphism $f: X_{\mathcal L} \to
{\mathbb P}^d_S$. Then, using \ref{ProjSpace} (b), we find that there
exist two torsion invertible sheaves $\G_0$ and $\G_1$ on $S$ and $m
\geq 0$ such that $(f \circ \varphi_0 \circ i_0)^*(\cO_{\mathbb
  P^d_S}(1))= \G_0$, and $(f \circ \varphi_1 \circ i_1)^*(\cO_{\mathbb
  P^d_S}(1))= \mathcal L^{\otimes m} \otimes \G_1$.
Since we must have then $\G_0$ isomorphic to $\mathcal L^{\otimes m} \otimes \G_1$ and since ${\mathcal L}$ is not torsion, 
we find that such a morphism $f$ cannot exist. 

To conclude this example, it remains to show that $\pi: X_{\mathcal L} \to S$ is a projective\footnote{For an example where the gluing of two projective spaces over a `common' closed subscheme is not projective, see \cite{Fer}, 6.3.} morphism.
For this, we exhibit an ample sheaf on $X_{\mathcal L}$ as follows. Consider the sheaf ${\mathcal F}_0:= (\pi \circ \varphi_0)^*({\mathcal L})(1)$ on  $X_0= {\mathbb P}^d_S$
and the sheaf ${\mathcal F}_1:=\cO(1)$ on $X_1= {\mathbb P}^d_S$. We clearly have a natural isomorphism of sheaves 
$i_0^*({\mathcal F}_0) \rightarrow i_1^*({\mathcal F}_1)$ on $S$. Thus, we can glue the sheaves  ${\mathcal F}_0$ and ${\mathcal F}_1$
to obtain a sheaf ${\mathcal F}$ on $X_{\mathcal L}$. Since both ${\mathcal F}_0$ and ${\mathcal F}_1$ are invertible, we find that 
${\mathcal F}$ is also invertible on $X_{\mathcal L}$ (such a statement in the affine case can be found in \cite{Fer}, 2.2). Under the finite $S$-morphism $(\varphi_0,\varphi_1): {\mathbb P}^d_S \sqcup {\mathbb P}^d_S \to X_{\mathcal L}$,
the sheaf ${\mathcal F}$ pulls back to the sheaf restricting to ${\mathcal F}_0$ on $X_0$ and ${\mathcal F}_1$ on $X_1$.
In particular, the pull-back is ample (since $\mathcal L$ is generated by its global sections), and since $X_{\mathcal L}\to S$ is proper, we can apply \cite{EGA}, III.2.6.2, to find that ${\mathcal F} $ is also ample.
 \end{example}
 \begin{remark} Let $R$ be a Dedekind domain and $S:= \Spec R$. Let $X \to S$ be a projective morphism with fibers of dimension $1$.
 When $R$ is pictorsion, Theorem \ref{theorem.finiteP^n} shows that there exists a finite $S$-morphism $X \to {\mathbb P}^1_S$.
It is natural to wonder, when $R$ is not assumed to be pictorsion, whether it would still be possible to find a locally free $\cO_S$-module  $\mathcal E$ 
of rank $2$
and a finite $S$-morphism $X \to {\mathbb P}(\mathcal E)$. The answer to this question is negative, as the following 
example shows. 

Assume that $\Pic(S)$ contains an element $\mathcal L$ of infinite order. 
Then $\mathcal L$ can be generated by $2$ sections.
Consider the projective morphism $X_{\mathcal L} \to S$ constructed
in  Example \ref{example.infiniteorder}. 
Suppose that there exists a locally free $\cO_S$-module  $\mathcal E$ 
of rank $2$ and a finite  $S$-morphism $X_{\mathcal L}  \to {\mathbb P}(\mathcal E)$. Proposition \ref{NoPictorsionHyp} shows  that there exists then 
a finite $S$-morphism ${\mathbb P}(\mathcal E) \to {\mathbb P}^1_S$. We would then obtain by composition a finite 
$S$-morphism $X_{\mathcal L} \to {\mathbb P}^1_S$, which is a contradiction. We thank Pascal Autissier for bringing this question to our attention.
\end{remark}

We are now ready to prove a converse to Theorem \ref{theorem.finiteP^n}.

\begin{proposition} \label{conversepictorsion}
Let $R$ be any commutative ring 
and let $S:=\Spec R$.
Suppose that for any $d \geq 0$, and for any projective morphism $X \to S$ such that $\dim X_s=d$  for all $s\in S$,
 there exists a finite surjective $S$-morphism $ X \to {\mathbb P}^d_S$.
Then $R$ is pictorsion.

When $R$ is noetherian of  finite Krull dimension $\dim R$, then $R$ is pictorsion if
 for all projective morphisms $X\to S$ such that
$\dim X_s \le \dim R$ for all $s\in S$, there exists
a finite $S$-morphism $X \to {\mathbb P}^{\dim R}_S$.
\end{proposition}

\proof  Let $R'$ be a finite extension of $R$, and let $\mathcal L' \in \Pic(\Spec R')$. 
The sheaf $\mathcal L'$ descends to an element ${\mathcal L}$ of $\Pic(\Spec R_0)$ for some noetherian subring $R_0$ of $R'$.
For each  connected component $S_i$ of $\Spec R_0$, let $d_i $
be such that ${\mathcal L}_{|S_i}$ can be generated by $d_i+1$ global sections. Let $d:= {\rm max}(d_i)$.  
When $R$ is noetherian we take $R=R_0$, and when in addition  $\dim(R) < \infty$, we can always choose $d \leq \dim(R)$.

Assume now that $\mathcal L'$ is of infinite order. It follows that ${\mathcal L}$ is of infinite order 
on some connected component of $\Spec R_0$. Apply the construction of Example \ref{example.infiniteorder} to each connected component $S_i$ of $\Spec R_0$
where ${\mathcal L}$ has infinite order, with a choice of $d$ global sections of ${\mathcal L}_{|S_i}$ which generate ${\mathcal L}_{|S_i}$. 
We obtain a projective scheme $X_i \to S_i$ 
with fibers of dimension $d$ and which does not admit a finite morphism to $ X_i \to {\mathbb P}^{d}_{S_i}$. 
If $S_j$ is a connected component of $\Spec R_0$
such that ${\mathcal L}$ has  finite order, we set $X_j\to S_j$ to be ${\mathbb P}^{d}_{S_j} \to S_j$. We let $X$ denote the disjoint union of the schemes $X_i$. 
The natural morphism $X\to \Spec R_0$ has fibers of dimension $d$, and does not  admit a finite morphism  $ X \to {\mathbb P}^{d}_{R_0}$. 

Let $R_1$ be any noetherian ring such that $R_0 \subset R_1 \subset R'$.
By construction, the pullback of $\mathcal L $ to $\Spec R_1$ has infinite order on some connected component of $\Spec R_1$. 
Since the construction in Example \ref{example.infiniteorder} is compatible with pullbacks, we conclude that $
X \times_{R_0} R_1$ does not admit a finite morphism to ${\mathbb P}^{d}_{R_1}$. It follows then 
from \cite{EGA} IV.8.8.2 and IV.8.10.5 that there is no finite morphism $X \times_{R}  R' \to {\mathbb P}^{d}_{R'}$.
Hence, there is no finite $R$-morphism $X \times_{R}  R' \to {\mathbb P}^{d}_{R}$. Replacing $X \times_{R}  R'$ with its disjoint
union with ${\mathbb P}^{d}_{R}$ if necessary, we  obtain a projective morphism to $S$ with fibers of dimension $d$
 and which does not factor through a finite morphism to  $ {\mathbb P}^d_S$. \qed
\medskip

 We present below an example of an affine regular scheme $S$ of dimension $3$ with a locally free sheaf $\mathcal E$ of rank $2 $
of the form ${\mathcal E}= \cO_S \oplus {\mathcal L}$ such that ${\mathbb P}({\mathcal E})$ does not admit a finite $S$-morphism to ${\mathbb P}^1_S$.

\begin{example} 
Let $V$ be any smooth connected quasi-projective variety over a field $k$. Let $\mathcal E$  be a locally free sheaf of rank $r $ on $V$.
Let $p : {\mathbb P}(\mathcal E) \to V$ denote the associated projective bundle. 
Denote by $A(V)$ the Chow ring   of algebraic cycles on $V$ modulo rational
equivalence. Let $\xi$ denote the class in $A({\mathbb P}(\mathcal E))$
of the invertible sheaf $\cO_{\mathbb P(\mathcal E)}(1)$. Then $p$ induces a ring homomorphism $p^*: A(V) \to A({\mathbb P}(\mathcal E))$, 
and $A({\mathbb P}(\mathcal E))$ is a free $A(V)$-module
generated by $1, \xi, \dots, \xi^{r-1}$. For $i=0,1,\dots, r$,  one defines (see, e.g., \cite{Har}, page 429) the $i$-th Chern class 
 of ${\mathcal E}$, $c_i(\mathcal E) \in A^i(V)$,
and these classes satisfy the requirements that $c_0(\mathcal E)=1$ and 
$$\sum_{i=0}^r (-1)^i p^*(c_i(\mathcal E)) \xi^{r-i}=0$$
in $A^r({\mathbb P}(\mathcal E))$. When $\mathcal E = \cO_V \oplus {\mathcal L}$ for some invertible sheaf ${\mathcal L}$, we find that $c_2(\mathcal E)=0$.

Consider now the case where $\mathcal E$ has rank $2$ and 
suppose that there exists a finite $V$-morphism $f: {\mathbb P}(\mathcal E) \to {\mathbb P}^1_V$. Then 
 $f^*(\cO_{{\mathbb P}^1_V}(1))$ is isomorphic to a locally free sheaf of the form $p^*\mathcal M \otimes \cO_{{\mathbb P}(\mathcal E)}(m)$
for some $m >0$ and some invertible sheaf $\mathcal M$ on $V$ 
(\ref{ProjSpace} (a)). 
Consider the ring homomorphism
$$f^*: A({\mathbb P}^1_V) =A(V)[h]/(h^2) \longrightarrow A({\mathbb P}(\mathcal E)) = A(V)[\xi]/(\xi^2 - c_1(\mathcal E) \xi + c_2(\mathcal E)) ,$$
where $h$ denote the class in $A({\mathbb P}^1_V)$
of the invertible sheaf $\cO_{\mathbb P^1_V}(1)$.
It follows that in $A({\mathbb P}(\mathcal E))$, $f^*(h)= a + m \xi$ with $a \in A^1(V)$, and   $(a + m \xi)^2=0$.
Hence, $m^4(c_1(\mathcal E)^2 -4 c_2(\mathcal E)) = 0$ in $A^2(V)$.
Thus,  in $A^2(V)_{\mathbb Q}$, $c_1(\mathcal E)^2 =4 c_2(\mathcal E)$. 
Choose now $\mathcal E = \cO_V \oplus {\mathcal L}$ for some invertible sheaf ${\mathcal L}$. 
Then $0=c_2(\mathcal E)=c_1(\mathcal E)^2 = c_1(\mathcal L)^2$ in $A^2(V)_{\mathbb Q}$. 

We are now ready to construct our example. Recall that under our hypotheses on $V$, there exists an affine variety $S$ and 
a surjective morphism $\pi: S \to V$ such that $\pi $ is a torsor under a vector bundle (Jouanolou's device, \cite{JouDevice}, 1.5).
We will use only the simplest case of this construction, when $V= {\mathbb P}^2_k$. 
In this case, $S$ is the affine variety formed by all $3 \times 3$-matrices which are idempotent and have rank $1$. 
We claim that we have an isomorphism
 $$\pi^*: A(V)_{\mathbb Q} \to A(S)_{\mathbb Q}.$$
Indeed, this statement with the Chow rings replaced by $K$-groups is proved in \cite{Jou2}, 1.1. 
Then we use the fact that the Chern character determines an isomorphism of ${\mathbb Q}$-algebras
$ch : K^0(X)_{\mathbb Q} \longrightarrow A(X)_{\mathbb Q} $
where $K^0(X)$ denotes the Grothendieck group of algebraic vector bundles on a smooth  quasi-projective variety
$X$ over a field (\cite{Ful}, 15.2.16 (b)).  

 Choose on $S$ the line bundle $\mathcal L:= \pi^*\cO_{\mathbb P^2}(1)$. Then since $A(V)_{\mathbb Q}= {\mathbb Q}[h]/(h^3)$, we find that
 $h^2 \neq 0$, so that $c_1(\mathcal L)^2 \neq 0$ in  $A(S)_{\mathbb Q}$. Hence, we have produced a smooth affine variety $S$ of dimension $3$, 
 and a locally free sheaf $\mathcal E:=\cO_S \oplus \mathcal L$ such that $\mathbb P(\mathcal E)$ does not admit a finite surjective 
$S$-morphism to $\mathbb P^1_S$.
\end{example}
 
We conclude this section with some remarks and examples of pictorsion rings (\ref{ConditionT*}).
We first note the following. 

\begin{lemma} Let $R$ be any commutative ring.
Denote by $R^{\mathrm{red}}$ the quotient of $R$ by its nilradical.
Then $R$ is pictorsion if and only if $R^{\mathrm{red}}$ is pictorsion.
\end{lemma}

\proof  
Since $R\to R^{\mathrm{red}}$ is a finite homomorphism, it is clear that if
$R$ is pictorsion, then so is $R^{\mathrm{red}}$. Assume now that $R^{\mathrm{red}}$ is pictorsion  
and let $R\to R'$ be a finite homomorphism. Then 
$R^{\mathrm{red}}\to (R')^{\mathrm{red}}$ is a finite homomorphism. Thus $\Pic((R')^{\mathrm{red}})$ is a
torsion group. As we can see using Nakayama's lemma, $\Pic(R')\to
\Pic((R')^{\mathrm{red}})$ is injective,
so $\Pic(R') $ is a torsion group. \qed

\medskip
Recall that a pictorsion Dedekind domain $R$ satisfies Condition {\rm (T) (a)} in \ref{ConditionT}  by definition
(that is, $\Pic(R_L)$ is a torsion group for any Dedekind domain $R_L$ obtained as the integral closure of $R$ in a finite extension $L/K$). 
The statement of (2) below is found in \cite{MB1}, 2.3, when $R$ is excellent. We follow the proof given in  \cite{MB1},
modifying it only in 2.5 to also treat the case where $R$ is not excellent. We do not know of an example of 
a Dedekind domain which satisfies Condition {\rm (T) (a)} and which is not pictorsion.

\begin{lemma} \label{lem.torsiondegreed} 
Let $R$ be a Dedekind domain with field of fractions $K$. 
\begin{enumerate} [{\rm (1)}] 
\item Let $L/K$ be a finite extension of degree $d$, 
and let $R'$ denote a sub-$R$-algebra of $L$, integral over $R$. 
Then the kernel of $\Pic(R) \to \Pic(R')$ is killed by $d$.
\item  If $R$ satisfies Condition {\rm (T)} in {\rm \ref{ConditionT}}, 
then $R$ is $\pictorsion$. 
\end{enumerate}
\end{lemma}
\proof (1) When $R'$ is finite and flat over $R$, this is well-known
(see, e.g., \cite{Gur}, 2.1). (The hypothesis that $R$ is Dedekind is used here to insure that the ring $R'$ is flat over $R$.)
In general, let $M$ be a locally free $R$-module of rank $1$ such that $M \otimes_R R'$ is isomorphic
as $R'$-module to $R'$. 
Then there exist a finite $R$-algebra $A$ contained in $R'$ such that $M \otimes_R A$ is isomorphic
as $A$-module to $A$. 
It follows that $M^{d}$ is trivial in $\Pic(R)$, since $A/R$ is finite.

(2) Let $S=\Spec R$. Let $Z$ be a finite $S$-scheme. We need to show that
$\Pic(Z)$ is torsion. The proof in \cite{MB1}, 2.3 - 2.6, is complete when $R$ is excellent.
When $R$ is not necessarily excellent, only 2.5 needs to be modified as follows. Assume that $Z$ is reduced. Let $Z'\to Z$ be the normalization 
morphism, which need not be finite. Then $Z'$ is a finite disjoint union of Dedekind schemes, and the hypothesis that $R$ satisfies
Condition (T)(a) implies that $\Pic(Z')$ 
is a torsion group. Let $\mathcal L\in\Pic(Z)$. Then there exists 
$n\ge 1$  
such that $\mathcal L^{\otimes n}\otimes \cO_{Z'}\simeq \cO_{Z'}$. This 
isomorphism descends to some 
$Z$-scheme $Z_{\alpha}$ with $\pi: Z_{\alpha}\to Z$ finite and birational. 
We now use the proof of 2.5 in \cite{MB1}, applied to  $Z_1=Z_{\alpha}$ (instead of the normalization
which is not necessarily finite), to find that the kernel of $\Pic(Z)\to \Pic(Z_{\alpha})$ is  torsion. Hence, $\mathcal L$ is torsion. 
\qed 

\begin{proposition}  \label{Bezout} Let $R$ be 
a Dedekind
domain with field of fractions $K$. Let $\overline{R}$ denote the integral closure of $R$ in an algebraic closure $\overline{K}$ of $K$.
The following are equivalent:
\begin{enumerate}[\rm (1)]
\item  Condition {\rm (T)(a)} in {\rm \ref{ConditionT}} holds.
\item  $\overline{R}$ is a B\'ezout domain (i.e., all finitely generated ideals of $\overline{R}$ are principal).
\end{enumerate}
\end{proposition}
\proof
That (1) implies (2) is the content of Theorem 102 in \cite{Ka}.
Assume that (2) holds, and let 
 $R_L$ be the integral closure of $R$ in a finite extension $L/K$. Let $I$ be a non-zero ideal in $R_L$. 
 Then $I\overline{R}$ is principal. Hence, there exists a finite extension $F/L$ 
such that in the integral 
 closure $R_F$ of $R$ in $F$, $IR_F$ is principal. Since the kernel of $\Pic(R_L) \to \Pic(R_F)$ is killed by $[F:L]$ (\ref{lem.torsiondegreed}),
 we find that $I$ has finite order in $\Pic(R_L)$. \hfill \qed

\begin{remark}
Keep the notation of \ref{Bezout}, and denote by
$R_F$ the integral closure of $R$ in any algebraic extension $F/K$. Then Condition (1) in \ref{Bezout}  implies that  $\Pic(R_F)$ is a torsion group. Indeed, 
one finds that $\Pic(R_F) = \varinjlim \Pic(R_L)$, with the direct limit taken over all finite extensions $L/K$ contained in $F$.

Condition (2) in \ref{Bezout} is equivalent to $\Pic(\overline{R})= (0)$. Indeed, the ring $\overline{R}$ is a Pr\"ufer domain (\cite{Ka}, Thm.\ 101), 
and 
a Pr\"ufer domain $D$ is a B\'ezout domain if and only if $\Pic(D)= (0)$.
\end{remark}

We now recall two properties of commutative rings and relate them to the notion of pictorsion 
introduced in this article. A {\it local-global} ring $R$ is a commutative ring where the following property holds:
 whenever $f \in R[x_1,\dots, x_n]$ is such  that the ideal of values $(f(r), r \in R^n)$ is equal to the full ring $R$, 
 then there exists $r \in R^n$ such that $f(r) \in R^*$. A commutative ring $R$ {\it satisfies the primitive criterion}
 if, whenever $f(x) = a_nx^n+\dots + a_0$ is such that $(a_n,\dots, a_0)=R$ (such $f$ is called {\it primitive}), then  there exists $r \in R$ such that $f(r) \in R^*$.
 A ring $R$ satisfies the primitive criterion if and only if it is local-global and for each maximal ideal 
$M$ of $R$, the  residue field $R/M$ is infinite
(\cite{McW}, Proposition, or \cite{EG}, 3.5). 
 
 \begin{proposition} \label{localglobal} Let $R$ be a local-global commutative ring.
 Then every finite $R$-algebra $R'$ has $\Pic(R')=(1)$. In particular, $R$ is pictorsion.
 \end{proposition}
 \proof The ring $R'$ is also a local-global ring (\cite{EG}, 2.3). In a local-global ring, every finitely generated 
projective $R$-module of constant rank is free (\cite{McW}, Theorem, or \cite{EG}, 2.10). 
It follows that $\Pic(R') = (1)$.
\qed

\begin{example} \label{pictrivial}
Rings which satisfy the primitive criterion 
can be constructed as follows (see, e.g., \cite{vK}, 1.13, and also \cite{EG}, section 5). Let $R$ be any commutative ring, and consider 
the multiplicative subset $S$ of $R[x]$ consisting of all primitive polynomials. Then the ring $R(x):= S^{-1}R[x]$
satisfies the primitive criterion\footnote{The ring $R(x)$ is considered already in \cite{Kru}, page 535 after Hilfssatz 1. 
The notation $R(x)$ was introduced  by Nagata (see the historical remark in \cite{Nag}, p.\ 213).
When $R$ is a local ring, the extension $R\to R(x)$
 is used to reduce some considerations to the case of local rings with infinite residue fields (see, e.g., \cite{SH}, 8.4, p.\ 159).
Let $X$ be any scheme with an ample invertible sheaf. An affine scheme $X'$ with a faithfully flat morphism $X' \to X$ is constructed
in \cite{Fer2}, 4.3, in analogy with the purely affine situation $\Spec R(x) \to \Spec R$.}. Indeed, suppose that $g(y) \in R(x)[y]$ is primitive. 
Then write $g(y) = \sum_{i=0}^n f_i(x) y^i$, with $f_i(x) \in R(x)$. It is easy to reduce to the case where
$f_i(x) \in R[x]$ for all $i$. Since $g(y) $ is primitive, we find that the ideal generated by the coefficients of the polynomials $f_0(x),\dots,f_n(x)$
is the unit ideal of $R$. Hence, choosing $y:=x^t$ for $t$ large enough, we find that $
g(x^t)$ is a primitive polynomial in $R[x]$ and thus is a unit in $R(x)$.
\end{example}

\begin{example} We have seen already in this article examples of commutative rings $R$ such that 
for every finite morphism $\Spec R' \to \Spec R$,
$\Pic(R')$ is trivial (\ref{pictrivial}), or $\Pic(R')$  is finite but not necessarily trivial (take $R={\mathbb Z}$ or ${\mathbb F}_p[x]$).  When needed, such rings could be called {\it pictrivial} and {\it picfinite}, respectively. 

Let us note in this example a ring $R$
which is pictorsion and such that at least one of the groups $\Pic(R')$ is not finite. Consider the algebraic closure $\overline{\mathbb F}_p$
of $\mathbb F_p$, and let $R:= \overline{\mathbb F}_p[x]$. Then $R$ is pictorsion because it satisfies Condition (T) (\ref{lem.torsiondegreed} (2)).
Indeed, let $R'$ be the integral closure of $R$ in a finite extension   of $\overline{\mathbb F}_p(x)$.
Then $U:=\Spec(R')$ is a dense open subset of a smooth connected projective curve $X/ \overline{\mathbb F}_p$. One shows that the natural restriction map
$\Pic(X) \to \Pic(U)$ induces a surjective map $\Pic^0(X) \to \Pic(U)$ with finite kernel. When the genus of $X$ is bigger than $0$,
it is known that $\Pic^0(X)$, which is isomorphic to the $\overline{\mathbb F}_p$-points of the Jacobian of $X$, is an infinite torsion group.

We also note that the ring $R:= \overline{\mathbb F}_p[x,y]$ is not pictorsion. Indeed, let $X/\overline{\mathbb F}_p$ be a smooth projective surface
over $\overline{\mathbb F}_p$ such that its N\'eron-Severi group NS$(X)$ has  
rank   greater than one (for instance, $X$ could be the product of two smooth projective curves). Let $D \subset X$ be an irreducible divisor whose complement $V:= X \setminus D$ is affine.
Write $V=\Spec A$, and use Noether's Normalization Lemma to view $A$ as a finite $\overline{\mathbb F}_p[x,y]$-algebra.
We claim that $\Pic(V)$ is not a torsion group. Indeed, the natural restriction map $\Pic(X) \to \Pic(V)$ is surjective, with kernel generated 
by the class of $D$. 
If $\Pic(V)$ is torsion, then the quotient of
$\mathrm{NS}(X)$ by the subgroup generated by image of $D$ is
torsion. This contradicts the hypothesis on the rank of
$\mathrm{NS}(X)$. 
\end{example}

\begin{example} \label{Varley}
 Robert Varley suggested the following example of a 
Dedekind domain $A$ which 
is $\pictorsion$   
with infinitely many maximal ideals, each  having 
residue field which is not an algebraic extension of a finite field, 
i.e., such that $A$ does not satisfies Condition (T)(b) in \ref{ConditionT}. Rather than providing a direct proof that 
the ring below is pictorsion, we  interpret the example in light of the above definitions:

{\it Let $Z$ denote a countable subset of ${\mathbb C}$.
 Consider the polynomial ring ${\mathbb C}[x]$, and let $T$ denote the multiplicative subset of all polynomials which do not vanish 
on $Z$. Then $A:= T^{-1}({\mathbb C}[x])$ satisfies the primitive criterion, and is thus pictorsion by \text{\rm \ref{localglobal}}. }

Indeed, let $F(y) \in A[y]$ be a primitive polynomial. Up to multiplication by elements of $T$, we
can assume that $F(y)=f_n(x) y^n + \dots + f_0(x)$ with $f_i \in {\mathbb C}[x]$ for all $i$,  and that $x-z$ does not divide $\gcd(f_i(x), i=0,\dots, n)$  for all $z \in Z$. We claim that there exists $a \in {\mathbb C}$ such that $F(a) \in {\mathbb C}[x]$ is coprime to $x-z$ for all $z \in Z$. This shows that $F(a)$ is invertible in $A:= T^{-1}({\mathbb C}[x])$, and thus $A$ satisfies the primitive criterion.
To prove this claim, let us think of $F(y)$ as a polynomial $F(x,y)$ in two variables, and let us first note that the curve $F(x,y)=0$ intersects the line $x-z=0$ in at most $\deg(F)$ places.
Thus, there are only countably many points in the plane ${\mathbb C}^2$ of the form $(z,v)$ with $z \in Z$ and $F(z,v)=0$. 
Therefore, it is possible to choose $a \in {\mathbb C}$ such that $F(x,a)=0$ does not contain any of these countably many points.
\end{example}
 
Let $\overline{\mathbb Q}$ denote the algebraic closure of ${\mathbb Q}$. In view of the above example,  it is natural to wonder
whether there exists a multiplicative subset $T$ of $\overline{\mathbb Q}[x]$
such that $R:=T^{-1}(\overline{\mathbb Q}[x])$ 
is pictorsion and $\Spec R$ is infinite. Clearly, the integral closure $\tilde{R}$ of $R$ in the algebraic closure of 
$\overline{\mathbb Q}(x)$ must be B\'ezout (\ref{Bezout}). 
A related question is addressed  in \cite{vDM}, section 5, and in \cite{GM}. 
 
\end{section}

\end{document}